\documentclass[12pt]{amsart}
\usepackage[top=0.9in, bottom=0.8in, left=1.2in, right=1.2in]{geometry}
\usepackage{graphicx}
\usepackage{amssymb, amsmath, sidecap}
\usepackage{amsfonts}
\usepackage{float}
\usepackage{mathrsfs}
\usepackage{booktabs}
\usepackage{verbatim}
\usepackage[dvipsnames]{xcolor}
\usepackage{tikz}
\usetikzlibrary{shapes.geometric}
\usepackage{esint}
\usepackage{setspace}
\usepackage{hyperref}
\hypersetup{
  colorlinks=true,
  linkcolor=MidnightBlue,
  citecolor=MidnightBlue,
  urlcolor=MidnightBlue,
  linktoc=all,
  bookmarksopen=true,
  bookmarksnumbered=true,
  pdfdisplaydoctitle=true,
  pdftitle={Generalized Harmonic Measures: Synchronized Approximation, Representation, and Quantitative Measure Recovery},
  pdfauthor={Duchao Liu and Yunjie Wang}
}

\renewcommand{\baselinestretch}{1}

\newcommand{\R}{\mathbb{R}}

\newcommand{\dist}{\text{\rm dist}}

\def\bt{\begin{thm}}
\def\et{\end{thm}}
\def\bl{\begin{lem}}
\def\el{\end{lem}}
\def\bd{\begin{defi}}
\def\ed{\end{defi}}
\def\bc{\begin{cor}}
\def\ec{\end{cor}}
\def\bp{\begin{proof}}
\def\ep{\end{proof}}
\def\br{\begin{rem}}
\def\er{\end{rem}}

\def\d{\, \mathrm{d}}

\def\be{\begin{equation}}
\def\ee{\end{equation}}
\def\bes{\begin{equation*}}
\def\ees{\end{equation*}}
\def\bea{\begin{equation} \begin{aligned}}
\def\eea{\end{aligned} \end{equation}}
\def\beas{\begin{equation*} \begin{aligned}}
\def\eeas{\end{aligned} \end{equation*}}
\def\ba{\begin{align}}
\def\ea{\end{align}}
\def\bas{\begin{align*}}
\def\eas{\end{align*}}

\newtheorem{thm}{Theorem}[section]
\newtheorem{lem}[thm]{Lemma}
\newtheorem{defi}{Definition}[section]
\newtheorem{ex}{Example}[section]
\newtheorem{prop}[thm]{Proposition}

\newtheorem{rem}{Remark}[section]
\newtheorem{cor}[thm]{Corollary}

\numberwithin{equation}{section}
\numberwithin{figure}{section}

\begin{document}


\title[Generalized Harmonic Measures]{Generalized Harmonic Measures:\\Synchronized Approximation, Representation,\\and Quantitative Measure Recovery\footnotemark[1]}
\author[Liu]{Duchao Liu}
\address[Duchao Liu]{School of Mathematics and Statistics, Lanzhou
University, Lanzhou 730000, P. R. China\\
Tel.: +8613893289235, fax: +8609318912481}
\email{liudch@lzu.edu.cn}
\email{liuduchao@gmail.com}

\author[Wang]{Yunjie Wang}
\address[Yunjie Wang]{School of Mathematics and Statistics, Lanzhou
University, Lanzhou 730000, P. R. China} \email{220220934191@lzu.edu.cn}

\footnotetext[1]{
Research supported by the National Natural Science Foundation of
China (NSFC 12371096).}


\makeatletter
\@namedef{subjclassname@2020}{\textup{2020} Mathematics Subject Classification}
\makeatother

\begin{abstract}
We study synchronized boundary and kernel approximation for harmonic
measure systems satisfying nested mean-value identities. Under boundary
concentration along a common subsequence, joint limits yield representation
on a common refinement, and gluing gives uniqueness on the minimal boundary.
Neighborhood nonvanishing gives full-sequence synchronized selection.
For a compact set whose removal from a bounded $C^{1,1}$ domain in dimension
at least three leaves a connected domain, zero Newtonian capacity is
equivalent to concentration of one normalized singular kernel, constructed
from harmonic measures, on a subsequence of one relatively compact exhaustion.
For such zero-capacity sets, capacitary estimates give quantitative
recovery of finite signed source measures on every prescribed relatively
compact exhaustion. On
rectifiable curve networks in three dimensions, the logarithmic error
bound is sharp for this recovery formula in the class of finite measures.
Boundary splitting and finite weighted metric graphs illustrate the gluing
conclusions.
\end{abstract}

\maketitle

\begingroup
\renewcommand{\baselinestretch}{0.83}\normalsize
\tableofcontents
\endgroup

\section{Introduction}\label{Sec1}

We study positive harmonic functions through the \textit{harmonic measures}
of increasing inner domains. The main question is how approximation by
boundary kernels can retain the positions of the corresponding boundary
points in a prescribed topology. We connect this approximation with
representation on a \textit{common refinement} of two boundaries and with
uniqueness after applying a \textit{gluing map}. The principal analytic
application is quantitative recovery of finite signed source measures
near compact sets of zero \textit{Newtonian capacity}.

\subsection{Harmonic measures and the approximation problem}\label{sec1_1}
Classical harmonic measure arises from the \textit{generalized Dirichlet
problem}; see \cite{Keni05,Helm69}. For a nonempty open set
$D\subseteq\R^n$ and boundary data $f$, write $\overline H_f$ and
$\underline H_f$ for the \textit{upper and lower Perron solutions}.
The function $f$ is \textit{resolutive} if these coincide in a finite
harmonic function, denoted by $H_f$.

\begin{thm}[Wiener]\label{Wiener}
If $f$ is a continuous real-valued function on the boundary $\partial D$ of the bounded open set $D$, then $f$ is resolutive.
\end{thm}

For bounded open $D$ and $x\in D$, the positive linear functional
$L_x(f):=H_f(x)$ on $C(\partial D)$ has a unique representing positive
Borel measure $\omega_x$, with $\omega_x(\partial D)=1$, such that
\begin{equation*}
 H_f(x)=L_x(f)=\int_{\partial D}f\,\d\omega_x.
\end{equation*}
This is the harmonic measure with \textit{reference point} $x$.
For Poisson integrals and positive harmonic functions on Lipschitz
domains, see \cite{Dahl79,Hunt70,Stei70}.

\medskip
\noindent\textbf{Definition (Exhaustion).}
An \textit{increasing exhaustion} of a domain $D$ is a sequence of domains
$(D_n)$ such that
\begin{equation*}
 D_n\subseteq D_{n+1}\subseteq D\quad(n\geq1),\qquad
 D=\bigcup_{n\geq1}D_n.
\end{equation*}
Such an exhaustion is \textit{relatively compact} if
\begin{equation*}
 D_n\Subset D_{n+1}\quad(n\geq1),
\end{equation*}
where $D_n\Subset D_{n+1}$ means that $\overline{D_n}$ is compact
and contained in $D_{n+1}$. A \textit{regular exhaustion} is a
relatively compact exhaustion for which each $D_n$ is regular.
\medskip

Fix $x_0\in D_1$. For each integer $n\geq1$, $x\in D_n$,
and $y\in\partial D_n$, the \textit{normalized boundary kernels}
and the measures used throughout the paper are
\begin{equation*}
 K_{D_n}(x,y)=\frac{\d\omega_x^{D_n}}{\d\omega_{x_0}^{D_n}}(y),
 \qquad
 \d\lambda_n^u(y)=u(y)\,\d\omega_{x_0}^{D_n}(y).
\end{equation*}
For each fixed $n$ and $x$, the density is initially defined
$\omega_{x_0}^{D_n}$-almost everywhere in $y$; under (A2) its
continuous version is used. For a positive harmonic function $u$, the mean-value identity gives
\begin{equation*}
 u(x)=\int_{\partial D_n}K_{D_n}(x,y)\,\d\lambda_n^u(y)
 \quad(x\in D_n),\qquad
 \lambda_n^u(\partial D_n)=u(x_0).
\end{equation*}
We seek limits of both the kernels and these measures on a prescribed
\textit{compactification}. Kernel convergence alone need not control
boundary positions; Example \ref{ex:failure-boundary-compatibility}
makes this distinction explicit.

The topological formulation is needed because the boundary used for
representation may differ from the geometric boundary. A common
refinement retains both boundary coordinates, and its projections
recover the two boundaries by identifying points in their fibers.
This motivates the measure framework on topological spaces in Section
\ref{Sec2}; compact metrizable compactifications suffice for Section
\ref{Sec4}. The construction also permits refined harmonic measures
to be pushed forward under boundary identifications.

\subsection{Main results}
The first two results use the Radon harmonic measure system of Section
\ref{Sec4}, whose measures are continuous and satisfy the nested
identity (N), together with assumptions (A1)--(A2). The capacity and
source recovery results use classical harmonic measures under the
geometric hypotheses stated below and allow irregular inner domains.

\medskip
\noindent\textbf{Synchronized selection (Theorem
\ref{thm:synchronized-kernel}).}
\emph{Let $h$ be a minimal positive harmonic function with $h(x_0)=1$,
and let $\xi$ be a boundary point of a prescribed compact metrizable
compactification. If
\begin{equation*}
 \liminf_{n\to\infty}
 \int_{\partial D_n\cap V}h(y)\,\d\omega_{x_0}^{D_n}(y)>0
\end{equation*}
for every open neighborhood $V$ of $\xi$, then, for every sufficiently
large $n$, there is $\xi_n\in\partial D_n$ such that
\begin{equation*}
 \xi_n\longrightarrow\xi,\qquad
 K_{D_n}(\cdot,\xi_n)\longrightarrow h
 \quad\text{locally uniformly in }D.
\end{equation*}}
This neighborhood nonvanishing condition allows selection along the
full sequence. The split-disk example in Subsection
\ref{subsec:split-disk} shows that it can hold without concentration
of $h\,\omega_{x_0}^{D_n}$ at $\xi$.

\medskip
\noindent\textbf{Refined representation and gluing (Theorem
\ref{prop:direct-refined-representation} and Corollary
\ref{thm:common-refinement-representation}).}
\emph{Let $D^K$ carry the surjective assignment (B) of normalized
minimal harmonic functions. Suppose one increasing sequence $(n_j)$,
common to all $\eta\in\partial_{\mathrm m}^KD$ and their open
neighborhoods $V$ in $D^K$, satisfies
\begin{equation*}
 \int_{\partial D_{n_j}\setminus V}K_D(y,\eta)\,
             \d\omega_{x_0}^{D_{n_j}}(y)\longrightarrow0.
\end{equation*}
For any second compact metrizable compactification $D^G$, form
$\widehat D=D^G\vee D^K$ with projection $p_K$, and put
$\widehat\partial_{\mathrm m}D=p_K^{-1}(\partial_{\mathrm m}^KD)$.
For every positive harmonic $u$, the measures
$u\,\omega_{x_0}^{D_{n_j}}$ have weakly convergent subsequences on
$\widehat D$. Every limit $\widehat\mu_u$ is carried by
$\widehat\partial_{\mathrm m}D$ and gives
\begin{equation*}
 u(x)=\int_{\widehat\partial_{\mathrm m}D}
       K_D(x,p_K(\widehat\eta))\,\d\widehat\mu_u(\widehat\eta)
 \quad(\forall x\in D).
\end{equation*}
At almost every point for this measure, the boundary point and its
kernel are simultaneously approximated along one common further
subsequence, with the selected points depending on the refined point. The
push-forward $(p_K)_\#\widehat\mu_u$ is the unique representing measure
on $\partial_{\mathrm m}^KD$.}
Uniqueness on the refinement is not required. The condition above is
$(C')$; under (C), these conclusions apply to all weak limits of the
original sequence. Intrinsic representation and its uniqueness already
follow from (A1)--(A2). Under $(C')$, injectivity and the Borel
properties of the boundary assignment follow automatically.
Concentration identifies that representation
with the limits in the chosen boundary topology.

\medskip
\noindent\textbf{Capacity and source recovery (Theorems
\ref{thm:polar-capacity-characterization},
\ref{thm:polar-measure-recovery}, and \ref{thm:curve-measure-recovery}).}
\emph{Let $\Omega\subset\R^d$, $d\geq3$, be a bounded $C^{1,1}$ domain,
let $\varnothing\ne F\Subset\Omega$ be compact, and suppose
$D=\Omega\setminus F$ is connected. Fix $x_0\in D$ and
let $K_D(\cdot,a)$, $a\in F$, be the normalized singular kernels
constructed from harmonic measures in Lemma \ref{lem:singular-kernels-harmonic-measure}.
Then $\operatorname{Cap}(F)=0$ if and only if
\begin{equation*}
 K_D(\cdot,a)\omega_{x_0}^{D_{n_j}}\Longrightarrow\delta_a
 \quad\text{for some }a\in F
\end{equation*}
on $\overline\Omega$ along a subsequence of one relatively compact exhaustion.
When this holds,
concentration holds on every prescribed relatively compact exhaustion,
uniformly in the boundary point; no regularity of the inner boundaries
is required.}

\emph{For every compact $F$ satisfying the geometric hypotheses above,
every finite signed Borel measure $\mu$ supported on $F$, and every
relatively compact exhaustion $(D_n)$ of $D$ with $x_0\in D_1$, set
$u_\mu(x)=\int_FK_D(x,a)\d\mu(a)$ for $x\in D$. Then, for all sufficiently large $n$,
\begin{equation*}
 d_{\mathrm{BL}}\bigl(u_\mu\omega_{x_0}^{D_n},\mu\bigr)
 \leq C|\mu|(F)\bigl(t_n+
          \operatorname{Cap}(F_{2t_n})^{1/(d-2)}\bigr),
\end{equation*}
where $t_n=\max_{y\in\partial D_n}\dist(y,\partial D)$ and
$F_t=\{x\in\R^d:\dist(x,F)\leq t\}$.
If $\operatorname{Cap}(F)=0$, this bound tends to zero and
$u_\mu\omega_{x_0}^{D_n}\Longrightarrow\mu$ on $\overline\Omega$.
If $d=3$ and $F$ is a finite rectifiable curve network, then, for all
sufficiently large $n$,
\begin{equation*}
 d_{\mathrm{BL}}\bigl(u_\mu\omega_{x_0}^{D_n},\mu\bigr)
 \leq\frac{C|\mu|(F)}{\log(1/t_n)}.
\end{equation*}}
Here $|\mu|$ is the \textit{variation measure} and $d_{\mathrm{BL}}$ is
the \textit{bounded Lipschitz distance} in
\eqref{eq:bounded-lipschitz-distance}. The constants are independent of
the shapes of $\partial D_n$. A segment and a point measure give the
matching logarithmic lower order for this recovery formula
(Example \ref{ex:curve-logarithmic-sharpness}). The estimates also allow
weighted $L^1$ errors in the boundary data. Corollary
\ref{cor:polar-source-recovery} recovers a source in the Poisson equation
in the presence of an unknown harmonic background, from function values
on boundaries approaching a known singular set.

The abstract proof uses joint limits recording position and kernel
together. Minimality gives concentration in kernel space. An integral
comparison and the maximal-measure criterion place the representing
measure on minimal functions. Kernel concentration and dominated
convergence then give uniqueness and convergence of the full sequence
of kernel measures. Boundary concentration identifies the refined
kernel, and gluing gives the projected representation. Condition (M)
is used only when the approximating inner kernels are required to be minimal.

The concentration criteria in Section \ref{subsubsec:concentration-criteria}
connect this argument with the analysis in Section \ref{Sec5}.
Proposition \ref{prop:exceptional-set-concentration} allows kernel decay
to fail on boundary sets with controlled harmonic measures. Its estimate
uses only normalized measures and the stated bounds, so it also applies
to irregular inner domains. Section \ref{Sec5} implements this criterion
using capacity. The singular kernels are constructed from normalized
harmonic measures of shrinking inner spheres. Dirichlet comparison
expresses them in terms of the Euclidean fundamental solution and a
harmonic correction. The distributional identity for that fundamental
solution gives the coefficient needed for source recovery, and an
exterior-ball barrier gives the outer boundary estimate. No domain
Green function is assumed in this construction.

Normalization and H\"older's inequality give a first moment estimate
in terms of the harmonic measure of the exceptional set. Integration
against a given signed measure then yields the recovery bounds, without
invoking the abstract representation theorem. Zero capacity makes the
capacitary bound tend to zero; for curve networks an explicit harmonic
majorant gives the logarithmic rate. For compact polar sets, the
concentration estimates verify (C). Together with the classical boundary
identification and a smooth inner exhaustion, this permits application of the representation
and gluing results of Section \ref{Sec4}.

\subsection{Relation to existing theory and organization}
Axiomatic potential theory and positive harmonic representation have
classical foundations in
\cite{Brelot62,Herve62,Helm69,EffrosKazdan70,Phelps01}.
Loeb \cite{Loeb76,Loeb82,Loeb19} constructs representing measures
from inner-domain harmonic measures and boundary partitions, including
a weak-limit construction. Our use of harmonic measures and the
absence of a prescribed Green kernel in the abstract proof therefore
have classical precedents. The specific problem here is simultaneous
control of boundary positions and kernels on varying inner domains,
together with realization of representing measures on a common refinement.
We use the maximal-measure criterion from Choquet theory
\cite{Phelps01} after constructing the measures from inner boundaries.

The framework is formulated in terms of mean-value identities and (N).
For PDE applications, Dirichlet solutions supply the measures, uniqueness
verifies (N), and appropriate Poisson-kernel regularity verifies (A2)
for the \textit{relative kernels}. Regularity for continuous Dirichlet
data alone does not establish the joint continuity required in (A2).
Under (A1)--(A2), positivity, local comparison estimates,
equicontinuity, and compactness are derived. Application of a
Brelot--Bauer representation theorem would additionally require
verification of its harmonic-sheaf axioms; see Example
\ref{ex:framework-local-to-global}. Projections between compactifications
also occur in nonlinear potential theory
\cite{BjornBjornSjodin18,Bjorn19}. Proposition
\ref{prop:harmonic-measure-gluing} gives harmonic measures on boundary
quotients by push-forward. A representation passes to such a quotient
when its kernel factors through a measurable kernel there, as in
Corollary \ref{cor:representation-under-gluing}.

Qualitative measure recovery is also classical. Marcus and V\'eron
\cite[Lemma 2.2]{MarcusVeron11} recover finite signed boundary measures
from inner-boundary values in bounded Lipschitz domains. Ancona and
Marcus \cite[Section 1.3]{AnconaMarcus15} formulate the recovery property
for arbitrary relatively compact exhaustions in that setting.
For deletion of a \textit{compact polar set}, the boundary identification
in \cite[Section 5.5]{GardinerPau03}, combined with the convergence
results of \cite[Theorems 5.1 and 7.1]{Doob57}, gives qualitative
concentration and recovery. Section \ref{Sec5} adds explicit capacity
and error bounds for the prescribed recovery measures, valid on every
relatively compact exhaustion without regularity assumptions on its
inner boundaries. The logarithmic scale is consistent with
classical thin-ellipsoid capacity formulas
\cite[Section 4]{vandenBerg25}; the sharpness assertion concerns the
specified formula over a class including point measures.
For equations with line sources, \cite{Gjerde19} studies singularity
splitting and finite element approximation. Our application estimates
source recovery using function values on approaching boundaries.

Sections \ref{Sec2} and \ref{Sec3} establish the measure framework,
maximum principles, Harnack estimates, and compactness. Section
\ref{Sec4} proceeds from inner-kernel approximation to refined
representation and then to gluing. Section \ref{Sec5} develops capacity
and source recovery, and verifies the boundary constructions on a
split disk and finite weighted metric graphs.

\section{The generalized harmonic measure}\label{Sec2}

In this section, we give the basic axiomatization of harmonic measure
and state the nested compatibility condition (N).  We then derive
the comparison results and local maximum principles from this framework.
Subsection \ref{subsec2_2} records the measure-theoretic notation used
throughout the paper, and the main discussion starts in Subsection
\ref{subsec2_3}.

\subsection{Measure theory}\label{subsec2_2}
We use $A\Subset U$ to mean that the closure of $A$ is compact and
contained in $U$.  The spaces $C(Y)$, $C_b(Y)$, and $C_c(Y)$
consist of real-valued continuous functions, bounded continuous functions,
and continuous functions with compact support, respectively.
For a function $f$ on $Y$, $f|_A$ is its \textit{restriction} to $A$, and
$\|f\|_A=\sup_{x\in A}|f(x)|$.  \textit{Locally uniform convergence} means
uniform convergence on every compact subset.

Unless otherwise stated, measures are countably additive positive Borel
measures, completed when
integrable functions are considered; see \cite{Evan91,Royd88}.  On a locally
compact Hausdorff space, \textit{Radon measures} are locally finite and inner regular.
We write $\mu(f)=\int f\d\mu$, and $L^1(\mu)$ denotes \textit{absolute
integrability}.  \textit{Mutual absolute continuity} means equality of null sets,
not equality of the corresponding $L^1$ classes.  We write
$\nu\ll\mu$ for \textit{absolute continuity}: $\mu(A)=0$ implies
$\nu(A)=0$.  The \textit{Radon--Nikodym derivative} $\d\nu/\d\mu$ is the
density satisfying $\nu(A)=\int_A(\d\nu/\d\mu)\d\mu$.
The notation $\chi_A$ denotes the \textit{indicator} of $A$, and
$\delta_y$ denotes the \textit{Dirac measure} at $y$, characterized by
$\int f\,d\delta_y=f(y)$.

For a measurable map $p$, the \textit{push-forward} is
$p_\#\mu(A)=\mu(p^{-1}(A))$; hence
\begin{equation}\label{eq:pushforward-integration}
 \int f\,d(p_\#\mu)=\int f\circ p\d\mu
\end{equation}
for nonnegative measurable or integrable $f$.  On a compact space,
$\mu_n\Longrightarrow\mu$ denotes \textit{weak convergence}, that is,
$\int f\,d\mu_n\to\int f\d\mu$ for every continuous $f$.  On a
noncompact locally compact space, convergence against $C_c$ is called
\textit{vague convergence}; it does not by itself imply $\mu_n(Y)\to\mu(Y)$ on the underlying space $Y$.

We use the Radon--Nikodym theorem, Tonelli's theorem, monotone convergence,
and the Riesz--Markov theorem in their standard forms
\cite{Evan91,Royd88}.  We also use sequential weak compactness of finite
positive measures $\mu_n$ on compact metric spaces $Y$ with $\sup_n\mu_n(Y)<\infty$, and vague
sequential compactness under uniform bounds on compact subsets of
$\mathbb{R}^n$.

\subsection{The generalized harmonic measure and maximum principle}\label{subsec2_3}
Let $X$ be \textit{first-countable} and Hausdorff.  A \textit{domain} means a nonempty
connected open set.  Every domain equipped with harmonic measures below has nonempty boundary.  For a disconnected open set in a
locally connected space, the statements are interpreted separately on
its components.  No equivalence of measures or constancy across distinct
components is asserted.  First countability allows boundary limits to
be formulated using sequences.

\begin{defi}[Pre-harmonic measure]\label{pre_HM}
A \textit{pre-harmonic measure} on a domain $D$ is a family of Borel measures
$\omega_\cdot^D=\{\omega_x^D:x\in D\}$ on $\partial D$ such that
\begin{equation*}
 \omega_x^D(\partial D)=1\quad(\forall x\in D),\qquad
 \omega_x^D\ll\omega_y^D\quad(\forall x,y\in D).
\end{equation*}
\end{defi}

\begin{prop}
Suppose $\omega_\cdot^D$ is a pre-harmonic measure on a domain $D$.
Then, for any $x,y\in D$ and any Borel set $A\subseteq\partial D$,
\begin{equation*}
 \omega_x^D(A)=0\quad\Longleftrightarrow\quad\omega_y^D(A)=0.
\end{equation*}
Consequently, the completed Borel sigma-algebras of $\omega_x^D$ and
$\omega_y^D$ coincide.
\end{prop}
\begin{proof}
Apply absolute continuity in both directions.  The completions are
obtained by adjoining subsets of the common Borel null sets.
\end{proof}

\begin{defi}[Harmonic measure]\label{H_m}
We call the completed family a \textit{harmonic measure} and put
\begin{equation*}
 L(\partial D;\omega_\cdot^D)
 :=\bigcap_{x\in D}L^1(\partial D,\omega_x^D),\qquad
 \omega_x^D(f):=\int_{\partial D}f(y)\,d\omega_x^D(y).
\end{equation*}
It is \textit{continuous} if $x\mapsto\omega_x^D(f)$ is continuous on $D$ for every $f$ in this
intersection.  A point $\zeta\in\partial D$ is \textit{regular} if
\begin{equation*}
 \omega_x^D(\partial D\setminus V)\longrightarrow0
 \quad(D\ni x\to\zeta)
\end{equation*}
for every Borel neighborhood $V$ of $\zeta$ in $\partial D$.
Equivalently,
\begin{equation}\label{s1}
 \omega_x^D(f)\longrightarrow f(\zeta)\quad(D\ni x\to\zeta)
\end{equation}
for every bounded Borel function $f$ continuous at $\zeta$.
The equivalence is proved in Proposition \ref{R_point}.
The domain is \textit{regular} if every boundary point is regular.
For values on the boundary we set $\omega_y^D(f)=f(y)$, $y\in\partial D$.
This convention does not assert a boundary limit at irregular points.
\end{defi}
Unbounded integrable data are allowed in the interior definition, but
boundary limits for such data require additional control and are not
part of regularity.

\begin{prop}\label{R_point}
Let $\omega_\cdot^D$ be a harmonic measure on a domain $D$.
Then the following hold.
\begin{enumerate}
\item $C_b(\partial D)\subset L(\partial D;\omega_\cdot^D)$;
if $\partial D$ is compact, this includes $C(\partial D)$.
\item For fixed $x,y\in D$, if there exist constants
$0<c\le C<\infty$ such that
$c\omega_y^D\le\omega_x^D\le C\omega_y^D$, then
\begin{equation*}
 L^1(\partial D,\omega_x^D)=L^1(\partial D,\omega_y^D).
\end{equation*}
The constants here may depend on the pair $(x,y)$.
Mutual absolute continuity alone does not suffice.
\item At a regular point $\zeta$, for every $f\in C_b(\partial D)$,
\begin{equation}\label{w1}
 \lim_{D\ni x\to\zeta}\omega_x^D(f)=f(\zeta).
\end{equation}
\item If $V$ is a Borel boundary neighborhood of a regular point $\zeta$,
then
\begin{equation*}
 \lim_{D\ni x\to\zeta}\omega_x^D(\partial D\setminus V)=0,
 \qquad
 \lim_{D\ni x\to\zeta}\omega_x^D(V)=1.
\end{equation*}
\end{enumerate}
\end{prop}
\begin{proof}
If $f\in C_b(\partial D)$, then for each $x\in D$,
\begin{equation*}
 \int_{\partial D}|f|\,d\omega_x^D
 \le\|f\|_\infty\omega_x^D(\partial D)=\|f\|_\infty<\infty.
\end{equation*}
Thus $f$ belongs to the intersection defining
$L(\partial D;\omega_\cdot^D)$.  When $\partial D$ is compact,
every continuous function on it is bounded.
For the second assertion, the assumed inequalities between measures give
\begin{equation*}
 c\int_{\partial D}|f|\,d\omega_y^D
 \le\int_{\partial D}|f|\,d\omega_x^D
 \le C\int_{\partial D}|f|\,d\omega_y^D.
\end{equation*}
Consequently either integral is finite if and only if the other is finite.

We first verify the equivalence in Definition \ref{H_m}.
Suppose the neighborhood condition holds, and let $f$ be bounded,
Borel, and continuous at $\zeta$.  For $\varepsilon>0$, choose a
relatively open neighborhood $V$ of $\zeta$ such that
$|f(y)-f(\zeta)|<\varepsilon$ for $y\in V$.
Normalization gives, for every $x\in D$,
\begin{equation*}
 \begin{aligned}
 |\omega_x^D(f)-f(\zeta)|
 &\le \int_{\partial D}|f(y)-f(\zeta)|\,\d\omega_x^D(y)\\
 &\le \varepsilon\omega_x^D(V)
       +2\|f\|_\infty\omega_x^D(\partial D\setminus V)\\
 &\le \varepsilon
       +2\|f\|_\infty\omega_x^D(\partial D\setminus V).
 \end{aligned}
\end{equation*}
Letting $D\ni x\to\zeta$ and then $\varepsilon\downarrow0$
proves \eqref{s1}, and in particular the third assertion.

Conversely, suppose \eqref{s1} holds for every bounded Borel function
continuous at $\zeta$, and let $V$ be a Borel boundary neighborhood
of $\zeta$.
There is a relatively open neighborhood $W$ of $\zeta$ with
$W\subseteq V$; hence $\chi_V=1$ near $\zeta$ and is continuous
at that point.  Applying \eqref{s1} yields
\begin{equation*}
 \lim_{D\ni x\to\zeta}\omega_x^D(V)
 =\lim_{D\ni x\to\zeta}\omega_x^D(\chi_V)
 =\chi_V(\zeta)=1.
\end{equation*}
Since $\omega_x^D(\partial D)=1$, it follows that
$\omega_x^D(\partial D\setminus V)=1-\omega_x^D(V)\to0$.
This proves the converse and the fourth assertion.
\end{proof}

\begin{prop}\label{H_A}
Suppose $\omega_\cdot^D$ is a harmonic measure on a domain $D$.
For any nonempty relatively open set $A\subseteq\partial D$ containing
a regular point,
\begin{equation*}
 \omega_x^D(A)>0\qquad(\forall x\in D).
\end{equation*}
In particular, if $D$ is regular, then
\begin{equation*}
 \operatorname{supp}\omega_x^D=\partial D\qquad(\forall x\in D).
\end{equation*}
\end{prop}
\begin{proof}
Let $A\subseteq\partial D$ be relatively open and let
$\zeta\in A$ be regular.  The function $\chi_A$ is bounded,
Borel, and continuous at $\zeta$, with value one.  Hence
$\omega_x^D(A)\to1$ as $D\ni x\to\zeta$.
Suppose that $\omega_{x_0}^D(A)=0$ at some interior point $x_0\in D$.
For every $x\in D$, absolute continuity
$\omega_x^D\ll\omega_{x_0}^D$ then gives $\omega_x^D(A)=0$,
contradicting this boundary limit.  Thus $\omega_x^D(A)>0$ at every
interior point.  If every boundary point is regular, every nonempty relatively
open boundary set contains such a point.  This is precisely full support.
\end{proof}

The following elementary consequences isolate the uses of boundary
regularity and equivalence of measures.
\begin{thm}[Strong maximum principle for $\chi_A$ on a regular open set]
Suppose $D$ is regular and $A\subseteq\partial D$ is nonempty and
relatively open.  Then
\begin{equation*}
 \sup_{x\in D}\omega_x^D(A)=1.
\end{equation*}
Moreover,
\begin{equation*}
 \bigl[\exists x_0\in D:\ \omega_{x_0}^D(A)=1\bigr]
 \quad\Longrightarrow\quad
 \bigl[\forall x\in D:\ \omega_x^D(A)=1\bigr].
\end{equation*}
\end{thm}
\begin{proof}
Since $\omega_x^D(\partial D)=1$,
$0\le\omega_x^D(A)\le1$ for every $x\in D$.  Choose
$\zeta\in A$.  Relative openness makes $\chi_A$ continuous at
$\zeta$, and regularity gives $\omega_x^D(A)\to1$ as
$D\ni x\to\zeta$.  More explicitly, first countability and
$\zeta\in\overline D$ give a sequence $x_j\in D$ tending to
$\zeta$, for which $\omega_{x_j}^D(A)\to1$.  Thus the supremum
is at least one and hence equals one.

If it is attained at $x_0\in D$, then
\begin{equation*}
 \omega_{x_0}^D(\partial D\setminus A)
 =1-\omega_{x_0}^D(A)=0.
\end{equation*}
Absolute continuity gives $\omega_x^D(\partial D\setminus A)=0$
for every $x\in D$.  Using $\omega_x^D(\partial D)=1$ proves
$\omega_x^D(A)=1$ throughout $D$.
\end{proof}

\begin{thm}[Weak maximum principle for regular harmonic measure]\label{WMPr}
Suppose $D$ is regular and $f\in C_b(\partial D)$. Put
$M=\sup_{y\in\partial D}f(y)$. Then
\begin{equation*}
 \omega_x^D(f)\le M\quad(\forall x\in D),\qquad
 \sup_{x\in D}\omega_x^D(f)=M.
\end{equation*}
In particular, $M$ is attained at a boundary point when $\partial D$
is compact.
\end{thm}
\begin{proof}
Put $M=\sup_{y\in\partial D}f(y)$.  Since $f$ is bounded and
$\partial D\ne\varnothing$, $M$ is finite.  Positivity and
normalization give
\begin{equation*}
 \omega_x^D(f)=\int_{\partial D}f(y)\,d\omega_x^D(y)
 \le\int_{\partial D}M\,d\omega_x^D=M,
 \qquad x\in D.
\end{equation*}
Consequently $\sup_{x\in D}\omega_x^D(f)\le M$.
For $\varepsilon>0$, choose $\zeta\in\partial D$ with
$f(\zeta)>M-\varepsilon$.  Regularity and continuity of $f$ at
$\zeta$ give
\begin{equation*}
 \omega_x^D(f)\longrightarrow f(\zeta)
 \quad(D\ni x\to\zeta).
\end{equation*}
Taking a sequence in $D$ converging to $\zeta$, we obtain
\begin{equation*}
 \sup_{x\in D}\omega_x^D(f)\ge f(\zeta)>M-\varepsilon.
\end{equation*}
Letting $\varepsilon\downarrow0$ proves the reverse inequality.
If the boundary is compact, continuity of $f$ makes its supremum a
maximum.  No boundary attainment is needed in the general case.
\end{proof}

\begin{thm}[Boundary strong maximum principle]\label{BSMP}
Let $\omega_\cdot^D$ be a harmonic measure on a domain $D$, and let
$f\in L(\partial D;\omega_\cdot^D)$. Suppose
$M:=\operatorname*{ess\,sup}_{\partial D}f<\infty$, where the
essential supremum is taken with respect to any one of the measures
$\omega_x^D$; its value is independent of $x\in D$. Then
\begin{equation*}
 \bigl[\exists x_0\in D:\ \omega_{x_0}^D(f)=M\bigr]
 \quad\Longrightarrow\quad
 \bigl[\forall x\in D:\ \omega_x^D(f)=M\bigr].
\end{equation*}
In this case $f=M$ almost everywhere for every $\omega_x^D$, $x\in D$.
\end{thm}
\begin{proof}
The essential supremum is the same for every interior point, because the
measures have the same null sets.  Thus $f\le M$ almost everywhere
for each of them.  At the specified interior point, normalization gives
\begin{equation*}
 \int_{\partial D}(M-f)\,d\omega_{x_0}^D
 =M-\omega_{x_0}^D(f)=0.
\end{equation*}
For each integer $m\ge1$, let
$E_m=\{y\in\partial D:f(y)\le M-1/m\}$.  Nonnegativity of
$M-f$ almost everywhere implies
\begin{equation*}
 0\ge\frac1m\omega_{x_0}^D(E_m)\ge0.
\end{equation*}
Hence every $E_m$ is null.  Since $\{f<M\}=\bigcup_{m\ge1}E_m$
and $\{f>M\}$ is null by the definition of $M$, we have
$f=M$ almost everywhere for $\omega_{x_0}^D$.
Absolute continuity transfers this equality to $\omega_x^D$ for
any $x\in D$.  Finally,
\begin{equation*}
 \omega_x^D(f)=\int_{\partial D}M\,d\omega_x^D=M.
\end{equation*}
No regularity of boundary points is needed in this argument.
\end{proof}

\begin{thm}[Strong maximum principle for regular harmonic measure]\label{SMPr}
Under the assumptions of Theorem \ref{WMPr}, put
$M=\sup_{z\in D}\omega_z^D(f)$. If
\begin{equation*}
 \exists x_0\in D:\quad \omega_{x_0}^D(f)=M,
\end{equation*}
then
\begin{equation*}
 f(y)=M\quad(\forall y\in\partial D),\qquad
 \omega_x^D(f)=M\quad(\forall x\in D).
\end{equation*}
\end{thm}
\begin{proof}
Let $x_0\in D$ be a point where the supremum is attained.
By Theorem \ref{WMPr},
\begin{equation*}
 \omega_{x_0}^D(f)=\sup_{x\in D}\omega_x^D(f)
 =M:=\sup_{\partial D}f.
\end{equation*}
The function $g=M-f$ is nonnegative and continuous on $\partial D$,
and its integral at $x_0$ is zero.  If $g(y_0)>0$ at some boundary
point, continuity supplies a nonempty relatively open neighborhood
$A$ of $y_0$ and $\varepsilon>0$ such that $g\ge\varepsilon$
on $A$.  Proposition \ref{H_A} gives $\omega_{x_0}^D(A)>0$, so
\begin{equation*}
 0=\int_{\partial D}g\,d\omega_{x_0}^D
 \ge\varepsilon\omega_{x_0}^D(A)>0,
\end{equation*}
a contradiction.  Hence $f=M$ on the entire boundary.  Since the
measures satisfy $\omega_x^D(\partial D)=1$, $\omega_x^D(f)=M$ for every $x\in D$.
This proves constancy and also identifies the boundary values.
\end{proof}

\begin{defi}[Harmonic measure system]\label{def:harmonic-measure-system}\label{Def_Hms}
A \textit{harmonic measure system} is a family of continuous harmonic measures
$\{\omega_\cdot^{D_\lambda}\}_{\lambda\in\Lambda}$ satisfying
the following \textit{nested compatibility condition}: for every pair
$V\Subset U$ in the family with $V$ regular, and every bounded
Borel function $f$ on $\partial U$,
\begin{equation*}
 \omega_x^U(f)=\int_{\partial V}\omega_y^U(f)\,\d\omega_x^V(y)
 \qquad(\forall x\in V).
 \tag{N}\label{eq:section-four-nested}
\end{equation*}
The integrand is continuous on the compact set $\partial V$ by
Definition \ref{H_m}, so the integral is defined.
The system is \textit{regular closed} if every domain in the family
admits an increasing exhaustion by regular domains in the same family,
with compact closures contained in the domain being exhausted.
Regular closedness is imposed only where it is stated. For the strong
maximum principle below, it suffices that relatively compact regular
domains cover the domain under consideration.
\end{defi}

For comparison, a stronger compatibility requirement on arbitrary
intersecting domains $U,\widetilde U$ is the following: if integrable
boundary functions $f,\widetilde f$ satisfy
\begin{equation*}
 \omega_\cdot^{\widetilde U}(\widetilde f)=f
       \text{ on }\widetilde U\cap\partial U,\qquad
 \omega_\cdot^U(f)=\widetilde f
       \text{ on }\partial\widetilde U\cap\overline U,
\end{equation*}
then their extensions agree on $U\cap\widetilde U$.
This requirement implies (N): take $\widetilde U=V\Subset U$
and $\widetilde f=\omega_\cdot^U(f)|_{\partial V}$.
The first matching condition is empty and the second holds by definition.
The results below use (N), without this stronger requirement.

\begin{prop}[Nested compatibility]\label{prop:nested-compatibility}\label{special_compatible}
Let $V\Subset U$ be domains in the system, with $V$ regular.
For every $f\in L(\partial U;\omega_\cdot^U)$,
\begin{equation}\label{eq:nested-integrable-data}
 \omega_x^U(f)=\int_{\partial V}\omega_y^U(f)\,\d\omega_x^V(y),
 \qquad(\forall x\in V).
\end{equation}
Thus (N) also holds for the completed measurable integrable data of
Definition \ref{H_m}.
\end{prop}

\begin{proof}
Choose a Borel representative of $f$.  The common null sets of the
measures on $\partial U$ ensure that this choice changes none of
the values $\omega_y^U(f)$.
First suppose $f\geq0$, and put $f_j=\min\{f,j\}$.
For every $y\in U$, monotone convergence gives
\begin{equation*}
 \omega_y^U(f_j)\uparrow\omega_y^U(f).
\end{equation*}
Apply (N) to $f_j$, and then apply monotone convergence on
$\partial V$.  For $x\in V$ this gives
\begin{equation*}
 \begin{aligned}
 \omega_x^U(f)
 &=\lim_j\omega_x^U(f_j)
 =\lim_j\int_{\partial V}\omega_y^U(f_j)\,\d\omega_x^V(y)\\
 &=\int_{\partial V}\omega_y^U(f)\,\d\omega_x^V(y)<\infty.
 \end{aligned}
\end{equation*}
For a real integrable $f$, apply this argument separately to $f^+$
and $f^-$.  Both integrals are finite, and subtraction proves
\eqref{eq:nested-integrable-data}.  In particular,
\begin{equation*}
 \int_{\partial V}|\omega_y^U(f)|\,\d\omega_x^V(y)
 \leq\int_{\partial V}\omega_y^U(|f|)\,\d\omega_x^V(y)
 =\omega_x^U(|f|)<\infty.
\end{equation*}
This also verifies the required integrability of the inner boundary
function directly.
\end{proof}

\begin{thm}[Strong maximum principle for an open set]\label{Str_MP_G}
Let $D$ be a domain such that
\begin{equation}\label{eq:local-regular-cover}
 \forall x\in D,\quad
 \exists\,V\in\{D_\lambda\}:\quad
 x\in V\Subset D,\qquad V\text{ is regular}.
\end{equation}
Suppose $u\in C(D)$ satisfies, for every regular $V\Subset D$
in the system,
\begin{equation*}
 u(x)=\int_{\partial V}u(y)\,d\omega_x^V(y)
 \qquad(\forall x\in V).
\end{equation*}
If
\begin{equation*}
 M:=\sup_{z\in D}u(z)<\infty,\qquad
 \exists x_0\in D:\ u(x_0)=M,
\end{equation*}
then $u(x)=M$ for every $x\in D$.
In particular, this applies to $u=\omega_\cdot^D(f)$ whenever
$D$ belongs to the system and $f\in L(\partial D;\omega_\cdot^D)$.
\end{thm}
\begin{proof}
Put $Z:=\{x\in D:u(x)=M\}$. Continuity makes $Z$ closed in $D$,
and $x_0\in Z$. We show that $Z$ is open.
Fix $x\in Z$ and choose a regular $V$ as in
\eqref{eq:local-regular-cover}. The function $M-u$ is nonnegative,
continuous, and bounded on the compact set $\partial V\subset D$.
The mean-value identity gives
\begin{equation*}
 \begin{aligned}
 0=M-u(x)
 &=M\omega_x^V(\partial V)-\int_{\partial V}u(y)\,d\omega_x^V(y)\\
 &=\int_{\partial V}(M-u(y))\,d\omega_x^V(y).
 \end{aligned}
\end{equation*}
Hence $M-u=0$ almost everywhere for $\omega_x^V$.
For every $z\in V$, absolute continuity
$\omega_z^V\ll\omega_x^V$ transfers this equality to $\omega_z^V$.
The same identity at $z$ yields
\begin{equation*}
 M-u(z)=\int_{\partial V}(M-u(y))\,d\omega_z^V(y)=0.
\end{equation*}
Thus $x\in V\subseteq Z$. The set $Z$ is both open and closed in
the connected set $D$, so $Z=D$.

For $u=\omega_\cdot^D(f)$, continuity is part of the system
assumptions, and Proposition \ref{prop:nested-compatibility} gives
the required identities. Only the restrictions of $u$ to compact
inner boundaries must be bounded; $f$ need only be integrable.
\end{proof}

\begin{rem}[A local regular cover suffices]
An increasing exhaustion by relatively compact regular domains implies
\eqref{eq:local-regular-cover}. Theorem \ref{Str_MP_G} therefore applies,
in particular, to every domain in a regular closed system. Its proof
only needs the covering condition on $D$; no ordering of the covering
domains or exhaustion of other domains is required.
\end{rem}

The corresponding minimum principles follow by applying the preceding
results to $-f$ or $-u$. In particular, if $D$ is regular and
$f\in C_b(\partial D)$, put $m=\inf_{y\in\partial D}f(y)$. Then
\begin{equation*}
 \omega_x^D(f)\ge m\quad(\forall x\in D),\qquad
 \inf_{x\in D}\omega_x^D(f)=m.
\end{equation*}
If there exists $x_0\in D$ with $\omega_{x_0}^D(f)=m$, then
$f(y)=m$ for every $y\in\partial D$ and $\omega_x^D(f)=m$
for every $x\in D$. For $f\in L(\partial D;\omega_\cdot^D)$,
without boundary regularity, the corresponding assertion is
\begin{equation*}
 \begin{gathered}
 m:=\operatorname*{ess\,inf}_{\partial D}f> -\infty,\qquad
 \exists x_0\in D:\ \omega_{x_0}^D(f)=m\\
 \Longrightarrow\quad \omega_x^D(f)=m\quad(\forall x\in D).
 \end{gathered}
\end{equation*}
Finally, for a continuous function $u$ satisfying the mean-value
identities and the covering assumption of Theorem \ref{Str_MP_G},
\begin{equation*}
 \begin{gathered}
 m:=\inf_{z\in D}u(z)>-\infty,\qquad
 \exists x_0\in D:\ u(x_0)=m\\
 \Longrightarrow\quad u(x)=m\quad(\forall x\in D).
 \end{gathered}
\end{equation*}
Linearity gives $\omega_x^D(-f)=-\omega_x^D(f)$, so these
statements require no repetition of the maximum-principle proofs.

\subsection{Harnack inequality and Harnack principle}\label{sec2.4}
\begin{defi}
For a nonempty subset $A\subset D$, define
\begin{equation*}
 H(A;D):=\sup_{\substack{x,y\in A\\0\le f\in L(\partial D;\omega_\cdot^D)\\
                         \omega_y^D(f)>0}}
                 \frac{\omega_x^D(f)}{\omega_y^D(f)}.
\end{equation*}
This is the \textit{Harnack index} of $A$ relative to $D$.
For $\alpha\in[1,\infty]$, write
\begin{equation*}
 D_\alpha:=\{A\subseteq D:A\ne\varnothing,\ H(A;D)=\alpha\}.
\end{equation*}
The sets in $D_\alpha$ are called \textit{Harnack subsets} with index $\alpha$.
Equivalence of the measures makes the positivity of the denominator
independent of the interior point.  Functions vanishing almost everywhere are
excluded.
\end{defi}

\begin{prop}\label{Harnack_basic}
Let $\varnothing\ne A\subseteq D$. One has $1\le H(A;D)\le\infty$.  If $A\subset V\Subset D$, where $V,D$ are
domains in the system and $V$ is regular, then
\begin{equation*}
 H(A;D)\le H(A;V).
\end{equation*}
Moreover, if $H(A;D)<\infty$, then, for every
$0\le f\in L(\partial D;\omega_\cdot^D)$,
\begin{equation*}
 \omega_x^D(f)\le H(A;D)\omega_y^D(f)
 \qquad(\forall x,y\in A).
\end{equation*}
There is no general assertion that $H(D;D)=\infty$: a family constant
in the interior point has $H(D;D)=1$.
\end{prop}
\begin{proof}
The constant function $f=1$ gives
$\omega_x^D(1)/\omega_y^D(1)=1$, so $H(A;D)\ge1$.
For monotonicity, let $0\le f\in L(\partial D;\omega_\cdot^D)$
with $\omega_y^D(f)>0$.  On $\partial V$, put
$g(z)=\omega_z^D(f)$.  Continuity makes $g$ bounded on the
compact set $\partial V\subset D$.
Proposition \ref{prop:nested-compatibility}, applied to the regular
domain $V$, gives, for $x,y\in A$,
\begin{equation*}
 \frac{\omega_x^D(f)}{\omega_y^D(f)}
 =\frac{\int_{\partial V}g(z)\,d\omega_x^V(z)}
        {\int_{\partial V}g(z)\,d\omega_y^V(z)}
 \le H(A;V).
\end{equation*}
Taking the supremum of these ratios gives
$H(A;D)\le H(A;V)$.  Similarly, inclusion of the sets of pairs of interior points
gives $H(A;D)\le H(V;D)$.

If $H(A;D)<\infty$ and $\omega_y^D(f)>0$, multiplying the
defining ratio inequality by $\omega_y^D(f)$ proves the stated bound.
If $\omega_y^D(f)=0$, nonnegativity implies $f=0$ almost everywhere
for that measure.  Equivalence gives $\omega_x^D(f)=0$ at every
interior point, so the inequality still holds.  Finally, a family independent of
the interior point gives ratio one for every admissible function, showing why
$H(D;D)=\infty$ cannot be asserted in general.
\end{proof}

\begin{thm}[General Harnack inequality]\label{Ha}
Fix $\varnothing\ne A\subseteq D$ and a constant $1\le C<\infty$.
The following conditions are equivalent:
\begin{enumerate}
\item $H(A;D)\le C$;
\item $\omega_x^D\le C\omega_y^D$ for every $x,y\in A$;
\item $\|d\omega_x^D/d\omega_y^D\|_{L^\infty(\omega_y^D)}\le C$
for every $x,y\in A$.
\end{enumerate}
\end{thm}
\begin{proof}
Suppose first that $H(A;D)\le C$.  For any Borel
$E\subseteq\partial D$, the indicator $\chi_E$ is integrable.
Proposition \ref{Harnack_basic}, including its zero-denominator case,
gives
\begin{equation*}
 \omega_x^D(E)=\omega_x^D(\chi_E)
 \le C\omega_y^D(\chi_E)=C\omega_y^D(E).
\end{equation*}
Thus $\omega_x^D\le C\omega_y^D$.
Conversely, this measure inequality gives the same integral inequality
first for indicators, then for nonnegative simple functions, and finally
by monotone convergence for all nonnegative measurable functions.  In
particular, for every admissible $f$,
\begin{equation*}
 \omega_x^D(f)=\int_{\partial D}f\,d\omega_x^D
 \le C\int_{\partial D}f\,d\omega_y^D=C\omega_y^D(f).
\end{equation*}
Taking the supremum of the ratios proves $H(A;D)\le C$.

Let $k_{x,y}=d\omega_x^D/d\omega_y^D$.  If
$k_{x,y}\le C$ almost everywhere for $\omega_y^D$, then
\begin{equation*}
 \omega_x^D(E)=\int_E k_{x,y}\,d\omega_y^D\le C\omega_y^D(E).
\end{equation*}
Conversely, assume the measure inequality and put
$E_m=\{k_{x,y}>C+1/m\}$.  Then
\begin{equation*}
 (C+1/m)\omega_y^D(E_m)
 \le\int_{E_m}k_{x,y}\,d\omega_y^D
 =\omega_x^D(E_m)\le C\omega_y^D(E_m),
\end{equation*}
so $\omega_y^D(E_m)=0$.  Since $\{k_{x,y}>C\}=\bigcup_m E_m$,
we obtain the asserted essential bound.  The constant must be uniform
in $x,y\in A$; separate boundedness for each pair is not sufficient.
\end{proof}

\begin{cor}[Harnack inequality]\label{C_Ha}
If $\partial D$ is compact and the relative densities have a jointly
continuous version in $(x,y,z)\in D\times D\times\partial D$, then
\begin{equation*}
 H(A;D)<\infty\qquad(\forall\,\varnothing\ne A\Subset D).
\end{equation*}
\end{cor}
\begin{proof}
Write $k(x,y,z)$ for the jointly continuous version of the
relative density.  Since $A\Subset D$, the product
$\overline A\times\overline A\times\partial D$ is compact.  Hence
\begin{equation*}
 C:=\max\left\{1,
 \max_{(x,y,z)\in\overline A\times\overline A\times\partial D}
 |k(x,y,z)|\right\}<\infty.
\end{equation*}
For $x,y\in A$ and $0\le f\in L(\partial D;\omega_\cdot^D)$,
\begin{equation*}
 \begin{aligned}
 \omega_x^D(f)&=\int_{\partial D}f(z)k(x,y,z)\,d\omega_y^D(z)\\
 &\le C\int_{\partial D}f(z)\,d\omega_y^D(z)
 =C\omega_y^D(f).
 \end{aligned}
\end{equation*}
Theorem \ref{Ha} therefore gives $H(A;D)\le C<\infty$.
Interchanging $x$ and $y$ also gives
$C^{-1}\omega_y^D(f)\le\omega_x^D(f)$.
\end{proof}

For a Euclidean system, \textit{translation--scaling invariance} (TSCI) means
that, for every $a\in\mathbb R^n$ and $t>0$, the map
$\nu(x)=a+tx$ carries every domain $V$ in the system to a domain
in the same system and
$\omega_{\nu(x)}^{\nu(V)}=\nu_\#\omega_x^V$ for every $x\in V$.
\begin{thm}[Harnack principle in $\mathbb{R}^n$]
Assume TSCI. Let $V,D$ be domains in the system, let $V$ be regular,
and let $\varnothing\ne A\subset V$ satisfy $H(A;V)<\infty$.
For any $a\in\mathbb R^n$ and $t>0$, set $\nu(x)=a+tx$.
If $\nu(V)\Subset D$, then
\begin{equation*}
 H(\nu(A);D)\le H(A;V).
\end{equation*}
Consequently, for $x,y\in\nu(A)$ and nonnegative
$f\in L(\partial D;\omega_\cdot^D)$,
\begin{equation*}
 \omega_x^D(f)\le H(A;V)\omega_y^D(f).
\end{equation*}
\end{thm}
\begin{proof}
Put $W=\nu(V)$.  Translation--scaling invariance means
$\omega_{\nu(s)}^W=\nu_\#\omega_s^V$.
It also preserves regularity: the inverse image under $\nu$ of a
boundary neighborhood is a boundary neighborhood, and the concentration
condition in Definition \ref{H_m} transfers through this identity.
Thus $W$ is regular.  Consequently, for every
nonnegative integrable boundary function $g$ on $\partial W$,
\begin{equation*}
 \omega_{\nu(s)}^W(g)
 =\int_{\partial V}g(\nu(z))\,d\omega_s^V(z)
 =\omega_s^V(g\circ\nu).
\end{equation*}
Composition with $\nu$ is a bijection between the admissible boundary
functions for the two domains; the inverse is composition with
$\nu^{-1}$.  Applying this identity at any two points of $A$ and taking
suprema of the ratios proves
\begin{equation*}
 H(\nu(A);W)=H(A;V).
\end{equation*}
Now let $0\le f\in L(\partial D;\omega_\cdot^D)$, and set
$g(z)=\omega_z^D(f)$ on $\partial W$.
Since $W\Subset D$, this restriction is continuous and bounded.  Proposition
\ref{prop:nested-compatibility} and the preceding
identity give, for $x,y\in\nu(A)$,
\begin{equation*}
 \begin{aligned}
 \omega_x^D(f)&=\int_{\partial W}g(z)\,d\omega_x^W(z)\\
 &\le H(A;V)\int_{\partial W}g(z)\,d\omega_y^W(z)\\
 &=H(A;V)\omega_y^D(f).
 \end{aligned}
\end{equation*}
Taking the supremum over nonzero denominators yields
$H(\nu(A);D)\le H(A;V)$, as claimed.
In particular, Corollary \ref{C_Ha} ensures $H(A;V)<\infty$ if
$A\Subset V$, $\partial V$ is compact, and the relative densities
on $V\times V\times\partial V$ have a jointly continuous version.
It is the boundary of the inner domain $V$ that is used here.
\end{proof}
This is a comparison statement, not a monotone convergence theorem.
Neither TSCI nor the Euclidean topology is used in Section \ref{Sec4}.

\section{Harmonic functions and compactness}\label{Sec3}

\subsection{Harmonic functions in a harmonic measure system}

Throughout this section, $\{\omega_\cdot^{D_\lambda}\}_{\lambda\in\Lambda}$
is a Radon harmonic measure system satisfying (N) on a first-countable
Hausdorff space $X$. Local regular covers and regular inner exhaustions
are assumed only where they are stated below.

\begin{defi}\label{def:system-harmonic}
Let $U\subseteq X$ be open.  A function $u:U\to\mathbb{R}$ is
\textit{harmonic in $U$} if it is continuous and
\begin{equation*}
 u(x)=\omega_x^V(u):=\int_{\partial V}u(y)\,\d\omega_x^V(y)
\end{equation*}
for every regular $V\in\{D_\lambda\}$ whose closure is compact and
contained in $U$, and every $x\in V$.  The notation $V\Subset U$
means that $\overline V$ is compact and contained in $U$.
\end{defi}

The mean-value identities and Proposition
\ref{prop:nested-compatibility} give the following elementary facts.

\begin{prop}\label{prop:basic-harmonic-functions}\label{pro_ha}
Let $U\subseteq X$ be open.
\begin{enumerate}
\item[(1)] If $u$ is harmonic in $U$ and $V\subseteq U$ is open,
then $u|_V$ is harmonic in $V$.
\item[(2)] If $u,v$ are harmonic in $U$ and $a,b\in\mathbb{R}$,
then $au+bv$ is harmonic in $U$.
\item[(3)] If $D\in\{D_\lambda\}$ and
$f\in L(\partial D;\omega_\cdot^D)$, then
$\omega_\cdot^D(f)$ is harmonic in $D$.
\end{enumerate}
\end{prop}

\begin{proof}
For (1), let $U'\subseteq U$ be open and let $V\Subset U'$
be a regular domain in the system.  Then $V\Subset U$, so the mean-value identity
for $u$ on $V$ also proves the identity for $u|_{U'}$.
Continuity is preserved by restriction.
For (2), if $u,v$ are harmonic and $a,b\in\mathbb{R}$, then
$au+bv$ is continuous, and for every such $V$ and $x\in V$,
\begin{equation*}
 \begin{aligned}
 \omega_x^V(au+bv)
 &=a\int_{\partial V}u(y)\,d\omega_x^V(y)
   +b\int_{\partial V}v(y)\,d\omega_x^V(y)\\
 &=au(x)+bv(x).
 \end{aligned}
\end{equation*}
The restrictions of $u,v$ to $\partial V$ are bounded, so the displayed integrals
are finite.  This proves closure under real linear combinations.

For (3), continuity of $x\mapsto\omega_x^D(f)$ is part of the
system assumptions.  Proposition \ref{special_compatible} gives
\begin{equation*}
 \omega_x^V\bigl(\omega_\cdot^D(f)\bigr)
 =\int_{\partial V}\omega_y^D(f)\,d\omega_x^V(y)
 =\omega_x^D(f),\qquad x\in V,
\end{equation*}
which is the required harmonicity.

\end{proof}

For a domain $D$ in the system, fix one reference point $x_0\in D$
and keep it fixed in the following kernel formulas. Write
\begin{equation*}
 K_D^{x_0}(x,y):=
 \frac{\d\omega_x^D}{\d\omega_{x_0}^D}(y),
 \qquad x\in D,\quad y\in\partial D.
\end{equation*}
For each fixed $x$, this density is defined
$\omega_{x_0}^D$-almost everywhere; its exceptional set may depend on $x$.
Whenever a continuous version exists, that version will be used.
The identity $K_D^{x_0}(x_0,y)=1$ initially holds almost everywhere;
for regular $D$, full support and continuity make it pointwise.
The next proposition records a useful
sufficient condition under which this \textit{relative kernel} is harmonic in its
first variable.

\begin{prop}\label{prop:relative-kernel-harmonic}\label{Kern_Harm}
Suppose $D\in\{D_\lambda\}$ is regular, $\partial D$ is compact,
and $K_D^{x_0}$ is jointly continuous on $D\times\partial D$.
Then, for each $y\in\partial D$, $K_D^{x_0}(\cdot,y)$ is
nonnegative and harmonic in $D$, and $K_D^{x_0}(x_0,y)=1$.
For every regular domain $V\Subset D$ of the system containing $x_0$,
\begin{equation*}
 K_D^{x_0}(x,y)>0\qquad(\forall x\in V,\ \forall y\in\partial D).
\end{equation*}
In particular, if $D$ satisfies the covering condition
\eqref{eq:local-regular-cover}, then
$K_D^{x_0}(x,y)>0$ for every $(x,y)\in D\times\partial D$.
\end{prop}

\begin{proof}
The continuous versions are nonnegative everywhere: they are nonnegative almost everywhere as
densities, and full support excludes a negative value at any point.
Fix $x_0\in D$, a regular domain $V\Subset D$ in the system, and
$x\in V$.  For every Borel set $A\subseteq\partial D$, condition (N)
for the indicator $\chi_A$ and the Radon--Nikodym identity give
\begin{equation*}
 \begin{aligned}
 \int_A K_D^{x_0}(x,y)\,d\omega_{x_0}^D(y)
 &=\omega_x^D(A)\\
 &=\int_{\partial V}\omega_z^D(A)\,d\omega_x^V(z)\\
 &=\int_{\partial V}\int_A K_D^{x_0}(z,y)\,
       d\omega_{x_0}^D(y)\,d\omega_x^V(z)\\
 &=\int_A\left(\int_{\partial V}K_D^{x_0}(z,y)\,
       d\omega_x^V(z)\right)d\omega_{x_0}^D(y).
 \end{aligned}
\end{equation*}
For the interchange, joint continuity on the compact Hausdorff
product $\partial V\times\partial D$ permits uniform approximation
by finite sums of products of continuous functions on its factors,
by Stone--Weierstrass.  The integrand is therefore measurable for the
product Borel sigma-algebra, and Tonelli's theorem applies.
Uniqueness of Radon--Nikodym derivatives now gives
\begin{equation*}
 K_D^{x_0}(x,y)=\int_{\partial V}K_D^{x_0}(z,y)\,d\omega_x^V(z)
\end{equation*}
for $\omega_{x_0}^D$-almost every $y$.
The right-hand side is continuous in $y$: joint continuity on the
compact product $\partial V\times\partial D$ implies that, as
$y\to y_0$,
\begin{equation*}
 \left|\int_{\partial V}
 [K_D^{x_0}(z,y)-K_D^{x_0}(z,y_0)]\,d\omega_x^V(z)\right|
 \le\sup_{z\in\partial V}|K_D^{x_0}(z,y)-K_D^{x_0}(z,y_0)|\to0.
\end{equation*}
The left-hand side is also continuous.  If the two sides differed at a
boundary point, they would differ on a nonempty relatively open set,
which has positive $\omega_{x_0}^D$-measure by Proposition \ref{H_A}.
This contradicts their almost-everywhere equality.  Thus the identity
holds at every $y\in\partial D$.  Since $V$ and $x$ were
arbitrary, it proves harmonicity of $K_D^{x_0}(\cdot,y)$.

For normalization, the continuous density at $x_0$ equals $1$ almost
everywhere, and full support makes this equality pointwise.
To prove the stated strict positivity, let $V\Subset D$ be regular
and contain $x_0$. Fix $y\in\partial D$ and $x\in V$, and suppose
$K_D^{x_0}(x,y)=0$. The mean-value identity just proved gives
\begin{equation*}
 0=K_D^{x_0}(x,y)
   =\int_{\partial V}K_D^{x_0}(z,y)\,\d\omega_x^V(z).
\end{equation*}
The integrand is nonnegative, so it vanishes almost everywhere for
$\omega_x^V$.  Mutual absolute continuity transfers this equality to
$\omega_{x_0}^V$.  Applying the same identity at $x_0$ yields
\begin{equation*}
 1=K_D^{x_0}(x_0,y)
   =\int_{\partial V}K_D^{x_0}(z,y)\,\d\omega_{x_0}^V(z)=0,
\end{equation*}
a contradiction. Thus the kernel is strictly positive on
$V\times\partial D$.

For the final assertion, fix $y\in\partial D$. If
$K_D^{x_0}(x,y)=0$ at some $x\in D$, then the continuous harmonic
function $-K_D^{x_0}(\cdot,y)$ attains its global maximum $0$ at $x$.
Under \eqref{eq:local-regular-cover}, Theorem \ref{Str_MP_G} gives
\begin{equation*}
 K_D^{x_0}(z,y)=0\quad(z\in D),
\end{equation*}
contradicting $K_D^{x_0}(x_0,y)=1$. Thus the kernel is strictly
positive throughout $D$. This uses the general mean-value form of
that theorem and requires no compatibility for intersecting domains.
\end{proof}

\subsection{Positive harmonic functions}

\begin{defi}\label{def:minimal-harmonic}
A positive harmonic function $h$ on a domain $D$ is \textit{minimal} if
every nonnegative harmonic function $v$ satisfying $v\leq h$ is of
the form $v(x)=ch(x)$ for every $x\in D$, with a constant
$c\in[0,1]$ which may depend on $v$ but not on $x$. If $K_D^{x_0}(\cdot,y)$ is
defined and harmonic for $y\in\partial D$, set
\begin{equation*}
 \partial_{\mathrm m}^{x_0}D
 :=\{y\in\partial D:K_D^{x_0}(\cdot,y)\text{ is minimal}\}.
\end{equation*}
\end{defi}

Let $\mathcal H_+(U)$ denote the \textit{cone of nonnegative harmonic
functions} on $U$.  For a domain $U$ containing $x_0$, put
\begin{equation}\label{eq:normalized-positive-cone}
 \mathcal{H}_1(U):=
 \{h:h>0,\ h\text{ is harmonic in }U,\ h(x_0)=1\},
\end{equation}
equipped with the topology of locally uniform convergence.

\begin{lem}\label{lem:compact-normalized-cone}
Suppose $D=\bigcup_{m\geq1}D_m$, where
\begin{equation*}
 x_0\in D_1,\qquad
 D_m\Subset D_{m+1}\subset D,
\end{equation*}
each $D_m$ is a regular domain in the system, and
$K_{D_m}^{x_0}$ is jointly continuous on
$D_m\times\partial D_m$.  Then $\mathcal{H}_1(D)$ is compact and
metrizable.  Neither metrizability of $\partial D_m$ nor a separate
Harnack bound is assumed.  The local bounds and continuity estimates
in \eqref{eq:derived-kernel-bounds} below also hold for every
nonnegative harmonic function $h$, with $h(x_0)=1$, defined on a
neighborhood of $\overline{D_{m+2}}$, for each $m\geq1$.
The constants $a_m,b_m$ and the function $\Omega_m$ depend on
$m$ and the prescribed kernels, and are independent of $h$.
\end{lem}

\begin{proof}
Put $L_m=\overline{D_m}$ and fix $m$. Compact containment makes
$\partial D_{m+2}$ compact. Apply Proposition
\ref{prop:relative-kernel-harmonic} with $D_{m+2}$ in place of $D$
and $V=D_{m+1}$. Since $x_0\in D_{m+1}$ and
$L_m\subset D_{m+1}$, the kernel is strictly positive on
$L_m\times\partial D_{m+2}$. Joint continuity on this compact
product yields
\begin{equation*}
 a_m:=\min_{\overline{D_m}\times\partial D_{m+2}}
 K_{D_{m+2}}^{x_0}>0,\qquad
 b_m:=\max_{\overline{D_m}\times\partial D_{m+2}}
 K_{D_{m+2}}^{x_0}<\infty.
\end{equation*}
For $s,t\in L_m$, put
\begin{equation*}
 \Omega_m(s,t):=\max_{y\in\partial D_{m+2}}
 |K_{D_{m+2}}^{x_0}(s,y)-K_{D_{m+2}}^{x_0}(t,y)|.
\end{equation*}
This is continuous on $L_m\times L_m$ and vanishes on the diagonal.
Let $h$ be nonnegative and harmonic on a neighborhood of
$\overline{D_{m+2}}$, with $h(x_0)=1$; in particular, $h$ may
belong to $\mathcal H_1(D)$.  The mean-value and Radon--Nikodym
identities give, for every $x\in D_{m+2}$,
\begin{equation}\label{eq:derived-kernel-identity}
 \begin{aligned}
 h(x)&=\int_{\partial D_{m+2}}K_{D_{m+2}}^{x_0}(x,y)h(y)\,
       \d\omega_{x_0}^{D_{m+2}}(y),\\
 1&=\int_{\partial D_{m+2}}h(y)\,\d\omega_{x_0}^{D_{m+2}}(y).
 \end{aligned}
\end{equation}
Consequently $a_m\leq h(x)\leq b_m$ for $x\in L_m$.
For $x,x'\in L_m$, the two identities in
\eqref{eq:derived-kernel-identity} give
\begin{equation*}
\begin{split}
 |h(x)-h(x')|
 &\leq
 \int_{\partial D_{m+2}}h(y)
 \bigl|K_{D_{m+2}}^{x_0}(x,y)
       -K_{D_{m+2}}^{x_0}(x',y)\bigr|
 \,\d\omega_{x_0}^{D_{m+2}}(y)\\
 &\leq
 \sup_{y\in\partial D_{m+2}}
 \bigl|K_{D_{m+2}}^{x_0}(x,y)
       -K_{D_{m+2}}^{x_0}(x',y)\bigr|.
\end{split}
\end{equation*}
We have therefore proved
\begin{equation}\label{eq:derived-kernel-bounds}
 a_m\leq h(s)\leq b_m,\qquad
 |h(s)-h(t)|\leq\Omega_m(s,t)\quad(s,t\in L_m).
\end{equation}
Joint continuity of the kernel and compactness of
$\partial D_{m+2}$ make the right-hand side of the difference estimate tend to zero as
$x'\to x$, uniformly over the functions $h$.  In particular, the restrictions of $\mathcal H_1(D)$
to $L_m$ are equicontinuous and uniformly bounded.
The compact Hausdorff form of the Arzel\`a--Ascoli theorem shows that
their closure in $C(\overline{D_m})$, with the uniform norm, is compact.
Although $\overline{D_m}$ need not be metrizable, the uniform norm
is a metric on this function space.  Its compact subsets are therefore
sequentially compact.  Successive subsequence selections for
$m=1,2,\ldots$, followed by a diagonal extraction, give a subsequence
converging uniformly on every $\overline{D_m}$.  Every compact subset
of $D$ lies in some $D_m$, because the increasing open sets
$D_m$ cover $D$.  The selected subsequence therefore converges
locally uniformly on $D$.  The
limit is harmonic, is normalized at $x_0$, and is strictly positive by
the Harnack lower bounds.  Thus it remains in $\mathcal{H}_1(D)$.
More explicitly, if $h_j\to h$ locally uniformly and $V\Subset D$
is regular, then
\begin{equation*}
 \left|\int_{\partial V}(h_j-h)\,d\omega_x^V\right|
 \le\sup_{\partial V}|h_j-h|\longrightarrow0.
\end{equation*}
Passing to the limit in $h_j(x)=\omega_x^V(h_j)$ proves the required
mean-value identity for $h$.  The normalization follows from
$h_j(x_0)=1$, and the inequalities $h\ge a_m$ on
$\overline{D_m}$ prove strict positivity.
A metric inducing the locally uniform topology is
\begin{equation*}
 d_{\mathcal{H}}(f,g)=\sum_{m=1}^{\infty}2^{-m}
 \min\{1,\|f-g\|_{\overline{D_m}}\}.
\end{equation*}
Indeed, convergence for this metric is equivalent to uniform convergence
on every $\overline{D_m}$, and every compact subset of $D$ is
contained in one of these sets.  The sequential compactness proved above
therefore implies compactness in this metric.
\end{proof}

Lemma \ref{lem:derived-kernel-compactness} applies these bounds in
Section \ref{Sec4} and also places kernels from the varying inner
domains in compatible compact sets.  This additional construction is
needed because $K_{D_n}(\cdot,y)$ is initially defined only on $D_n$.
Neither argument assumes a separate Harnack bound.

\begin{lem}\label{lem:minimal-extreme}
For any $h\in\mathcal{H}_1(D)$,
\begin{equation*}
 h\text{ is minimal}\quad\Longleftrightarrow\quad
 h\in\operatorname{Ext}\mathcal{H}_1(D).
\end{equation*}
\end{lem}

\begin{proof}
Suppose first that $h$ is minimal and
$h=tu+(1-t)v$, where $0<t<1$ and
$u,v\in\mathcal{H}_1(D)$.  Since $0\leq tu\leq h$, minimality gives
$tu=ch$.  Evaluation at $x_0$ gives $c=t$, and hence $u=h$;
similarly $v=h$.

Conversely, suppose $h$ is extreme, $0\le v\le h$ is harmonic,
and $a=v(x_0)\in[0,1]$.  Put $q=v-ah$.  Since $|q|\le h$,
the functions $h\pm q/2$ are strictly positive, harmonic, normalized
at $x_0$, and have midpoint $h$.  Indeed,
\begin{equation*}
 h\pm q/2\ge h/2>0,\qquad
 (h\pm q/2)(x_0)=1,\qquad
 h=\tfrac12(h+q/2)+\tfrac12(h-q/2).
\end{equation*}
Extremality gives $h+q/2=h-q/2=h$, so $q=0$ and $v=ah$.  No separate maximum principle is required for this argument.
\end{proof}

\section{Approximation and representation for harmonic measure systems on metric spaces}\label{Sec4}

Let $X$ be a metric space, let $D\subseteq X$ be a domain,
and fix $x_0\in D$ throughout this section. All normalized kernels
and measures below use this same reference point.  We use a Radon harmonic measure system in
the sense of Definition \ref{def:harmonic-measure-system}; thus its
measures are continuous and satisfy (N).  Harmonicity has the meaning
in Definition \ref{def:system-harmonic}.
Proposition \ref{prop:nested-compatibility} supplies (N) for
integrable data, and the prescribed regular exhaustion is specified
in (A1) below.  No compatibility
condition for arbitrary intersecting domains is imposed.

\begin{ex}[The domain family and the local-to-global property]
\label{ex:framework-local-to-global}
The family of domains in the standing framework need not be a
topological base, and harmonicity defined by its mean-value identities
need not have the local-to-global property.

For example, let $D$ be the unit disk and take only the classical
harmonic measures of the concentric disks $B(0,r)$, $0<r<1$.
These measures satisfy the nested identity (N), and each disk admits
a regular relatively compact exhaustion by smaller concentric disks.
Their positive continuous Poisson kernels verify (A2) for any strictly
increasing radial exhaustion satisfying (A1).  The domain family is
not a topological base.  Choose a continuous nonnegative function
$\varphi$ vanishing on $[0,1/4]$ with $\varphi(1/2)=1$, and set
$u(x)=\varphi(|x|)$.  On $B(0,1/4)$ the function is zero and hence
harmonic.  On the punctured disk it is harmonic in the mean-value sense
of Definition \ref{def:system-harmonic}, since no disk in the chosen
family is compactly contained in that open set.  These two open sets
cover $D$, but
\begin{equation*}
 u(0)=0\ne1=\omega_0^{B(0,1/2)}(u).
\end{equation*}
Thus $u$ is not harmonic on $D$ in that sense.  This distinguishes
the stated framework from the harmonic-sheaf axioms; it does not give a
new representation theorem for classical disk harmonic functions.
\end{ex}

Throughout this section, assertions under (A1)--(A2) include the
Radon harmonic measure system just specified. Strict positivity on
each fixed compact subset for all sufficiently late inner kernels,
and compactness, will be derived from these assumptions.  Boundary data and
concentration conditions are introduced only for the conclusions that
use them, and the common refinement supplies its quotient maps by
construction.

A \textit{normalized measure} on a measurable space $Y$ means a positive Borel measure $\mu$ satisfying $\mu(Y)=1$.  All measure limits below are weak limits of finite measures.
Set
\begin{equation}\label{eq:extreme-boundary}
 \mathcal{E}(D):=\operatorname{Ext}\mathcal{H}_1(D).
\end{equation}
Here $\mathcal{H}_1(D)$ consists of the positive harmonic functions
normalized by $h(x_0)=1$.  A function $h\in\mathcal{H}_1(D)$ is an \textit{extreme point}
if
\begin{equation*}
 h=tu+(1-t)v,\qquad u,v\in\mathcal{H}_1(D),\quad 0<t<1,
\end{equation*}
implies $u=v=h$.  Lemma \ref{lem:minimal-extreme} identifies these
extreme points with the normalized minimal positive harmonic functions.
Thus $\mathcal{E}(D)$ is intrinsic to the harmonic structure.  It is
equipped with the Borel structure inherited from the topology of locally
uniform convergence.

For a regular $U\Subset D$ containing $x_0$, abbreviate
\begin{equation*}
 K_U(x,y):=K_U^{x_0}(x,y)
 =\frac{\d\omega_x^U}{\d\omega_{x_0}^U}(y),
 \qquad x\in U,\quad y\in\partial U.
\end{equation*}
The notation $\partial_{\mathrm m}U$ has the meaning in Definition
\ref{def:minimal-harmonic}.

\subsection{Approximation of refined boundary kernels by inner-boundary kernels}
\label{subsec:refined-kernel-approximation}

We first obtain kernel concentration and synchronized selection at a
prescribed boundary point.  The common refinement retains the two
boundary coordinates.  Selection at a specified refined point uses its
own neighborhood nonvanishing condition; concentration near a projection
fiber alone gives approximation to that fiber.

\subsubsection{Regular exhaustion and derived compactness}

The two standing analytic assumptions are:
\begin{enumerate}
\item[(A1)] There are connected regular domains $D_n$ of the above
family such that
\begin{equation*}
 x_0\in D_1,\qquad
 D_n\Subset D_{n+1},\qquad
 D=\bigcup_{n\geq1}D_n.
\end{equation*}
Here $D_n\Subset D_{n+1}$ means that $\overline{D_n}$ is compact
and contained in $D_{n+1}$.
\item[(A2)] For every $n$, the relative kernel $K_{D_n}$ has a
jointly continuous version on
$D_n\times\partial D_n$.
\end{enumerate}
Condition (A1) already makes $D$ locally compact and separable.
Indeed, each $x\in D$ belongs to some $D_n$ with compact closure
contained in $D$.  Also $D=\bigcup_n\overline{D_n}$, a countable
union of compact metric spaces; the union of countable dense subsets
of these compact spaces is dense in $D$.

Proposition \ref{prop:relative-kernel-harmonic}, applied to each
$D_n$, gives, for $x\in D_n$ and $y\in\partial D_n$,
\begin{equation*}
 K_{D_n}(x,y)\geq0,\qquad K_{D_n}(x_0,y)=1,
 \qquad K_{D_n}(\cdot,y)\text{ is harmonic in }D_n.
\end{equation*}
For $n\geq2$, applying the strict-positivity assertion with
$V=D_{n-1}$ gives
\begin{equation*}
 K_{D_n}(x,y)>0\qquad(x\in D_{n-1},\ y\in\partial D_n).
\end{equation*}
This is sufficient for all local estimates below: for
$L_m=\overline{D_m}$, use $D_{m+2}$ in the mean-value identity
and $L_m\subset D_{m+1}$ for strict positivity of its kernel.
Thus (A1) requires only the one prescribed exhaustion; no regular
inner exhaustion of each $D_n$, or separate positivity assumption
in (A2), is needed.

We distinguish the following optional condition:
\begin{enumerate}
\item[(M)] Every boundary section $K_{D_n}(\cdot,y)$ is minimal; that is,
\begin{equation}\label{eq:inner-minimal-boundary}
 \partial_{\mathrm m}D_n=\partial D_n
 \qquad(n\geq1).
\end{equation}
\end{enumerate}

\begin{lem}[Compactness from the inner kernels]
\label{lem:derived-kernel-compactness}
Assume {\rm(A1)--(A2)} and put $L_m:=\overline{D_m}$.
Then $\mathcal{H}_1(D)$ is compact and metrizable in the topology of
locally uniform convergence.  Moreover, there are nonempty compact convex
sets $\mathcal{Q}_m\subset C(L_m)$, consisting of functions harmonic in
$D_m$, such that:
\begin{enumerate}
\item $g|_{L_m}\in\mathcal{Q}_m$ for every $g\in\mathcal{H}_1(D)$;
\item $K_{D_n}(\cdot,y)|_{L_m}\in\mathcal{Q}_m$ whenever
$n\geq m+3$ and $y\in\partial D_n$;
\item restriction defines continuous maps
$R_{\ell m}:\mathcal{Q}_\ell\to\mathcal{Q}_m$, $\ell>m$, and
\begin{equation}\label{eq:compact-model-inverse-limit}
 \mathcal{H}_1(D)\cong\varprojlim_m\mathcal{Q}_m .
\end{equation}
\end{enumerate}
No compactness of the full sets $\mathcal{H}_1(D_m)$ is assumed.
\end{lem}

\begin{proof}
Lemma \ref{lem:compact-normalized-cone} gives compactness of
$\mathcal H_1(D)$ and the quantities $a_m,b_m,\Omega_m$ in
\eqref{eq:derived-kernel-bounds}.  On the compact metric set $L_m$,
$\Omega_m(s,t)\to0$ uniformly as $d(s,t)\to0$.
Those estimates apply both to $g\in\mathcal H_1(D)$ and to
$g=K_{D_n}(\cdot,y)$ when $n\geq m+3$, since in the latter
case $g$ is nonnegative and harmonic on a neighborhood of
$\overline{D_{m+2}}$ and $g(x_0)=1$.

Define $\mathcal{Q}_m$ to consist of the functions $g\in C(L_m)$
which are harmonic in $D_m$, satisfy $g(x_0)=1$, and satisfy
\eqref{eq:derived-kernel-bounds} at every level $i\leq m$, with $m$
replaced by $i$ and $s,t\in L_i$.  These conditions are closed under uniform convergence on $L_m$.
Indeed, if $g_j\to g$ uniformly there, normalization and the bounds
pass to the limit pointwise.  For a regular $V\Subset D_m$,
\begin{equation*}
 \left|\int_{\partial V}(g_j-g)\,d\omega_x^V\right|
 \le\|g_j-g\|_{L_m}\longrightarrow0,
\end{equation*}
so $g_j(x)=\omega_x^V(g_j)$ implies $g(x)=\omega_x^V(g)$.
The modulus inequality passes to the limit for every pair $s,t$.
They are convex.  The estimates at level $m$ give uniform boundedness
and equicontinuity on $L_m$.  Arzel\`a--Ascoli therefore makes
$\mathcal{Q}_m$ compact in $C(L_m)$.
The constant function $1$ belongs to every $\mathcal{Q}_m$, and the
preceding estimates prove (1) and (2).

The definition at all levels $i\leq m$ ensures that restriction maps
$\mathcal{Q}_\ell$ into $\mathcal{Q}_m$.  These maps are continuous.
A compatible family $(g_m)$ patches to a unique continuous function
$g$ on $D$.  It is positive by the bounds $a_m>0$, normalized at
$x_0$, and harmonic because every regular $V\Subset D$ has compact
closure in some $D_m$.  Conversely, (1) places the restrictions of any
$g\in\mathcal{H}_1(D)$ in this inverse limit.  The resulting bijection
is a homeomorphism: every compact subset of $D$ lies in some $L_m$,
so uniform convergence on all $L_m$ is exactly local uniform
convergence.  The inverse limit is a closed subset of the countable
product of the compact metric spaces $\mathcal{Q}_m$.  This proves
\eqref{eq:compact-model-inverse-limit} and the compactness assertion.
\end{proof}

The \textit{extreme boundary} $\mathcal{E}(D)$ is Borel in $\mathcal{H}_1(D)$.
Indeed, for a compatible metric $d_{\mathcal{H}}$, each set
\begin{equation*}
 \left\{\frac{g_1+g_2}{2}:
 g_1,g_2\in\mathcal{H}_1(D),\
 d_{\mathcal{H}}(g_1,g_2)\geq 1/j\right\}
\end{equation*}
is compact.  Their countable union is exactly the set of nonextreme
points.  Thus $\mathcal{E}(D)$ is a $G_\delta$ subset of a compact
metric space; here $G_\delta$ means a countable intersection of open sets.

\subsubsection{Common refinement of two compactifications}
\label{subsec:common-refinement}

We first construct the \textit{common refinement} and its \textit{coordinate projections}.
These topological facts require no harmonic-kernel assumptions.

Let $D^G$ and $D^K$ be two compact metrizable compactifications of
the same locally compact separable metric space $D$.  Define
\begin{equation}\label{eq:common-refinement}
 \widehat D=D^G\vee D^K
 :=\overline{\{(x,x):x\in D\}}^{\,D^G\times D^K}.
\end{equation}
We identify $D$ with its diagonal image and denote the coordinate
projections by
\begin{equation*}
 p_G:\widehat D\longrightarrow D^G,\qquad
 p_K:\widehat D\longrightarrow D^K.
\end{equation*}
Write $\widehat\partial D:=\widehat D\setminus D$,
$\partial^GD:=D^G\setminus D$, and $\partial^KD:=D^K\setminus D$.

\begin{prop}[Canonical common refinement]
\label{prop:common-refinement-projections}
The space $\widehat D$ is a compact metrizable compactification of $D$.
The maps $p_G$ and $p_K$ are continuous surjections, equal the identity
on $D$, and satisfy
\begin{equation*}
 p_G^{-1}(D)=p_K^{-1}(D)=D.
\end{equation*}
Both maps and their boundary restrictions
\begin{equation*}
 p_G:\widehat\partial D\longrightarrow\partial^GD,\qquad
 p_K:\widehat\partial D\longrightarrow\partial^KD
\end{equation*}
are quotient maps.  For a sequence $(x_n)$ in $D$ and
$\widehat\xi=(\xi_G,\xi_K)\in\widehat D$,
\begin{equation*}
 x_n\longrightarrow\widehat\xi\text{ in }\widehat D
 \quad\Longleftrightarrow\quad
 \begin{cases}
 x_n\longrightarrow\xi_G&\text{in }D^G,\\
 x_n\longrightarrow\xi_K&\text{in }D^K.
 \end{cases}
\end{equation*}
If a compact Hausdorff compactification $L$ of $D$ has continuous
maps $q_G:L\to D^G$ and $q_K:L\to D^K$, both equal to the identity
on $D$, then
\begin{equation*}
 Q:L\longrightarrow\widehat D,\qquad
 Q(z):=(q_G(z),q_K(z))
\end{equation*}
is a continuous surjection satisfying
$p_G\circ Q=q_G$ and $p_K\circ Q=q_K$.
\end{prop}

\begin{proof}
The diagonal map $x\mapsto(x,x)$ is continuous and injective.
Its inverse on its image is the restriction of either coordinate
projection, so it is an embedding.  The product $D^G\times D^K$
is compact and metrizable.  Its closed subset $\widehat D$ is
therefore compact and metrizable, with the diagonal copy of $D$
dense in it.  The images $p_G(\widehat D)$ and $p_K(\widehat D)$
are compact and contain $D$, so they equal $D^G$ and $D^K$,
respectively.

Suppose $p_G(z)=x\in D$.  Choose diagonal points $x_n\to z$
in $\widehat D$.  Then $x_n\to x$ in $D^G$, and hence in the
original topology of $D$.  Thus $x_n\to x$ also in $D^K$, and
$z=(x,x)$.  This proves $p_G^{-1}(D)=D$; interchanging the two
coordinates proves $p_K^{-1}(D)=D$.

For completeness, $D$ is open in any Hausdorff compactification $L$.
Given $x\in D$, choose an open neighborhood $U$ of $x$ in $D$
whose closure $C$ is compact and contained in $D$.  Write
$U=O\cap D$ with $O$ open in $L$.  Density of $D$ gives
$O\subseteq\overline{O\cap D}^{\,L}\subseteq C$, since $C$
is closed in $L$.  Hence $x\in O\subseteq D$.
The three boundaries above are consequently compact.  Their restricted
projections are onto by the identities for the preimages of $D$.
All the asserted quotient maps are continuous surjections from compact
spaces to Hausdorff spaces.  The convergence assertion follows from
coordinatewise convergence in a product.

Finally, $Q=(q_G,q_K)$ is continuous into $D^G\times D^K$.
Its image is compact and contains the diagonal copy of $D$.
Density of $D$ in $L$ and continuity give
$Q(L)=\overline{Q(D)}=\widehat D$.  This proves surjectivity and
the two projection identities.
\end{proof}

In particular, define equivalence relations on $\widehat\partial D$ by
\begin{equation*}
 \widehat\xi\sim_G\widehat\eta
 \ \Longleftrightarrow\ p_G(\widehat\xi)=p_G(\widehat\eta),
 \qquad
 \widehat\xi\sim_K\widehat\eta
 \ \Longleftrightarrow\ p_K(\widehat\xi)=p_K(\widehat\eta).
\end{equation*}
The boundary $\partial^GD$ is recovered by identifying the points
in each $p_G$-fiber, and $\partial^KD$ is recovered by identifying
the points in each $p_K$-fiber.  The two quotient spaces, with their
quotient topologies, are homeomorphic to the respective boundaries.
Thus the common refinement retains both kinds of boundary information.
The final assertion of the proposition also makes it the smallest
compactification dominating both $D^G$ and $D^K$.

\begin{rem}[More than two compactifications]
\label{rem:multiple-common-refinement}
For a nonempty finite or countable family $(D^{(i)})_{i\in I}$ of
compact metrizable compactifications of $D$, the same construction gives
\begin{equation}\label{eq:abstract-common-refinement}
 D^{\vee I}:=
 \overline{\{(x)_{i\in I}:x\in D\}}^{\,\prod_{i\in I}D^{(i)}}.
\end{equation}
The countable product is compact and metrizable.  The proof above
applies to each coordinate projection and gives the corresponding
boundary quotient maps and universal property.  Only the two-space
construction \eqref{eq:common-refinement} is needed below.
\end{rem}

\subsubsection{Boundary data and the approximation theorem}

For the approximation of one extreme point, only the following boundary
data are needed:
\begin{enumerate}
\item[$(K_\xi)$] A compact metrizable compactification $D^K$ of
$D$, a point $\xi\in D^K\setminus D$, and a function
$h_\xi\in\mathcal{E}(D)$.
\end{enumerate}
Write $h:=h_\xi$; when convenient, $K_D(\cdot,\xi)$ denotes this
single function.  The superscript $K$ labels the chosen compactification;
it does not denote a power.  No parametrization of all of $\mathcal{E}(D)$ is
required here.

By the mean-value property and $h(x_0)=1$,
\begin{equation}\label{eq:nu-n}
 \d\nu_n^\xi(y):=h(y)\,\d\omega_{x_0}^{D_n}(y)
\end{equation}
is a normalized measure on $\partial D_n$, viewed as a measure on
$D^K$ through $\partial D_n\subset D$.
The neighborhoods in the following conditions are open subsets of the
indicated compactification containing the specified point or fiber.
We impose the following \textit{nonvanishing condition}:
\begin{enumerate}
\item[$(C_\xi)$] For every open neighborhood $V$ of $\xi$ in $D^K$,
\begin{equation}\label{eq:neighborhood-nonvanishing}
 \liminf_{n\to\infty}\nu_n^\xi(V)>0 .
\end{equation}
\end{enumerate}
Equivalently, for each $V$ there are $c_V>0$ and $N_V$ such that
$\nu_n^\xi(V)\ge c_V$ for all $n\ge N_V$.
Both constants may depend on $V$.  It suffices to check
\eqref{eq:neighborhood-nonvanishing} on one countable neighborhood base.
The stronger condition
\begin{equation}\label{eq:weak-concentration}
 \nu_n^\xi(D^K\setminus V)\longrightarrow0
 \quad(n\to\infty)
\end{equation}
for every open neighborhood $V$ of $\xi$ in $D^K$ implies
$(C_\xi)$, because $\nu_n^\xi(D^K)=1$ gives
$\nu_n^\xi(V)=1-\nu_n^\xi(D^K\setminus V)\to1$.
It also suffices to check \eqref{eq:weak-concentration} on one
countable neighborhood base.

For use in the proofs, \eqref{eq:weak-concentration} is equivalent to
$\nu_n^\xi\Longrightarrow\delta_\xi$ on $D^K$.
Indeed, if \eqref{eq:weak-concentration} holds, then, for
$F\in C(D^K)$ and $\varepsilon>0$, choose $V$ so that
$|F(y)-F(\xi)|<\varepsilon$ on $V$.  Normalization gives
\begin{equation*}
 \left|\int_{D^K}F\,\d\nu_n^\xi-F(\xi)\right|
 \le\varepsilon+2\|F\|_\infty\nu_n^\xi(D^K\setminus V).
\end{equation*}
Let $n\to\infty$ and then $\varepsilon\downarrow0$.
Conversely, suppose $\nu_n^\xi\Longrightarrow\delta_\xi$.
Fix a compatible metric $d_K$ and an open neighborhood $V$ of $\xi$.
Choose $r>0$ with $B_{d_K}(\xi,r)\subseteq V$, and put
$\varphi(y)=\min\{1,d_K(y,\xi)/r\}$.
Then $\varphi\in C(D^K)$, $\varphi(\xi)=0$, and
$\varphi=1$ on $D^K\setminus V$.  Hence
\begin{equation*}
 0\le\nu_n^\xi(D^K\setminus V)
 \le\int_{D^K}\varphi\,\d\nu_n^\xi\longrightarrow0.
\end{equation*}
The same argument applies to any fixed subsequence.
Condition $(C_\xi)$ connects the prescribed compactification topology
with the harmonic measures; it is not implied by metricity.

\begin{thm}[Synchronized point and kernel approximation]
\label{thm:synchronized-kernel}
Assume {\rm(A1)--(A2)}, $(K_\xi)$, and $(C_\xi)$.
Then there exist $n_0\geq1$ and $\xi_n\in\partial D_n$, $n\geq n_0$,
such that
\begin{align}
 \xi_n&\longrightarrow\xi
 &&\text{in }D^K,\label{eq:point-convergence}\\
 K_{D_n}(\cdot,\xi_n)&\longrightarrow h_\xi
 &&\text{locally uniformly in }D.\label{eq:kernel-convergence}
\end{align}
If {\rm(M)} also holds, then
$\xi_n\in\partial_{\mathrm m}D_n=\partial D_n$ for every $n\geq n_0$.
The sequence may depend on $(\xi,h_\xi)$. For this single selected
sequence, \eqref{eq:kernel-convergence} means that, for every compact
$L\subset D$ and every $\varepsilon>0$, there is $N\geq n_0$ such that
\begin{equation*}
 L\subset D_n,\qquad
 \sup_{x\in L}|K_{D_n}(x,\xi_n)-h_\xi(x)|<\varepsilon
 \quad(\forall n\geq N).
\end{equation*}
\end{thm}

\begin{ex}[Failure of synchronized approximation without boundary compatibility]
\label{ex:failure-boundary-compatibility}
Conditions {\rm(A1)--(A2)}, {\rm(M)}, and $(K_\xi)$ do not by
themselves imply the conclusion of Theorem \ref{thm:synchronized-kernel}.
This remains true when the boundary points are in Borel bijection with
the normalized minimal functions, a stronger requirement than (B) below.
\end{ex}

\begin{proof}
Let $D=(0,1)$, $x_0=1/2$, $D^K=[0,1]$, and
$D_n=(\varepsilon_n,1-\varepsilon_n)$, where
$0<\varepsilon_n\downarrow0$ strictly and $\varepsilon_1<1/2$.
Use the usual affine harmonic functions on intervals. The Dirichlet
solution is affine interpolation of the endpoint values, so
\begin{equation*}
 \omega_x^{D_n}
 =\frac{1-\varepsilon_n-x}{1-2\varepsilon_n}\delta_{\varepsilon_n}
  +\frac{x-\varepsilon_n}{1-2\varepsilon_n}\delta_{1-\varepsilon_n},
 \qquad x\in D_n.
\end{equation*}
These measures are normalized, have common null sets for interior
points, and satisfy (N) because affine functions satisfy the same
interpolation identity on every smaller interval. Both endpoints are
regular. Thus (A1) holds. At $x_0$ both coefficients are $1/2$, and hence
\begin{equation*}
 \begin{aligned}
 K_{D_n}(x,\varepsilon_n)
 &=\frac{2(1-\varepsilon_n-x)}{1-2\varepsilon_n},\\
 K_{D_n}(x,1-\varepsilon_n)
 &=\frac{2(x-\varepsilon_n)}{1-2\varepsilon_n}.
 \end{aligned}
\end{equation*}
These kernels are positive and jointly continuous. If an affine
function $v$ satisfies $0\leq v\leq K_{D_n}(\cdot,\varepsilon_n)$,
then its value at $1-\varepsilon_n$ is zero. It is therefore a
nonnegative multiple of $K_{D_n}(\cdot,\varepsilon_n)$. The other
endpoint is treated in the same way, proving (A2) and (M).
On $D$ the same argument identifies the normalized minimal functions:
\begin{equation*}
 \mathcal E(D)=\{x\mapsto2(1-x),\ x\mapsto2x\}.
\end{equation*}
Set $\xi=1$, $h_\xi(x)=2(1-x)$, and assign the boundary labels by
\begin{equation*}
 \kappa(1)(x)=2(1-x),\qquad \kappa(0)(x)=2x.
\end{equation*}
This gives $(K_\xi)$ and (B), even with a Borel bijection.

If $y_n\in\partial D_n$ and $y_n\to1$, then
$y_n=1-\varepsilon_n$ for all sufficiently large $n$. The displayed
kernel formula gives
\begin{equation*}
 K_{D_n}(\cdot,y_n)\longrightarrow2x
 \quad\text{locally uniformly in }(0,1),
\end{equation*}
whereas $h_\xi(x)=2(1-x)$. The two functions differ, for example at
$x=1/4$. Thus synchronized approximation at $\xi$ fails.
On the other hand,
\begin{equation*}
 K_{D_n}(\cdot,\varepsilon_n)\longrightarrow2(1-x)=h_\xi(x)
 \quad\text{locally uniformly in }(0,1),
\end{equation*}
so kernel approximation alone holds.

The corresponding weighted measures are
\begin{equation*}
 \nu_n^1=(1-\varepsilon_n)\delta_{\varepsilon_n}
          +\varepsilon_n\delta_{1-\varepsilon_n}
 \Longrightarrow\delta_0.
\end{equation*}
For the open neighborhood $V=(3/4,1]$ of $1$ in $[0,1]$,
$\nu_n^1(V)=\varepsilon_n\to0$. Thus $(C_\xi)$ fails.
\end{proof}

This example shows that the analytic assumptions and the single-point
data alone are insufficient: some condition linking the prescribed
boundary topology to the target harmonic function is needed for a general
synchronized approximation theorem.  It does not show that $(C_\xi)$
is necessary, or that it is the weakest possible sufficient condition.
The labels used here differ from the usual endpoint correspondence.
The split-boundary construction in Subsection \ref{subsec:split-disk} shows that $(C_\xi)$ can hold
without weak convergence of $\nu_n^\xi$ to $\delta_\xi$.

\subsubsection{Exact representation and concentration in kernel space}

The following measure construction uses only {\rm(A1)--(A2)}.
For a minimal function it gives kernel concentration.  The boundary
point and condition $(C_\xi)$ enter only in the subsequent selection
step.

\begin{lem}[Common representing measures]\label{lem:common-measures}
Assume {\rm(A1)--(A2)}.  For $u\in\mathcal H_+(D)$, define
\begin{equation}\label{eq:inner-boundary-measures}
 \d\lambda_n^u(y):=u(y)\,\d\omega_{x_0}^{D_n}(y),
 \qquad y\in\partial D_n.
\end{equation}
Then, for every integer $n\geq1$,
\begin{equation}\label{eq:inner-exact-representation}
 \lambda_n^u(\partial D_n)=u(x_0),\qquad
 u(x)=\int_{\partial D_n}K_{D_n}(x,y)\,\d\lambda_n^u(y)
 \quad(\forall x\in D_n).
\end{equation}
In particular, for the normalized minimal function $h=h_\xi$,
$\lambda_n^h=\nu_n^\xi$ and $\nu_n^\xi(\partial D_n)=1$.
\end{lem}

\begin{proof}
Since $D_n\Subset D$, the restriction of $u$ to $\partial D_n$
is continuous and bounded.  Harmonicity on $D$ and the
Radon--Nikodym identity give, for $x\in D_n$,
\begin{equation*}
 \begin{aligned}
 u(x)&=\omega_x^{D_n}(u)
 =\int_{\partial D_n}u(y)\,\d\omega_x^{D_n}(y)\\
 &=\int_{\partial D_n}K_{D_n}(x,y)u(y)\,
                          \d\omega_{x_0}^{D_n}(y)
 =\int_{\partial D_n}K_{D_n}(x,y)\,\d\lambda_n^u(y).
 \end{aligned}
\end{equation*}
At $x_0$, the mean-value identity also gives
$\lambda_n^u(\partial D_n)=\omega_{x_0}^{D_n}(u)=u(x_0)$.
Taking $u=h$ proves the final assertion.  Thus a single measure
represents all the values on each inner domain.
\end{proof}

We use one compact product for kernel approximation and representation.
Write $H:=\mathcal H_1(D)$ and
$P:=\prod_{m\geq1}\mathcal Q_m$, with the compact sets of Lemma
\ref{lem:derived-kernel-compactness}.  Identify $H$ with the closed
subset of compatible sequences in $P$.  Define the continuous maps
\begin{equation}\label{eq:full-inner-kernel-map}
 T_n:\partial D_n\longrightarrow P,\qquad
 T_n(y)_m:=
 \begin{cases}
 K_{D_n}(\cdot,y)|_{L_m},&m\leq n-3,\\
 1|_{L_m},&m>n-3.
 \end{cases}
\end{equation}
Each fixed coordinate is the actual inner kernel for all sufficiently
large $n$.  Let $\widehat D$ be a compact metrizable compactification
of $D$; no boundary parametrization is needed at this stage.  For
$u\in\mathcal H_+(D)$, use \eqref{eq:inner-boundary-measures} and put
\begin{equation}\label{eq:joint-inner-measures}
 \rho_n^u:=(T_n)_\#\lambda_n^u,\qquad
 \Lambda_n^u:=(\operatorname{id},T_n)_\#\lambda_n^u.
\end{equation}
Thus $\rho_n^u$ is a measure on $P$, and $\Lambda_n^u$ is a
measure on $\widehat D\times P$ recording position and kernel together.

\begin{lem}[Joint limits of inner-boundary measures]
\label{lem:joint-inner-limits}
Assume {\rm(A1)--(A2)} and fix $u\in\mathcal H_+(D)$.
The sequence $(\Lambda_n^u)$ has weakly convergent subsequences on
$\widehat D\times P$.  Every limit $\Lambda$ is carried by
$(\widehat D\setminus D)\times H$ and satisfies, for every $x\in D$,
\begin{equation}\label{eq:joint-limit-representation}
 \Lambda(\widehat D\times P)=u(x_0),\qquad
 u(x)=\int_{(\widehat D\setminus D)\times H}
          g(x)\,\d\Lambda(\widehat\eta,g).
\end{equation}
Every weak subsequential limit $\sigma$ of $(\rho_n^u)$ is carried
by $H$ and satisfies
\begin{equation*}
 u(x)=\int_Hg(x)\,\d\sigma(g),\qquad x\in D.
\end{equation*}
For every finite positive measure $\tau$ on $H$ and Borel
$A\subseteq H$, the function
\begin{equation*}
 u_A(x):=\int_A g(x)\,\d\tau(g)\qquad(x\in D)
\end{equation*}
is continuous and harmonic, and $u_A(x_0)=\tau(A)$.
It is strictly positive if $\tau(A)>0$ and is zero otherwise.
\end{lem}

\begin{proof}
Lemma \ref{lem:common-measures} gives
$\Lambda_n^u(\widehat D\times P)=u(x_0)$ and the exact identity
\eqref{eq:inner-exact-representation}.
Positive measures with this fixed value on the compact metric space
$\widehat D\times P$ form a compact metrizable space in the weak
topology.  Hence subsequences with a weak limit exist.
For fixed $m$, the position coordinates eventually belong to the
closed set $\widehat D\setminus D_m$.  For fixed $\ell>m$, the
kernel coordinates eventually satisfy the closed condition
$R_{\ell m}g_\ell=g_m$.  These conditions pass to the limit, and
their countable intersection shows that $\Lambda$ is carried by
$(\widehat D\setminus D)\times H$.
Integration of $1$ gives the first identity in
\eqref{eq:joint-limit-representation}.  For $x\in L_m$, the map
$(\widehat\eta,(g_i))\mapsto g_m(x)$ is continuous and bounded
on the compact product.  Its integral against $\Lambda_n^u$ is
$u(x)$ once $n\geq m+3$.  Passing to the limit proves the second
identity.  If $\rho_{n_j}^u\Longrightarrow\sigma$, take a further
subsequence along which the joint measures converge.  Their function
projection is $\sigma$, giving the representation by $\sigma$.

Finally, let $\tau$ be a
finite positive measure on $H$ and let $A\subseteq H$ be Borel.  Then
\begin{equation*}
 u_A(x):=\int_Ag(x)\,\d\tau(g)
\end{equation*}
is finite and continuous by the local bounds and moduli of Lemma
\ref{lem:derived-kernel-compactness}.  For regular $V\Subset D$,
Tonelli's theorem gives
\begin{equation*}
 \int_{\partial V}u_A(y)\,\d\omega_x^V(y)
 =\int_A\left(\int_{\partial V}g(y)\,\d\omega_x^V(y)\right)
                \d\tau(g)=u_A(x).
\end{equation*}
Thus $u_A\in\mathcal H_+(D)$ and $u_A(x_0)=\tau(A)$.
The bounds $a_m>0$ also prove strict positivity when $\tau(A)>0$.
\end{proof}

For a fixed $h\in\mathcal E(D)$, let $\pi_m:P\to\mathcal Q_m$
be the $m$th coordinate projection and put
\begin{equation}\label{eq:rho-pushforward}
 \rho_n^{(m)}:=(\pi_m)_\#\rho_n^h.
\end{equation}
For $n\geq m+3$, this is the push-forward of
$\nu_n^\xi=h\,\omega_{x_0}^{D_n}$ under
$y\mapsto K_{D_n}(\cdot,y)|_{L_m}$, by
\eqref{eq:full-inner-kernel-map}. Thus the local kernel measures
are coordinate projections of the single measure already constructed
on $P$.

\begin{lem}[Kernel-space concentration]\label{lem:kernel-space-concentration}
Assume {\rm(A1)--(A2)} and fix $h\in\mathcal E(D)$, with
$\nu_n^\xi=h\,\omega_{x_0}^{D_n}$ as above.
For every fixed integer $m\geq1$,
\begin{equation}\label{eq:kernel-space-convergence}
 \rho_n^{(m)}\Longrightarrow\delta_{h|_{L_m}}
 \quad\text{weakly on }\mathcal{Q}_m .
\end{equation}
Moreover, $\rho_n^h\Longrightarrow\delta_h$ on $P$.
No boundary concentration condition is needed; the superscript $\xi$
only records the function $h$ used to define these measures.
\end{lem}

\begin{proof}
The restriction $h|_{L_m}$ need not be extreme in $\mathcal Q_m$.
Use instead the measures $\rho_n^h=(T_n)_\#\nu_n^\xi$ on $P$.
Every subsequence has a further weakly convergent subsequence, and
Lemma \ref{lem:joint-inner-limits} shows that its limit $\rho$ is
carried by $H=\mathcal H_1(D)$, satisfies $\rho(H)=1$, and gives
\begin{equation}\label{eq:global-integral-identity}
 \int_H g(z)\,\d\rho(g)=h(z),\qquad z\in D.
\end{equation}
For every Borel set $A\subseteq H$, put
\begin{equation*}
 v_A(x):=\int_Ag(x)\,\d\rho(g).
\end{equation*}
By Lemma \ref{lem:joint-inner-limits}, $v_A$ is harmonic,
$v_A(x_0)=\rho(A)$, and \eqref{eq:global-integral-identity} gives
$0\leq v_A\leq h$.  Minimality of $h$ implies $v_A=c_Ah$.
Evaluation at $x_0$ gives $c_A=\rho(A)$, so
\begin{equation*}
 \int_A[g(x)-h(x)]\,\d\rho(g)=0
 \qquad(x\in D,\ A\subseteq H\text{ Borel}).
\end{equation*}
For fixed $x$, apply this identity to
$\{g:g(x)>h(x)\}$ and $\{g:g(x)<h(x)\}$.
The respective nonnegative differences have zero integral; hence
$g(x)=h(x)$ for $\rho$-almost every $g$.
Choose a countable dense set $\{q_r:r\geq1\}\subset D$,
which exists by (A1).  Outside a single null set the equality holds
at every $q_r$, and continuity then gives $g=h$ throughout $D$.
Thus $\rho=\delta_h$.
Every subsequential limit on $P$ is the same, so compactness gives
$\rho_n^h\Longrightarrow\delta_h$ on $P$.
For $n\geq m+3$, its $m$th marginal is $\rho_n^{(m)}$.
Continuity of coordinate projection proves
\eqref{eq:kernel-space-convergence}.
\end{proof}

\begin{cor}[Preliminary kernel approximation]
\label{cor:preliminary-kernel-approximation}
Under the hypotheses of Lemma \ref{lem:kernel-space-concentration},
for every nonempty $L\Subset D$ and every $\varepsilon>0$,
\begin{equation*}
 \lim_{n\to\infty}\nu_n^\xi\left(
 \left\{y\in\partial D_n:
 \sup_{z\in L}|K_{D_n}(z,y)-h(z)|<\varepsilon\right\}\right)=1.
\end{equation*}
The sets in this limit are considered only for the indices with
$\overline L\subset D_n$, which include every sufficiently large $n$.
In particular, there is $n_0=n_0(L,\varepsilon,h)$ such that, for
every $n\geq n_0$, one has $\overline L\subset D_n$ and
\begin{equation*}
 \left\{y\in\partial D_n:
 \sup_{z\in L}|K_{D_n}(z,y)-h(z)|<\varepsilon\right\}\ne\varnothing.
\end{equation*}
\end{cor}

\begin{proof}
Choose $m$ with $L\subseteq L_m$, and put
\begin{equation*}
 U=\{g\in\mathcal{Q}_m:\sup_{z\in L}|g(z)-h(z)|<\varepsilon\}.
\end{equation*}
The map $g\mapsto\sup_L|g-h|$ is continuous in the uniform norm
on $L_m$, because the difference of its values at $g_1,g_2$ is
at most $\|g_1-g_2\|_{L_m}$.  Thus $U$ is open and contains
$h|_{L_m}$.  Lemma \ref{lem:kernel-space-concentration} and
Portmanteau give
\begin{equation*}
 1\ge\limsup_n\rho_n^{(m)}(U)
 \ge\liminf_n\rho_n^{(m)}(U)
 \ge\delta_{h|_{L_m}}(U)=1.
\end{equation*}
For $n\ge m+3$, the push-forward definition identifies
\begin{equation*}
 \rho_n^{(m)}(U)=\nu_n^\xi\left(
 \left\{y\in\partial D_n:\sup_{z\in L}|K_{D_n}(z,y)-h(z)|
 <\varepsilon\right\}\right).
\end{equation*}
This value tends to one, so it is positive, and the set is nonempty,
for every sufficiently large $n$.
\end{proof}

\subsubsection{Synchronized approximation on the common refinement}

\begin{proof}[Proof of Theorem \ref{thm:synchronized-kernel}]
Fix a compatible metric $d_K$ on $D^K$, put
$V_m:=\{z\in D^K:d_K(z,\xi)<1/m\}$, and let
\begin{equation*}
 A_{n,m}:=\left\{y\in\partial D_n:
 \sup_{z\in L_m}|K_{D_n}(z,y)-h(z)|<1/m\right\}.
\end{equation*}
For each fixed $m$, Corollary
\ref{cor:preliminary-kernel-approximation} gives
\begin{equation}\label{eq:good-kernel-measure}
 \nu_n^\xi(A_{n,m})\longrightarrow1 .
\end{equation}
Condition $(C_\xi)$ allows a choice of $c_m>0$ such that
$\nu_n^\xi(V_m)>c_m$ for all sufficiently large $n$.
For each $m$, choose an integer threshold beyond which both
inequalities below hold.  Inductively enlarge it to an integer $N_m$
with $N_m\ge m+3$ and $N_m>N_{m-1}$.  Then
\begin{equation*}
 \nu_n^\xi(V_m)>c_m,\qquad
 \nu_n^\xi(A_{n,m}^c)<c_m/2
 \qquad(n\geq N_m).
\end{equation*}
Then
\begin{equation*}
 \nu_n^\xi(A_{n,m}\cap V_m)
 \geq\nu_n^\xi(V_m)-\nu_n^\xi(A_{n,m}^c)>c_m/2 .
\end{equation*}
For $N_m\leq n<N_{m+1}$, select
\begin{equation}\label{eq:synchronized-choice}
 \xi_n\in A_{n,m}\cap V_m\cap\partial D_n .
\end{equation}
The block index $m$ tends to infinity, so $d_K(\xi_n,\xi)<1/m$
proves \eqref{eq:point-convergence}.  Every compact $L\Subset D$
is contained in $L_m$ for all large $m$; therefore
\begin{equation*}
 \sup_{z\in L}|K_{D_n}(z,\xi_n)-h(z)|<1/m
\end{equation*}
on those blocks.  This proves \eqref{eq:kernel-convergence}.
Condition (M), when imposed, identifies every selected point as minimal.
\end{proof}

For the following corollaries, choose a second compact metrizable
compactification $D^G$ of $D$ and form the common refinement
$\widehat D=D^G\vee D^K$ in \eqref{eq:common-refinement}.
Proposition \ref{prop:common-refinement-projections} gives the
continuous surjection
\begin{equation}\label{eq:ambient-kernel-projection}
 p_K:\widehat D\longrightarrow D^K,\qquad p_K|_D=\operatorname{id}_D,
 \qquad p_K^{-1}(D)=D.
\end{equation}
These properties follow from the construction and are not additional
hypotheses on the harmonic measures.  Taking $D^G=D^K$ gives the
identity refinement.  More generally, a prescribed compact metrizable
refinement $L$ of $D^K$ is recovered, up to the natural homeomorphism,
by taking $D^G=L$: the diagonal closure is the graph of its projection
to $D^K$.

\begin{cor}[Passage to a compact refinement]
\label{cor:refinement-approximation}
Under the hypotheses of Theorem \ref{thm:synchronized-kernel}, use the
common refinement \eqref{eq:common-refinement} and put
\begin{equation}\label{eq:kernel-fiber}
 F_\xi:=p_K^{-1}(\{\xi\}).
\end{equation}
Every sequence selected in that theorem satisfies
\begin{equation}\label{eq:fiber-distance}
 \operatorname{dist}_{\widehat D}(\xi_n,F_\xi)\longrightarrow0 .
\end{equation}
It has a subsequence converging in $\widehat D$ to a point of $F_\xi$.
If $F_\xi=\{\widehat\xi\}$, then
\begin{equation*}
 \xi_n\longrightarrow\widehat\xi\quad\text{in }\widehat D.
\end{equation*}
\end{cor}

\begin{proof}
Since $p_K$ is continuous and onto, $F_\xi$ is a nonempty
closed subset of the compact metric space $\widehat D$.
Suppose the distance assertion fails.  Then there are
$\varepsilon>0$ and a subsequence $(\xi_{n_j})$ such that
$\operatorname{dist}_{\widehat D}(\xi_{n_j},F_\xi)\ge\varepsilon$
for every $j$.  Compactness gives a further subsequence converging
to some $z\in\widehat D$.  Continuity and the already proved point
convergence imply
\begin{equation*}
 p_K(z)=\lim_jp_K(\xi_{n_j})=\lim_j\xi_{n_j}=\xi.
\end{equation*}
Hence $z\in F_\xi$, whereas
$\operatorname{dist}(\xi_{n_j},F_\xi)\le d(\xi_{n_j},z)\to0$.
This contradiction proves \eqref{eq:fiber-distance}.

Compactness also ensures that the original sequence has a convergent
subsequence, and the same projection calculation puts every such limit
in $F_\xi$.  If $F_\xi=\{z\}$, failure of convergence to $z$
would give a subsequence outside some ball about $z$.  A convergent
further subsequence would then have a limit different from $z$, which
is impossible.  Thus the full sequence converges in the singleton case.
\end{proof}

The measures in \eqref{eq:nu-n} may also be regarded as measures on
$\widehat D$.  Since $p_K|_D=\operatorname{id}_D$,
\begin{equation}\label{eq:nu-kernel-pushforward}
 (p_K)_\#\nu_n^\xi=\nu_n^\xi
 \quad\text{as measures on }D^K .
\end{equation}
For every open $W\subseteq\widehat D$ containing $F_\xi$, the set
$V:=D^K\setminus p_K(\widehat D\setminus W)$ is an open neighborhood
of $\xi$ with $p_K^{-1}(V)\subseteq W$.
This proves that $(C_\xi)$ is equivalent to
$\liminf_n\nu_n^\xi(W)>0$ for every such $W$; for the converse take
$W=p_K^{-1}(V)$.  The same argument shows that the stronger
\eqref{eq:weak-concentration} is equivalent to
\begin{equation}\label{eq:fiber-concentration}
 \nu_n^\xi(\widehat D\setminus W)\longrightarrow0
 \quad\text{for every open }W\subseteq\widehat D
 \text{ containing }F_\xi .
\end{equation}
Since $\nu_n^\xi(\widehat D)=1$, this is also equivalent to
$\nu_n^\xi(W)\to1$ for every such $W$.

For a prescribed $\widehat\xi\in F_\xi$, the pulled-back kernel at
that point is $\widehat K_D(\cdot,\widehat\xi):=h_\xi$.
If
\begin{equation}\label{eq:refined-point-nonvanishing}
 \liminf_{n\to\infty}\nu_n^\xi(\widehat V)>0
 \quad\text{for every open neighborhood }\widehat V
 \text{ of }\widehat\xi\text{ in }\widehat D,
\end{equation}
where the inner-boundary measures are regarded as measures on
$\widehat D$, Theorem \ref{thm:synchronized-kernel}, applied to this
compactification and the same minimal function, gives
\begin{equation*}
 y_n\in\partial D_n,\qquad y_n\to\widehat\xi\text{ in }\widehat D,
 \qquad K_{D_n}(\cdot,y_n)\to\widehat K_D(\cdot,\widehat\xi)
 \text{ locally uniformly in }D.
\end{equation*}
This selection uses every sufficiently late inner domain.  A
nonvanishing condition only on neighborhoods of $\xi$ in $D^K$
provides the fiber conclusion above; it does not by itself prescribe
which point of $F_\xi$ is approached.  The representation argument in
Subsection \ref{subsec:refined-representation} will instead provide
subsequence approximation at almost every refined boundary point for
the measure it constructs.

\begin{cor}[Positive continuous interpolation]
\label{cor:positive-continuous-interpolation}
Under the hypotheses of Theorem \ref{thm:synchronized-kernel}, let
$(\xi_n)$ be selected as there.  For every fixed $x\in D$, there is
a positive $f_x\in C(D^K)$ and an integer $n(x)\geq\max\{2,n_0\}$
such that, for every $n\geq n(x)$, one has $x\in D_{n-1}$ and
\begin{equation}\label{eq:interpolation-values}
 f_x(\xi_n)=K_{D_n}(x,\xi_n),\qquad f_x(\xi)=h_\xi(x).
\end{equation}
On the common refinement,
\begin{equation}\label{eq:refined-interpolation}
 \widehat f_x:=f_x\circ p_K\in C(\widehat D)
\end{equation}
satisfies $\widehat f_x>0$ on $\widehat D$ and
\begin{equation*}
 \widehat f_x(\xi_n)=K_{D_n}(x,\xi_n)\quad(n\text{ sufficiently large}),
 \qquad
 \widehat f_x|_{F_\xi}\equiv h_\xi(x).
\end{equation*}
\end{cor}

\begin{proof}
Choose $n(x)\geq2$ so that $x\in D_{n-1}$ and $\xi_n$ is
defined for $n\geq n(x)$. The strict-positivity assertion following
(A2) gives $K_{D_n}(x,\xi_n)>0$ for all these $n$.  By (A1), the points $\xi_n$ are pairwise distinct:
if $n<k$, then $\partial D_n\subset D_k$, whereas
$\xi_k\in\partial D_k$.  The set
\begin{equation*}
 S_x:=\{\xi\}\cup\{\xi_n:n\geq n(x)\}
\end{equation*}
is a convergent sequence together with its limit and is compact in
$D^K$.  The prescribed values in \eqref{eq:interpolation-values}
define a continuous positive function $g_x$ on $S_x$, by
\eqref{eq:kernel-convergence}.  Write $a=h_\xi(x)>0$.  For all sufficiently large $n$, the
prescribed values lie between $a/2$ and $3a/2$.  The finitely many
remaining values are positive and finite.  Taking the minimum and maximum
of those values together with $a/2$ and $3a/2$ gives numbers
$0<a_x\le b_x<\infty$ containing the whole range of $g_x$.
All sequence points are isolated in $S_x$, and convergence of the
values proves continuity at its only possible accumulation point $\xi$.
The range-preserving Tietze extension theorem \cite{Munkres00} gives
$f_x\in C(D^K)$ with values in this interval and $f_x|_{S_x}=g_x$.
On the common refinement, continuity of $p_K$ proves the asserted properties
of $f_x\circ p_K$.
\end{proof}

\begin{rem}\label{rem:interpolation-not-canonical}
The function $f_x$ depends on $x$ and on the selected sequence.
It is an interpolating function; no kernel on all of $D^K$ is
being asserted.  At $\xi_n\in D$ its prescribed value is the
inner-domain quantity $K_{D_n}(x,\xi_n)$.
The pullback $\widehat f_x$ on the common refinement is likewise an
interpolating function on the refinement.
\end{rem}

\begin{rem}[Scope of the selection condition]
Condition $(C_\xi)$ is a convenient sufficient condition, not a necessary
one.  The proof of Theorem \ref{thm:synchronized-kernel} only needs, for
each fixed $m$ and all sufficiently large $n$,
\begin{equation*}
 \nu_n^\xi(V_m)>\nu_n^\xi(A_{n,m}^{c}).
\end{equation*}
This inequality still gives a nonempty intersection $V_m\cap A_{n,m}$,
even if both measures of the indicated sets tend to zero.  Without a quantitative concentration
estimate in kernel space, mere positivity of $\nu_n^\xi(V_m)$ does not
ensure this intersection.  Minimality is used in kernel-space concentration;
(M) is used only to identify the selected inner points as minimal.

With the other hypotheses of Theorem \ref{thm:synchronized-kernel}
unchanged, selection along a subsequence only requires
\begin{equation*}
 \limsup_{n\to\infty}\nu_n^\xi(V)>0
 \quad\text{for every open neighborhood }V\text{ of }\xi.
\end{equation*}
Indeed, use $V_m$ and $A_{n,m}$ from the proof of Theorem
\ref{thm:synchronized-kernel}, and choose
$0<c_m<\limsup_n\nu_n^\xi(V_m)$.
There are arbitrarily large $n$ with $\nu_n^\xi(V_m)>c_m$,
whereas \eqref{eq:good-kernel-measure} gives
$\nu_n^\xi(A_{n,m}^c)\to0$ for each fixed $m$.
We can therefore choose $n_m>n_{m-1}$, with $n_m\geq m+3$, such that
\begin{equation*}
 \nu_{n_m}^\xi(V_m)>c_m,\qquad
 \nu_{n_m}^\xi(A_{n_m,m}^c)<c_m/2.
\end{equation*}
It follows that
\begin{equation*}
 \nu_{n_m}^\xi(V_m\cap A_{n_m,m})
 \geq\nu_{n_m}^\xi(V_m)-\nu_{n_m}^\xi(A_{n_m,m}^c)>c_m/2.
\end{equation*}
Choose $\xi_{n_m}\in\partial D_{n_m}\cap V_m\cap A_{n_m,m}$.
Then $d_K(\xi_{n_m},\xi)<1/m$, and
\begin{equation*}
 \sup_{z\in L_m}|K_{D_{n_m}}(z,\xi_{n_m})-h(z)|<1/m.
\end{equation*}
Every compact subset of $D$ lies in $L_m$ for all sufficiently
large $m$. Thus the boundary points and their kernels converge
simultaneously along $(n_m)$. These indices may depend on $h$ and
$\xi$; the argument gives a subsequence conclusion only.
\end{rem}

\subsection{Positive harmonic representation on the common refinement}
\label{subsec:refined-representation}

We reuse the joint-limit construction to obtain representation on the
common refinement. Before identifying the boundary kernel, we prove
intrinsic representation and uniqueness and give criteria for boundary
concentration.

\subsubsection{Minimal kernels and intrinsic representation}
\label{subsec:joint-kernel-limits}

The joint measures and their limits were constructed in Lemma
\ref{lem:joint-inner-limits}.  We now prove that their function
projection is concentrated on minimal functions and is unique.  No
boundary parametrization or concentration condition is used in this step.

\begin{lem}[Minimal kernels and uniqueness]
\label{lem:finite-minimal-concentration}
Assume {\rm(A1)--(A2)}. For each positive harmonic function $u$ on $D$,
there is a unique finite positive measure $\sigma_u$ carried by
$\mathcal E(D)$ such that
\begin{equation}\label{eq:finite-minimal-representation}
 u(x)=\int_{\mathcal E(D)}g(x)\,\d\sigma_u(g)
 \qquad(\forall x\in D).
\end{equation}
The same measure represents all values of $u$, and
$\sigma_u(\mathcal E(D))=u(x_0)$.
Moreover,
\begin{equation}\label{eq:all-kernel-measures-converge}
 \rho_n^u\Longrightarrow\sigma_u\quad\text{on }P.
\end{equation}
Every joint limit in Lemma \ref{lem:joint-inner-limits} is therefore
carried by $(\widehat D\setminus D)\times\mathcal E(D)$.
Condition {\rm(M)} is not required.
\end{lem}

\begin{proof}
Fix a subsequence with $\rho_{n_j}^u\Longrightarrow\sigma$.
Lemma \ref{lem:joint-inner-limits} shows that $\sigma$ is carried
by $H=\mathcal H_1(D)$ and represents $u$. We first show that
\begin{equation}\label{eq:greatest-convex-representing-measure}
 \int_H\varphi\,d\tau\leq\int_H\varphi\,d\sigma
\end{equation}
for every finite positive measure $\tau$ representing $u$ on $H$
and every continuous convex function $\varphi$ on $H$.

We use the following integral form of convexity. If a normalized
measure $\theta$ on $H$ represents $g\in H$, then
\begin{equation*}
 \varphi(g)\leq\int_H\varphi\,d\theta.
\end{equation*}
Indeed, approximate $\theta$ weakly by finitely supported normalized
measures. Their finite convex combinations belong to the compact set
$H$. Point evaluation shows that every convergent subsequence of
these combinations has limit $g$, so the combinations converge to
$g$. Finite convexity and continuity give the displayed inequality.

Since $H$ is closed in the compact metric space $P$, extend $\varphi$
to $\widetilde\varphi\in C(P)$ by the Tietze extension theorem.
No convexity of the extension is needed. For $g\in H$, set
\begin{equation*}
 A_j(g):=\int_P\widetilde\varphi\,d\rho_{n_j}^g
 =\int_{\partial D_{n_j}}\widetilde\varphi(T_{n_j}(y))g(y)
                          \,d\omega_{x_0}^{D_{n_j}}(y).
\end{equation*}
For fixed $j$, this is continuous in $g$, because convergence in
$H$ is uniform on $\partial D_{n_j}$. Also
$|A_j(g)|\leq\|\widetilde\varphi\|_\infty$, since
$\rho_{n_j}^g(P)=g(x_0)=1$.
Every subsequence of $(\rho_{n_j}^g)$ has a further weak limit
$\theta$ representing $g$ on $H$, by Lemma
\ref{lem:joint-inner-limits}. The preceding convexity inequality
therefore implies
\begin{equation*}
 \liminf_{j\to\infty}A_j(g)\geq\varphi(g).
\end{equation*}
Otherwise a subsequence violating this inequality by a fixed positive
amount would have a further limit $\theta$ with
$\int_H\varphi\,d\theta<\varphi(g)$.

Fubini's theorem and the representation of $u$ by $\tau$ give
\begin{equation*}
 \int_H A_j(g)\,d\tau(g)
 =\int_{\partial D_{n_j}}\widetilde\varphi(T_{n_j}(y))u(y)
                               \,d\omega_{x_0}^{D_{n_j}}(y)
 =\int_P\widetilde\varphi\,d\rho_{n_j}^u.
\end{equation*}
Absolute integrability follows from
$\|\widetilde\varphi\|_\infty\tau(H)<\infty$.
Apply Fatou's lemma to $A_j+\|\widetilde\varphi\|_\infty$ and
use $\tau(H)=u(x_0)$. It follows that
\begin{equation*}
 \int_H\varphi\,d\tau
 \leq\int_H\liminf_j A_j\,d\tau
 \leq\liminf_j\int_H A_j\,d\tau
 =\int_H\varphi\,d\sigma,
\end{equation*}
which proves \eqref{eq:greatest-convex-representing-measure} directly.

Write $\tau\preceq\sigma$ when all these convex integral inequalities
hold. This is an order: differences of continuous convex functions
form a vector lattice containing constants and separating points,
and hence are dense in $C(H)$ by the lattice Stone--Weierstrass
theorem. Thus mutual comparison implies equality of measures.
If $\sigma\preceq\theta$, both $g\mapsto g(x)$ and its negative
are continuous convex functions, so $\theta$ also represents $u$.
Equation \eqref{eq:greatest-convex-representing-measure} then gives
$\theta\preceq\sigma$, hence $\theta=\sigma$.
The measure $\sigma$ is therefore maximal for this order.
On a compact metrizable convex set, a maximal measure is carried
by the extreme points \cite{Phelps01}.
Apply this criterion to $\sigma/u(x_0)$ to obtain a representing
measure $\sigma_u$ carried by $\mathcal E(D)$.

Uniqueness and convergence follow directly from kernel concentration.
Let $\tau$ be any finite positive measure carried by $\mathcal E(D)$
and representing $u$. For $F\in C(P)$, Fubini's theorem gives
\begin{equation*}
 \int_PF\,d\rho_n^u
 =\int_{\mathcal E(D)}\left(\int_PF\,d\rho_n^g\right)d\tau(g).
\end{equation*}
For each $g\in\mathcal E(D)$, Lemma
\ref{lem:kernel-space-concentration} gives
$\int_PF\,d\rho_n^g\to F(g)$, and these integrals are bounded
by $\|F\|_\infty$. Dominated convergence yields
\begin{equation*}
 \int_PF\,d\rho_n^u\longrightarrow
 \int_{\mathcal E(D)}F(g)\,d\tau(g).
\end{equation*}
Taking $\tau=\sigma_u$ proves \eqref{eq:all-kernel-measures-converge}.
The same limit for every other $\tau$ proves $\tau=\sigma_u$.
Finally, every joint limit has function projection $\sigma_u$ and
is carried by $(\widehat D\setminus D)\times\mathcal E(D)$.
\end{proof}

This convex-comparison method has classical precedents in constructions
from harmonic measures \cite{Loeb76,Loeb82,Loeb19}. We next identify the
function coordinate from the boundary position, obtain the refined
representation, and then apply the gluing map.

\subsubsection{Boundary parametrization and concentration}
\label{subsec:refined-representation-descent}

To express the joint-limit representation on a chosen boundary, introduce
the following \textit{boundary parametrization data}:
\begin{enumerate}
\item[(B)] A compact metrizable compactification $D^K$, a set
$\partial_{\mathrm m}^{K}D\subseteq D^K\setminus D$, and a surjection
\begin{equation}\label{eq:kappa-map}
 \kappa:\partial_{\mathrm m}^{K}D\longrightarrow\mathcal{E}(D),
 \qquad K_D(x,\eta):=\kappa(\eta)(x).
\end{equation}
\end{enumerate}
For every $\eta\in\partial_{\mathrm m}^{K}D$, the function
$K_D(\cdot,\eta)=\kappa(\eta)$ is positive and minimal, and
$K_D(x_0,\eta)=1$. No measurability or injectivity is imposed in
(B). These properties follow from the concentration condition below;
see Lemma \ref{lem:automatic-boundary-measurability}.

Choose any second compact metrizable compactification $D^G$, form
$\widehat D=D^G\vee D^K$ as above, and set
\begin{equation}\label{eq:lifted-minimal-boundary}
 \widehat{\partial}_{\mathrm m}D
 :=p_K^{-1}(\partial_{\mathrm m}^{K}D).
\end{equation}
This is a subset of the boundary of $\widehat D$. Under $(C')$,
it is Borel by Lemma \ref{lem:automatic-boundary-measurability}.

\begin{enumerate}
\item[(C)] For every $\eta\in\partial_{\mathrm m}^{K}D$, put
$\nu_n^\eta:=K_D(\cdot,\eta)\,\omega_{x_0}^{D_n}$ and assume,
for every open neighborhood $V$ of $\eta$ in $D^K$, that
\begin{equation}\label{eq:global-concentration}
 \nu_n^\eta(D^K\setminus V)
 =\int_{\partial D_n\setminus V}K_D(y,\eta)\,
                 \d\omega_{x_0}^{D_n}(y)\longrightarrow0
 \quad(n\to\infty).
\end{equation}
\end{enumerate}
For existence of a refined representation by limits of inner-boundary
measures, we use the weaker requirement that there exist
$n_1<n_2<\cdots$ such that
\begin{equation*}
 \nu_{n_j}^\eta(D^K\setminus V)\longrightarrow0
 \quad(j\to\infty)
 \tag{$C'$}\label{eq:common-subsequence-concentration}
\end{equation*}
for every $\eta\in\partial_{\mathrm m}^KD$ and every open
neighborhood $V$ of $\eta$ in $D^K$.
The order of quantifiers is
\begin{equation*}
 \exists (n_j)_{j\geq1}\ \ \forall\eta\in\partial_{\mathrm m}^KD\ \ 
 \forall V\ni\eta\text{ open in }D^K.
\end{equation*}
The sequence $(n_j)$ is the same for every $\eta$ and $V$, and is
chosen before the harmonic function $u$ to be represented. The rate
of convergence may depend on $\eta$ and $V$; no uniform convergence
in these parameters is required in either (C) or $(C')$.
Since $\nu_n^\eta(D^K)=K_D(x_0,\eta)=1$, the equivalence
proved after \eqref{eq:weak-concentration} shows that (C) and $(C')$
are respectively equivalent to
\begin{equation*}
 \nu_n^\eta\Longrightarrow\delta_\eta,
 \qquad \nu_{n_j}^\eta\Longrightarrow\delta_\eta
 \quad\text{on }D^K,\qquad \eta\in\partial_{\mathrm m}^KD.
\end{equation*}
Replacing
$(D_n)$ by $(D_{n_j})$ preserves (A1)--(A2).  Thus $(C')$ becomes
(C) after this replacement, without any additional assumption.

\begin{lem}[Measurability from boundary concentration]
\label{lem:automatic-boundary-measurability}
Assume {\rm(A1)--(A2)}, {\rm(B)}, and $(C')$. Then
$\partial_{\mathrm m}^KD$ is Borel in $D^K$, and $\kappa$ is a
Borel bijection with Borel inverse. The kernel $K_D$ is jointly
Borel on $D\times\partial_{\mathrm m}^KD$.
\end{lem}
\begin{proof}
Suppose $\kappa(\eta)=\kappa(\zeta)$. The corresponding measures
$\nu_{n_j}^\eta$ and $\nu_{n_j}^\zeta$ are equal for every $j$.
Condition $(C')$ makes them converge both to $\delta_\eta$ and to
$\delta_\zeta$. Uniqueness of a weak limit on the compact Hausdorff
space $D^K$ gives $\eta=\zeta$. Thus $\kappa$ is bijective.
Write $\eta_g=\kappa^{-1}(g)$.

Let $\mathcal M_1(D^K)$ denote the compact metric space of positive
Borel measures $\nu$ on $D^K$ with $\nu(D^K)=1$, equipped with
weak convergence. For $g\in\mathcal E(D)$, set
\begin{equation*}
 S_n(g):=g\,\omega_{x_0}^{D_n}\in\mathcal M_1(D^K).
\end{equation*}
For $f\in C(D^K)$ and $g,g'\in\mathcal E(D)$,
\begin{equation*}
 \left|\int f\,dS_n(g)-\int f\,dS_n(g')\right|
 \leq\|f\|_\infty\|g-g'\|_{\partial D_n}.
\end{equation*}
Thus $S_n$ is continuous. Condition $(C')$ gives
\begin{equation*}
 S_{n_j}(g)\longrightarrow\delta_{\eta_g}
 \quad(g\in\mathcal E(D)).
\end{equation*}
The limit map is Borel, being a pointwise limit of continuous maps
into a metric space. The map $z\mapsto\delta_z$ is a continuous
injection from compact $D^K$ into $\mathcal M_1(D^K)$, and its
inverse on its image is continuous. Consequently $g\mapsto\eta_g$
is Borel and injective. The set $\mathcal E(D)$ is a $G_\delta$
subset of the compact metric space $H$, hence a standard Borel space.
The Borel injection theorem \cite[Theorem~3.2.7]{Berberian88} shows
that its image $\partial_{\mathrm m}^KD$ is Borel and that its
inverse $\kappa$ is Borel. Finally,
\begin{equation*}
 K_D(x,\eta)=\kappa(\eta)(x)
\end{equation*}
is jointly Borel because evaluation is continuous on $D\times C(D)$
in the locally uniform topology. Here $D$ is locally compact by (A1).
\end{proof}

One may take $D^G=D^K$, giving the identity refinement. Thus the
representation theorem requires only the assignment (B) and the common
concentration condition $(C')$ as boundary data; their Borel properties
are conclusions. Without concentration, the independent representation
statement in Remark \ref{rem:representation-without-concentration}
retains the Borel bijection explicitly.

\subsubsection{Criteria for boundary concentration}
\label{subsubsec:concentration-criteria}

We retain (C) and $(C')$ in their neighborhood form. The first three
criteria use only the following measure data. Let $D^K$ be a compact
metrizable compactification of a metrizable space $D$, let
$\partial_{\mathrm m}^KD$ be a subset of $D^K\setminus D$,
and let $D_n\Subset D$ be domains containing $x_0$. Suppose that
$\omega_{x_0}^{D_n}$ are positive Borel measures on $\partial D_n$
and $K_D(\cdot,\eta)$ are nonnegative Borel functions on $D$, with
\begin{equation}\label{eq:criterion-normalization}
 \omega_{x_0}^{D_n}(\partial D_n)=1,\qquad
 \nu_n^\eta:=K_D(\cdot,\eta)\,\omega_{x_0}^{D_n},\qquad
 \nu_n^\eta(D^K)=1.
\end{equation}
The normalization holds for every integer $n\geq1$ and every
$\eta\in\partial_{\mathrm m}^KD$. Here measures on $\partial D_n$
are extended by zero to $D^K$.
For these three criteria, $\partial_{\mathrm m}^KD$ is only an
index set of boundary points: neither minimality of the functions
nor a bijective parametrization is required. Under (A1) and (B),
the preceding notation has its original meaning and the harmonic
mean-value identity gives \eqref{eq:criterion-normalization}.
Regularity of $D_n$ is not used in the criteria. An increasing
exhaustion is needed only for the implication from boundary limits
in Corollary \ref{prop:kernel-decay-concentration}.

\begin{prop}[An equivalent condition using distance]
\label{prop:distance-concentration}
Under the measure assumptions above, let $d_K\leq1$
be a metric inducing the topology of $D^K$, and put
\begin{equation}\label{eq:distance-concentration-integral}
 A_n(\eta):=\int_{\partial D_n}
 d_K(y,\eta)^2K_D(y,\eta)\,\d\omega_{x_0}^{D_n}(y),
 \qquad \eta\in\partial_{\mathrm m}^KD.
\end{equation}
Then
\begin{equation*}
 {\rm(C)}\quad\Longleftrightarrow\quad
 A_n(\eta)\longrightarrow0
 \quad\text{for every }\eta\in\partial_{\mathrm m}^KD.
\end{equation*}
Likewise, $(C')$ is equivalent to the existence of one increasing
sequence $(n_j)$ such that
\begin{equation*}
 A_{n_j}(\eta)\longrightarrow0
 \quad\text{for every }\eta\in\partial_{\mathrm m}^KD.
\end{equation*}
In particular, if there are $a_n\geq0$ and finite constants
$c_\eta>0$ such that
\begin{equation}\label{eq:common-rate-concentration}
 A_n(\eta)\leq c_\eta a_n
 \quad(n\geq1,\ \eta\in\partial_{\mathrm m}^KD),
 \qquad \liminf_{n\to\infty}a_n=0,
\end{equation}
then $(C')$ holds.  If $a_n\to0$, then {\rm(C)} holds.
The constants $c_\eta$ need not be bounded as $\eta$ varies.
\end{prop}

\begin{proof}
Fix $\eta\in\partial_{\mathrm m}^KD$.  For $0<r<1$, write
$B_K(\eta,r):=\{z\in D^K:d_K(z,\eta)<r\}$.  On its complement,
$d_K(y,\eta)^2\geq r^2$, whereas on this ball
$d_K(y,\eta)^2<r^2$.  Since $d_K\leq1$ and
$\nu_n^\eta(D^K)=1$, we obtain
\begin{equation}\label{eq:distance-neighborhood-comparison}
 \begin{aligned}
 r^2\nu_n^\eta(D^K\setminus B_K(\eta,r))
 &\leq A_n(\eta),\\
 A_n(\eta)
 &\leq r^2+\nu_n^\eta(D^K\setminus B_K(\eta,r)).
 \end{aligned}
\end{equation}
If $A_n(\eta)\to0$ and $V$ is an open neighborhood of $\eta$,
choose $r>0$ with $B_K(\eta,r)\subseteq V$.  The first inequality
gives
\begin{equation*}
 0\leq\nu_n^\eta(D^K\setminus V)
 \leq r^{-2}A_n(\eta)\longrightarrow0.
\end{equation*}
Conversely, (C) and the second inequality give
$\limsup_n A_n(\eta)\leq r^2$ for every $r>0$.
Letting $r\downarrow0$ proves $A_n(\eta)\to0$.

The same inequalities apply along any fixed subsequence, proving
the assertion about $(C')$.  Under
\eqref{eq:common-rate-concentration}, choose $n_j\uparrow\infty$
with $a_{n_j}\to0$.  For every fixed $\eta$,
$0\leq A_{n_j}(\eta)\leq c_\eta a_{n_j}\to0$.
The indices depend only on $(a_n)$, so this is one subsequence for all
$\eta$.  The case $a_n\to0$ uses the original sequence.
The same argument works with any fixed positive power of $d_K$
in place of its square.
\end{proof}

The next criterion allows the kernel to remain large on an exceptional
part of the inner boundary. It uses only an upper bound for the harmonic
measure of that part; a function satisfying a mean-value inequality
provides a sufficient way to obtain the bound.

\begin{prop}[Concentration with exceptional boundary sets]
\label{prop:exceptional-set-concentration}
Under the same measure assumptions, suppose that, for every $n\geq1$,
there are a Borel set $E_n\subseteq\partial D_n$ and a finite number
$e_n\geq0$ such that
\begin{equation}\label{eq:exceptional-measure-bound}
 \omega_{x_0}^{D_n}(E_n)\leq e_n.
\end{equation}
For every $\eta\in\partial_{\mathrm m}^KD$ and every open
neighborhood $V$ of $\eta$ in $D^K$, suppose there are an integer
$N=N(\eta,V)$ and finite numbers $a_n(\eta,V),M_n(\eta,V)\geq0$
for $n\geq N$ such that
\begin{equation}\label{eq:exceptional-kernel-bounds}
 \begin{aligned}
 K_D(y,\eta)&\leq a_n(\eta,V)
       &&(y\in\partial D_n\setminus(V\cup E_n)),\\
 K_D(y,\eta)&\leq M_n(\eta,V)
       &&(y\in E_n\setminus V).
 \end{aligned}
\end{equation}
For each fixed $n,\eta,V$, these inequalities need only hold
$\omega_{x_0}^{D_n}$-almost everywhere on the indicated sets.
Assume, for every such pair $(\eta,V)$, that
\begin{equation}\label{eq:exceptional-combined-error}
 a_n(\eta,V)+M_n(\eta,V)e_n\longrightarrow0.
\end{equation}
Then
\begin{equation}\label{eq:exceptional-concentration-estimate}
 \nu_n^\eta(D^K\setminus V)
 \leq a_n(\eta,V)+M_n(\eta,V)e_n
 \quad(\forall n\geq N(\eta,V)),
\end{equation}
and hence (C) holds. If these assumptions hold along one subsequence
common to all $\eta$ and $V$, then $(C')$ holds along that subsequence.

In particular, one may take $e_n=q_n(x_0)$ if a nonnegative Borel
function $q_n:D\to[0,+\infty]$ satisfies
\begin{equation}\label{eq:exceptional-harmonic-majorant}
 \begin{aligned}
 q_n&\geq1\quad\omega_{x_0}^{D_n}\text{-almost everywhere on }E_n,\\
 \omega_{x_0}^{D_n}(q_n)&\leq q_n(x_0)<+\infty.
 \end{aligned}
\end{equation}
\end{prop}

\begin{proof}
Fix $\eta,V$, and let $n\geq N(\eta,V)$. The two sets
$\partial D_n\setminus(V\cup E_n)$ and $E_n\setminus V$ are
disjoint and their union is $\partial D_n\setminus V$. Therefore
\begin{equation*}
 \begin{aligned}
 \nu_n^\eta(D^K\setminus V)
 &=\int_{\partial D_n\setminus(V\cup E_n)}K_D(y,\eta)
                         \d\omega_{x_0}^{D_n}(y)
   +\int_{E_n\setminus V}K_D(y,\eta)\d\omega_{x_0}^{D_n}(y)\\
 &\leq a_n(\eta,V)\omega_{x_0}^{D_n}(\partial D_n)
       +M_n(\eta,V)\omega_{x_0}^{D_n}(E_n)\\
 &\leq a_n(\eta,V)+M_n(\eta,V)e_n\longrightarrow0.
 \end{aligned}
\end{equation*}
Null exceptional sets do not change these integrals. This proves
\eqref{eq:exceptional-concentration-estimate} and (C). The calculation
uses no indices outside the chosen sequence, so it also proves the
assertion about $(C')$.

Under \eqref{eq:exceptional-harmonic-majorant}, nonnegativity gives
\begin{equation}\label{eq:exceptional-harmonic-measure}
 0\leq\omega_{x_0}^{D_n}(E_n)
 \leq\int_{\partial D_n}q_n(y)\d\omega_{x_0}^{D_n}(y)
 \leq q_n(x_0).
\end{equation}
Thus \eqref{eq:exceptional-measure-bound} holds with $e_n=q_n(x_0)$,
proving the last assertion.
\end{proof}

The proposition requires only \eqref{eq:exceptional-combined-error}.
For each fixed $\eta,V$, the separate conditions $a_n(\eta,V)\to0$,
$e_n\to0$, and boundedness of $M_n(\eta,V)$ are sufficient, but
are not required individually.
The functions $q_n$ are optional: if the bound for
$\omega_{x_0}^{D_n}(E_n)$ is known directly, they are unnecessary.
When they are used, harmonicity on $D$ gives the mean-value inequality
with equality under (A1). Capacitary potentials in Lemma
\ref{lem:polar-capacity-control} satisfy the inequality, while the
functions for curve networks in Lemma \ref{lem:curve-harmonic-majorant}
are harmonic. No estimate of the weighted measure $\nu_n^\eta(E_n)$
is required. The kernel may have a nonzero limit on $E_n\setminus V$,
and minimality of $K_D(\cdot,\eta)$ is not used.

\begin{cor}[Concentration from boundary kernel decay]
\label{prop:kernel-decay-concentration}
Under the same measure assumptions, suppose that, for every $\eta\in\partial_{\mathrm m}^KD$ and every open
neighborhood $V$ of $\eta$ in $D^K$,
\begin{equation}\label{eq:kernel-decay-away-point}
 \sup_{y\in\partial D_n\setminus V}K_D(y,\eta)
 \longrightarrow0,
\end{equation}
where the supremum is set equal to zero when the set is empty.
Then {\rm(C)} holds. If, in addition,
\begin{equation*}
 D_n\subseteq D_{n+1},\qquad D=\bigcup_{n\geq1}D_n,
\end{equation*}
then \eqref{eq:kernel-decay-away-point} follows from
\begin{equation}\label{eq:kernel-zero-other-boundary}
 \lim_{D\ni y\to\zeta}K_D(y,\eta)=0
 \quad\left(
 \begin{array}{l}
 \eta\in\partial_{\mathrm m}^KD,\\
 \zeta\in(D^K\setminus D)\setminus\{\eta\}
 \end{array}\right),
\end{equation}
where convergence is in $D^K$.  If
\eqref{eq:kernel-decay-away-point} holds along one subsequence
common to all $\eta$ and $V$, then $(C')$ holds.
\end{cor}

\begin{proof}
Apply Proposition \ref{prop:exceptional-set-concentration} with
$E_n=\varnothing$, $e_n=0$, $M_n(\eta,V)=0$, and
\begin{equation*}
 a_n(\eta,V)=\sup_{y\in\partial D_n\setminus V}K_D(y,\eta).
\end{equation*}
These choices satisfy the two kernel bounds, and
\eqref{eq:exceptional-concentration-estimate} becomes
\begin{equation*}
 \nu_n^\eta(D^K\setminus V)\leq a_n(\eta,V)\longrightarrow0.
\end{equation*}
This proves the first assertion, also along a common subsequence.

Suppose that \eqref{eq:kernel-zero-other-boundary} holds and that
\eqref{eq:kernel-decay-away-point} fails for some $\eta$ and $V$.
There are $\varepsilon>0$, $n_j\uparrow\infty$, and
$y_j\in\partial D_{n_j}\setminus V$ such that
$K_D(y_j,\eta)\geq\varepsilon$.  Compactness of $D^K$ allows us
to pass to a further subsequence with $y_j\to\zeta\in D^K\setminus V$.
For each fixed $m$, all sufficiently late $y_j$ lie outside $D_m$.
The increasing relatively compact domains cover $D$, so every
point of $D$ has a neighborhood with compact closure in $D$.
Thus $D$ is locally compact and open in its Hausdorff compactification;
in particular, $D_m$ is open in $D^K$.
Thus $\zeta\notin D_m$ for every $m$, and
$\zeta\in D^K\setminus D$.  Also $\zeta\ne\eta$, because
$\eta\in V$.  Equation \eqref{eq:kernel-zero-other-boundary}
now gives $K_D(y_j,\eta)\to0$, a contradiction.
\end{proof}

The next criterion uses the relative inner kernels and therefore
returns to (A1)--(A2) and the boundary data (B).

\begin{prop}[Concentration from separation of inner kernels]
\label{prop:kernel-separation-concentration}
Assume {\rm(A1)--(A2)} and use the boundary data {\rm(B)}.
Suppose that, for every $\eta\in\partial_{\mathrm m}^KD$ and
every open neighborhood $V$ of $\eta$ in $D^K$, there are
integers $n_0,s\geq1$, points $x_1,\ldots,x_s\in D_{n_0}$,
and $\varepsilon>0$, all allowed to depend on $(\eta,V)$ but not on
$n$ or $y$, such that
\begin{equation}\label{eq:finite-kernel-separation}
 \max_{1\leq i\leq s}
 |K_{D_n}(x_i,y)-K_D(x_i,\eta)|\geq\varepsilon
 \quad(n\geq n_0,\ y\in\partial D_n\setminus V).
\end{equation}
Then {\rm(C)} holds.  If the same requirement is imposed only
for all sufficiently late indices of one common subsequence,
then $(C')$ holds.
\end{prop}

\begin{proof}
Fix $\eta,V,x_1,\ldots,x_s,\varepsilon$ as in the statement, and
put $h=K_D(\cdot,\eta)$.  Choose $m$ with
$\{x_1,\ldots,x_s\}\subset L_m$.  On the compact set
$\mathcal Q_m$ used in Lemma \ref{lem:kernel-space-concentration},
the function
\begin{equation*}
 F(g):=\sum_{i=1}^s|g(x_i)-h(x_i)|
\end{equation*}
is continuous and bounded.  That lemma, applied to the minimal
function $h$, gives
\begin{equation*}
 \begin{aligned}
 &\int_{\partial D_n}
 \sum_{i=1}^s|K_{D_n}(x_i,y)-h(x_i)|\,\d\nu_n^\eta(y)\\
 &\hspace{25mm}=\int_{\mathcal Q_m}F(g)\,\d\rho_n^{(m)}(g)
 \longrightarrow F(h|_{L_m})=0.
 \end{aligned}
\end{equation*}
Here the equality holds for $n\geq m+3$, and
$\rho_n^{(m)}$ is defined using this $h$.
For all sufficiently large $n$, \eqref{eq:finite-kernel-separation}
therefore implies
\begin{equation*}
 \begin{aligned}
 0\leq\nu_n^\eta(D^K\setminus V)
 &\leq\frac1\varepsilon\int_{\partial D_n\setminus V}
       \sum_{i=1}^s|K_{D_n}(x_i,y)-h(x_i)|\,\d\nu_n^\eta(y)\\
 &\leq\frac1\varepsilon\int_{\partial D_n}
       \sum_{i=1}^s|K_{D_n}(x_i,y)-h(x_i)|\,\d\nu_n^\eta(y)
 \longrightarrow0.
 \end{aligned}
\end{equation*}
This proves (C).  Restricting the calculation to the specified
common subsequence proves $(C')$.
\end{proof}

Proposition \ref{prop:distance-concentration} is an equivalent
formulation; the other criteria are sufficient alternatives. The
exceptional-set criterion will be used when uniform kernel decay fails.
All these criteria concern the topology of $D^K$. A metric on a quotient
identifying distinct minimal points cannot replace $d_K$, and no
single-point concentration inside a refinement fiber is imposed.

\subsubsection{Representation by refined boundary kernels}

\begin{thm}[Positive harmonic representation in a metric space]
\label{prop:direct-refined-representation}
Assume {\rm(A1)--(A2)} and $(C')$, use the boundary data
{\rm(B)}, and let $\widehat D=D^G\vee D^K$.
Fix a common subsequence from $(C')$ and relabel its domains as
$(D_n)$ throughout this theorem and its gluing corollary.
If {\rm(C)} holds, retain the original sequence.
For any positive harmonic function $u$ on $D$, regard
\begin{equation*}
 \d\lambda_n^u(y)=u(y)\,\d\omega_{x_0}^{D_n}(y)
 \quad(y\in\partial D_n)
\end{equation*}
as measures on $\widehat D$.  They have weakly convergent
subsequences.  Every such limit $\widehat\mu_u$ satisfies
\begin{equation}\label{eq:direct-refined-support}
 \widehat\mu_u(\widehat D)=u(x_0),\qquad
 \widehat\mu_u(\widehat D\setminus
              \widehat\partial_{\mathrm m}D)=0
\end{equation}
and, for every $x\in D$,
\begin{equation}\label{eq:direct-refined-integral}
 u(x)=\int_{\widehat\partial_{\mathrm m}D}
      \widehat K_D(x,\widehat\eta)\,\d\widehat\mu_u(\widehat\eta),
 \qquad
 \widehat K_D(x,\widehat\eta):=K_D(x,p_K(\widehat\eta)).
\end{equation}
There are a Borel set $A_u\subseteq\widehat\partial_{\mathrm m}D$
with $\widehat\mu_u(\widehat D\setminus A_u)=0$ and a single sequence
$n_j\uparrow\infty$ such that, for every $\widehat\eta\in A_u$,
one can choose $y_j=y_j(\widehat\eta)\in\partial D_{n_j}$ satisfying
\begin{equation}\label{eq:refined-synchronized-approximation}
 y_j\longrightarrow\widehat\eta\text{ in }\widehat D,\qquad
 K_{D_{n_j}}(\cdot,y_j)\longrightarrow
       \widehat K_D(\cdot,\widehat\eta)
 \text{ locally uniformly in }D.
\end{equation}
The indices are chosen from the subsequence defining $\widehat\mu_u$
and are common to all $\widehat\eta\in A_u$; they may depend on $u$
and $\widehat\mu_u$. The points $y_j$ and the convergence rates may
depend on $\widehat\eta$. The set $A_u$ is also common to all $x\in D$
and all compact sets in the locally uniform convergence assertion.
Under {\rm(M)}, every approximating kernel is
minimal on its inner domain.  No uniqueness of $\widehat\mu_u$
is asserted.
\end{thm}

\begin{proof}
After the stated relabeling, (C) holds.  Apply the preceding lemmas
to this exhaustion, using the compact product
$\widehat D\times P$ and the joint measures of Lemma
\ref{lem:joint-inner-limits}.  Lemma
\ref{lem:finite-minimal-concentration} constructs a measure
$\sigma_u$ on $\mathcal E(D)$ from the inner-boundary measures,
with
\begin{equation*}
 u(y)=\int_{\mathcal E(D)}g(y)\,\d\sigma_u(g),
 \qquad \sigma_u(\mathcal E(D))=u(x_0).
\end{equation*}
This measure is obtained from the inner-kernel limits and the preceding
integral comparison, before any representation on
$\partial_{\mathrm m}^{K}D$ is used.

For $g\in\mathcal E(D)$, write $\eta_g:=\kappa^{-1}(g)$.
By (C) and the equivalence proved after
\eqref{eq:weak-concentration}, the position measures
$g\,\omega_{x_0}^{D_n}$ converge to $\delta_{\eta_g}$ on $D^K$.
By Lemma \ref{lem:kernel-space-concentration}, their kernel-coordinate
measures converge to $\delta_g$ on $P$.  Consequently
\begin{equation}\label{eq:joint-minimal-point-convergence}
 (\operatorname{id},T_n)_\#
      (g\,\omega_{x_0}^{D_n})
 \Longrightarrow\delta_{(\eta_g,g)}
 \quad\text{on }D^K\times P.
\end{equation}
Indeed, for neighborhoods $U$ of $\eta_g$ and $V$ of $g$, the
measure outside $U\times V$ is at most the sum of the two marginal
measures outside $U$ and $V$, which tends to zero.  Uniform continuity
of a continuous test function on the compact product proves weak
convergence.

For $F\in C(D^K\times P)$, Fubini's theorem gives
\begin{equation*}
 \begin{aligned}
 &\int_{\partial D_n}F(y,T_n(y))u(y)\,\d\omega_{x_0}^{D_n}(y)\\
 &\quad=\int_{\mathcal E(D)}
       \left(\int_{\partial D_n}F(y,T_n(y))g(y)\,
                    \d\omega_{x_0}^{D_n}(y)\right)\d\sigma_u(g).
 \end{aligned}
\end{equation*}
The absolute value of the inner integral is at most
$\|F\|_\infty$, since $\omega_{x_0}^{D_n}(g)=1$.
Equation \eqref{eq:joint-minimal-point-convergence} and dominated
convergence yield
\begin{equation}\label{eq:joint-geometric-limit}
 (\operatorname{id},T_n)_\#\lambda_n^u
 \Longrightarrow
 (g\mapsto(\eta_g,g))_\#\sigma_u
 \quad\text{on }D^K\times P.
\end{equation}
The map on the right is Borel by Lemma
\ref{lem:automatic-boundary-measurability}.

Now take a subsequence for which
$\lambda_{n_j}^u\Longrightarrow\widehat\mu_u$ on $\widehat D$.
It exists because $\lambda_n^u(\widehat D)=u(x_0)$.
Take a further subsequence so that $\Lambda_{n_j}^u$ converges
on $\widehat D\times P$ to $\Lambda$, and keep these indices fixed
for the rest of the proof. They are chosen before any refined boundary
point. The position projection of
$\Lambda$ is $\widehat\mu_u$.  Its projection under
$(\widehat\eta,g)\mapsto(p_K(\widehat\eta),g)$ is the measure
on the right of \eqref{eq:joint-geometric-limit}, because $p_K$
is continuous and equals the identity on $D$.
The latter measure is carried by the Borel graph
\begin{equation*}
 \{(\eta,\kappa(\eta)):
       \eta\in\partial_{\mathrm m}^{K}D\}.
\end{equation*}
Thus, for $\Lambda$-almost every $(\widehat\eta,g)$,
\begin{equation}\label{eq:joint-graph-identification}
 p_K(\widehat\eta)\in\partial_{\mathrm m}^{K}D,
 \qquad g=\kappa(p_K(\widehat\eta)).
\end{equation}
This proves the support assertion in
\eqref{eq:direct-refined-support}.  Evaluation at $x$ is bounded
and continuous on a sufficiently high kernel coordinate, so
\eqref{eq:inner-exact-representation} gives
\begin{equation*}
 \begin{aligned}
 u(x)&=\int_{(\widehat D\setminus D)\times H}g(x)\,\d\Lambda(\widehat\eta,g)\\
 &=\int_{\widehat\partial_{\mathrm m}D}
       K_D(x,p_K(\widehat\eta))\,\d\widehat\mu_u(\widehat\eta).
 \end{aligned}
\end{equation*}
Hence \eqref{eq:direct-refined-integral} has been obtained on the
refinement before passing its representing measure through $p_K$.
For the simultaneous approximation, consider the Borel set
\begin{equation*}
 A_u:=\{\widehat\eta\in\widehat\partial_{\mathrm m}D:
 (\widehat\eta,\kappa(p_K(\widehat\eta)))
       \in\operatorname{supp}\Lambda\}.
\end{equation*}
It is Borel because the displayed map is Borel and
$\operatorname{supp}\Lambda$ is closed in $\widehat D\times P$.
Equation \eqref{eq:joint-graph-identification} and the fact that the
support has full measure give $\widehat\mu_u(A_u)=u(x_0)$.
Fix $\widehat\eta\in A_u$ and put
$g=\kappa(p_K(\widehat\eta))$.  Choose a decreasing neighborhood
base $(W_k)$ at $(\widehat\eta,g)$ in $\widehat D\times P$.
Since this point belongs to $\operatorname{supp}\Lambda$, every
$W_k$ has positive $\Lambda$-measure.  The open-set inequality for
weak convergence gives
\begin{equation*}
 \liminf_{j\to\infty}\Lambda_{n_j}^u(W_k)
 \geq\Lambda(W_k)>0.
\end{equation*}
For each $k$, choose $J_k$ so that
$\Lambda_{n_j}^u(W_k)>0$ for every $j\geq J_k$, and enlarge these
integers successively to have $J_{k+1}>J_k$. For
$J_k\leq j<J_{k+1}$, select $y_j\in\partial D_{n_j}$ with
\begin{equation*}
 (y_j,T_{n_j}(y_j))\in W_k.
\end{equation*}
For the finitely many $j<J_1$, choose arbitrary points of
$\partial D_{n_j}$. This is the same block selection used in the
proof of Theorem \ref{thm:synchronized-kernel}. Since $k\to\infty$
as $j\to\infty$, it gives $y_j\to\widehat\eta$. For every fixed
$m$, the $m$th coordinate of $T_{n_j}(y_j)$ is the restriction of
the actual inner kernel once $n_j\geq m+3$. Consequently,
\begin{equation*}
 \sup_{x\in L_m}
 |K_{D_{n_j}}(x,y_j)-K_D(x,p_K(\widehat\eta))|
 \longrightarrow0.
\end{equation*}
Every compact subset of $D$ is contained in some $L_m$, proving
\eqref{eq:refined-synchronized-approximation}. The thresholds $J_k$
and the selected points depend on $\widehat\eta$; the already fixed
sequence $(n_j)$ does not. Thus a single subsequence works at all
points of $A_u$, with no uniform convergence rate asserted.
Condition (M), when imposed, makes each selected kernel minimal on
its own inner domain.
\end{proof}

\subsection{Gluing to a boundary parametrizing minimal harmonic functions}
\label{subsec:glued-representation}

Having obtained the representation on $\widehat D$, we pass its measure
through $p_K$ to $\partial_{\mathrm m}^KD$.  The resulting measure is
unique although a representing measure on the refinement need not be.
We first record the general measure identities, then apply them to the
refined representation, and finally discuss the roles of the assumptions.

\subsubsection{Harmonic measures and integral representations under a quotient}
\label{subsubsec:gluing-measures}

The measure identities below hold for Borel measurable maps. In the
common refinement construction, Proposition
\ref{prop:common-refinement-projections} supplies continuous quotient
maps, so these identities apply to the boundary identifications.

\begin{prop}[Harmonic measures under a gluing map]
\label{prop:harmonic-measure-gluing}
Let $B$ and $B_q$ be topological boundary spaces, let
$q:B\to B_q$ be Borel measurable, and let
$\omega_\cdot^D=\{\omega_x^D:x\in D\}$ be a harmonic measure on $B$.
Define
\begin{equation}\label{eq:glued-harmonic-measure}
 \omega_x^{D,q}:=q_\#\omega_x^D,
 \qquad x\in D.
\end{equation}
Then $\omega_\cdot^{D,q}$ is a harmonic measure on $B_q$.  If
$\omega_\cdot^D$ is continuous, then so is
$\omega_\cdot^{D,q}$. More precisely, for every fixed $x\in D$
and every $\varphi\in L^1(B_q,\omega_x^{D,q})$,
\begin{equation}\label{eq:gluing-integral-identity}
 \int_{B_q}\varphi\,\d\omega_x^{D,q}
 =
 \int_B\varphi\circ q\,\d\omega_x^D.
\end{equation}
\end{prop}

\begin{proof}
The push-forward in \eqref{eq:glued-harmonic-measure} is a positive
Borel measure and
\begin{equation*}
 \omega_x^{D,q}(B_q)=\omega_x^D(B)=1.
\end{equation*}
If $\omega_y^{D,q}(A)=0$, then
$\omega_y^D(q^{-1}(A))=0$.  Mutual absolute continuity gives
$\omega_x^D(q^{-1}(A))=0$, hence $\omega_x^{D,q}(A)=0$.
Thus the measures have common null sets and common completed
measurable sets.  The push-forward identity \eqref{eq:pushforward-integration}, which
holds for any finite positive Borel measure on $B$, gives
\eqref{eq:gluing-integral-identity}. For a completed measurable
function on $B_q$, choose a Borel representative. The inverse image
of its exceptional null set is null for the original measure, so the
same identity holds for the completed measures.
Applying it to $|\varphi|$ proves equivalence of the relevant
integrability conditions.  If $\varphi$ is integrable for every
$\omega_x^{D,q}$, then $\varphi\circ q$ is integrable for every
$\omega_x^D$, and continuity of the original family gives continuity
of $x\mapsto\omega_x^{D,q}(\varphi)$.
\end{proof}

\begin{cor}[Integral representations under a quotient]
\label{cor:representation-under-gluing}
Let $q:B\to B_q$ be as in Proposition
\ref{prop:harmonic-measure-gluing}, let
$K_q:D\times B_q\to[0,\infty]$ be measurable, and set
\begin{equation*}
 K(x,z):=K_q(x,q(z)),
 \qquad x\in D,\quad z\in B.
\end{equation*}
If a finite positive measure $\mu$ on $B$ represents a function $u$
by
\begin{equation*}
 u(x)=\int_B K(x,z)\,\d\mu(z)\qquad(\forall x\in D),
\end{equation*}
then, for every $x\in D$,
\begin{equation}\label{eq:quotient-kernel-representation}
 u(x)=\int_{B_q}K_q(x,\zeta)\,
       \d(q_\#\mu)(\zeta).
\end{equation}
Conversely, if $\nu$ is a finite positive Radon measure on $B_q$ and
$\widehat\nu$ is a finite positive Radon lift satisfying
$q_\#\widehat\nu=\nu$, then $K$ and $\widehat\nu$ give the
corresponding representation on $B$.
\end{cor}

\begin{proof}
For fixed $x\in D$, apply the push-forward identity used in
Proposition \ref{prop:harmonic-measure-gluing} to the nonnegative
measurable function $K_q(x,\cdot)$. Since $K(x,z)=K_q(x,q(z))$,
\begin{equation*}
 u(x)=\int_B K_q(x,q(z))\d\mu(z)
     =\int_{B_q}K_q(x,\zeta)\,d(q_\#\mu)(\zeta).
\end{equation*}
This proves \eqref{eq:quotient-kernel-representation}. Substituting
$\widehat\nu$ for $\mu$ and using $q_\#\widehat\nu=\nu$
proves the converse by the same identity.
If, in addition, the spaces are compact Hausdorff and $q$ is a
continuous surjection, Lemma \ref{lem:compact-measure-lifting}
supplies such a Radon lift. No uniqueness of the lift is required.
\end{proof}

\begin{cor}[Harmonic measures on the common refinement]
\label{cor:refined-harmonic-measures}
Under the hypotheses of Theorem \ref{prop:direct-refined-representation},
let $\widehat\mu_1$ be a refined representing measure obtained there
for the constant function $1$.  For Borel
$A\subseteq\widehat\partial D$ and every $x\in D$, define
\begin{equation}\label{eq:refined-harmonic-measures}
 \widehat\omega_x^D(A):=
 \int_{A\cap\widehat\partial_{\mathrm m}D}
       \widehat K_D(x,\widehat\eta)\,\d\widehat\mu_1(\widehat\eta).
\end{equation}
Then $\{\widehat\omega_x^D:x\in D\}$ is a continuous Radon harmonic
measure on $\widehat\partial D$, with
$\widehat\omega_{x_0}^D=\widehat\mu_1$.  Its integrable boundary
functions are exactly $L^1(\widehat\mu_1)$, and their integrals are
harmonic functions of $x$.  Both projections give continuous harmonic
measures on their respective boundaries:
\begin{equation*}
 \omega_x^{D,G}:=(p_G)_\#\widehat\omega_x^D,
 \qquad \omega_x^{D,K}:=(p_K)_\#\widehat\omega_x^D.
\end{equation*}
No boundary regularity is asserted here.
\end{cor}

\begin{proof}
The representation of $1$ and the normalization of the kernels give
\begin{equation*}
 \widehat\omega_x^D(\widehat\partial D)=
 \int_{\widehat\partial_{\mathrm m}D}
       \widehat K_D(x,\widehat\eta)\,\d\widehat\mu_1(\widehat\eta)=1,
 \qquad \widehat\omega_{x_0}^D=\widehat\mu_1.
\end{equation*}
The kernels are strictly positive, so these measures have the same
null sets.  They are Radon because they are finite Borel measures on a
compact metric space.  The bounds in Lemma
\ref{lem:derived-kernel-compactness} give, for $x\in L_m$,
\begin{equation*}
 a_m\widehat\mu_1\le\widehat\omega_x^D\le b_m\widehat\mu_1.
\end{equation*}
Thus integrability for any one of these measures is equivalent to
integrability for $\widehat\mu_1$.  For $f\in L^1(\widehat\mu_1)$
and $s,t\in L_m$, the same lemma gives
\begin{equation*}
 |\widehat\omega_s^D(f)-\widehat\omega_t^D(f)|
 \le\Omega_m(s,t)\int_{\widehat\partial D}|f|\,\d\widehat\mu_1.
\end{equation*}
Since every point of $D$ lies in the interior of some $L_m$, this
proves continuity on $D$.

Let $V\Subset D$ be regular.  The kernel bound on the compact set
$\partial V$ justifies interchanging the integrals for $|f|$ by
Tonelli, and then for $f$ by positive and negative parts.  Harmonicity
of each kernel gives
\begin{equation*}
 \begin{aligned}
 \int_{\partial V}\widehat\omega_y^D(f)\,\d\omega_x^V(y)
 &=\int_{\widehat\partial_{\mathrm m}D}f(\widehat\eta)
   \left(\int_{\partial V}\widehat K_D(y,\widehat\eta)
                    \,\d\omega_x^V(y)\right)
                         \d\widehat\mu_1(\widehat\eta)\\
 &=\int_{\widehat\partial_{\mathrm m}D}
       f(\widehat\eta)\widehat K_D(x,\widehat\eta)
                         \,\d\widehat\mu_1(\widehat\eta)
 =\widehat\omega_x^D(f).
 \end{aligned}
\end{equation*}
For completed measurable data, a Borel representative gives the same
integrals, since the null sets are common.  This proves harmonicity.
Finally, Proposition \ref{prop:common-refinement-projections} supplies
the two boundary quotient maps, and Proposition
\ref{prop:harmonic-measure-gluing} applies to the family just constructed.
\end{proof}

\subsubsection{Representation on the minimal boundary and projected uniqueness}
\label{subsubsec:glued-minimal-representation}

\begin{cor}[Representation and uniqueness under a gluing map]
\label{thm:common-refinement-representation}
\label{thm:metric-positive-representation}
\label{cor:inner-boundary-realization}
Under the hypotheses of Theorem \ref{prop:direct-refined-representation},
let $u>0$ be harmonic and let $\widehat\mu_u$ be any weak limit
constructed there, with the same relabeling of the exhaustion.  Then
\begin{equation}\label{eq:direct-glued-integral}
 \mu_u:=(p_K)_\#\widehat\mu_u
\end{equation}
is the unique finite positive Borel measure on
$\partial_{\mathrm m}^KD$ satisfying, for every $x\in D$,
\begin{equation}\label{eq:metric-kernel-representation}
 u(x)=\int_{\partial_{\mathrm m}^KD}K_D(x,\eta)\,\d\mu_u(\eta),
 \qquad \mu_u(\partial_{\mathrm m}^KD)=u(x_0).
\end{equation}
Conversely, every finite nonzero positive Borel measure on
$\partial_{\mathrm m}^KD$ defines a positive harmonic function by
this integral.

A finite positive Radon measure $\widehat\nu$ on $\widehat D$,
carried by $\widehat\partial_{\mathrm m}D$, represents $u$ through
$\widehat K_D$ if and only if
\begin{equation}\label{eq:unique-kernel-pushforward}
 (p_K)_\#\widehat\nu=\mu_u.
\end{equation}
In particular, all weak limits of $(\lambda_n^u)$ on $\widehat D$
have the same push-forward.  The projected sequence satisfies
\begin{equation}\label{eq:lambda-weak}
 (p_K)_\#\lambda_n^u\Longrightarrow\mu_u
 \quad\text{weakly on }D^K.
\end{equation}
Since $p_K|_D=\operatorname{id}_D$, this also reads
$\lambda_n^u\Longrightarrow\mu_u$ when the inner-boundary measures
are regarded as measures on $D^K$.  On the refinement the full
sequence need not converge, and representing measures need not be
unique. The exact inner-domain formula
\eqref{eq:inner-exact-representation} holds throughout: for every
compact $L\subset D$, its integral equals $u(x)$ for all $x\in L$
and all sufficiently large $n$.
\end{cor}

\begin{proof}
Theorem \ref{prop:direct-refined-representation} first supplies
$\widehat\mu_u$ and its representation of $u$.  Since it is carried
by $p_K^{-1}(\partial_{\mathrm m}^KD)$, its push-forward is carried
by $\partial_{\mathrm m}^KD$.  The push-forward integration identity
gives
\begin{equation*}
 \begin{aligned}
 u(x)
 &=\int_{\widehat\partial_{\mathrm m}D}
       K_D(x,p_K(\widehat\eta))\,\d\widehat\mu_u(\widehat\eta)\\
 &=\int_{\partial_{\mathrm m}^KD}K_D(x,\eta)\,\d\mu_u(\eta).
 \end{aligned}
\end{equation*}
Evaluation at $x_0$ gives the normalization in
\eqref{eq:metric-kernel-representation}.

If $\nu$ is another finite positive measure giving this representation,
then, for every $x\in D$,
\begin{equation*}
 \int_{\mathcal E(D)}g(x)\,\d(\kappa_\#\nu)(g)
 =\int_{\partial_{\mathrm m}^KD}K_D(x,\eta)\,\d\nu(\eta)=u(x).
\end{equation*}
Lemma \ref{lem:finite-minimal-concentration} therefore gives
\begin{equation*}
 \kappa_\#\nu=\sigma_u=\kappa_\#\mu_u.
\end{equation*}
Applying the Borel inverse $\kappa^{-1}$ proves $\nu=\mu_u$.
The same calculation shows that any refined representing measure has
push-forward $\mu_u$.  Conversely, if
$(p_K)_\#\widehat\nu=\mu_u$, the first displayed calculation,
read in the opposite direction, proves its refined representation.
This establishes \eqref{eq:unique-kernel-pushforward} without imposing
uniqueness within any fiber of $p_K$.

For the converse, let $\nu$ be a finite nonzero positive Borel
measure on $\partial_{\mathrm m}^KD$.  Its push-forward
$\tau=\kappa_\#\nu$ is a finite measure on $H$, carried by
$\mathcal E(D)$, and
\begin{equation*}
 v(x):=\int_{\partial_{\mathrm m}^KD}K_D(x,\eta)\,\d\nu(\eta)
      =\int_Hg(x)\,\d\tau(g).
\end{equation*}
The integral assertion of Lemma \ref{lem:joint-inner-limits},
with $A=H$, shows that $v$ is finite, continuous, strictly positive,
and harmonic, with $v(x_0)=\nu(\partial_{\mathrm m}^KD)$.

Equation \eqref{eq:joint-geometric-limit}, projected onto $D^K$,
already gives
\begin{equation*}
 (p_K)_\#\lambda_n^u\Longrightarrow
 (\kappa^{-1})_\#\sigma_u=\mu_u.
\end{equation*}
This proves \eqref{eq:lambda-weak} along the chosen sequence,
and along the original full sequence when (C) holds. Finally,
every compact $L\subset D$ is contained in $D_n$ for all sufficiently
large $n$. Equation \eqref{eq:inner-exact-representation} then gives
the asserted equality on $L$.
\end{proof}

\subsubsection{Concentration variants and the scope of the conclusions}
\label{subsubsec:representation-scope}

\begin{rem}[The concentration required for canonical limits]
\label{rem:common-subsequence-concentration}
Condition $(C')$ suffices for the refined representation and projected
uniqueness, with canonical limits taken along its common subsequence.
Under (C), that subsequence can be the full sequence.  No conclusion
about all weak limits of the original sequence is made under $(C')$
alone.  Allowing a different subsequence for each $\eta$ does not
justify the dominated-convergence step for an arbitrary $u$.

There is also a necessary condition for the stronger conclusion.
Under (A1)--(A2) and (B), with $\partial_{\mathrm m}^KD$ Borel and
$\kappa$ a Borel bijection, suppose that, for every positive harmonic
$u$, every weak limit of the original $(\lambda_n^u)$ on
$\widehat D$ gives its representation through $\widehat K_D$ on
$\widehat\partial_{\mathrm m}D$.  Fix
$\eta\in\partial_{\mathrm m}^KD$ and take $u=K_D(\cdot,\eta)$.
For any such limit $\widehat\mu$, the measure
$\kappa_\#(p_K)_\#\widehat\mu$ represents the minimal function
$\kappa(\eta)$ on $\mathcal E(D)$.  Lemma
\ref{lem:finite-minimal-concentration} therefore gives
\begin{equation*}
 \kappa_\#(p_K)_\#\widehat\mu=\delta_{\kappa(\eta)},
 \qquad (p_K)_\#\widehat\mu=\delta_\eta.
\end{equation*}
Every subsequence of $(\lambda_n^u)$ has a further weak limit on
$\widehat D$.  Since $p_K$ is continuous, all projected limits are
$\delta_\eta$.  Compactness now yields
\begin{equation*}
 K_D(\cdot,\eta)\,\omega_{x_0}^{D_n}
 \Longrightarrow\delta_\eta\quad\text{on }D^K.
\end{equation*}
By the equivalence proved after \eqref{eq:weak-concentration}, this
gives $\nu_n^\eta(D^K\setminus V)\to0$ for every open neighborhood
$V$ of $\eta$, which is (C).  Thus (C) cannot be removed while retaining the stated
canonical-limit conclusion for every $u$ and the full sequence.
\end{rem}

\begin{rem}\label{rem:direct-route-assumptions}
The restriction of $p_K$ to the full compact boundaries is a closed
surjection.  Its restriction over the Borel minimal boundary is also
a closed surjection.  Indeed, a relatively closed set in
$p_K^{-1}(\partial_{\mathrm m}^KD)$ is $F\cap
p_K^{-1}(\partial_{\mathrm m}^KD)$ for a closed boundary set $F$,
and its image is $p_K(F)\cap\partial_{\mathrm m}^KD$.
This is relatively closed because $p_K(F)$ is compact.
The restricted map is therefore a quotient map; its continuity
also gives the measurability needed for the push-forward in
\eqref{eq:direct-glued-integral}. Compactness of that minimal boundary
is not assumed.
\end{rem}

The following measure-theoretic fact will be used only for
existence on a refinement when concentration is not assumed.

\begin{lem}[Lifting through a compact surjection]
\label{lem:compact-measure-lifting}
Let $K$ and $L$ be compact Hausdorff spaces and let
$p:K\to L$ be a continuous surjection.  For every finite positive Radon
measure $\mu$ on $L$, there is a finite positive Radon measure
$\widehat\mu$ on $K$ such that
\begin{equation}\label{eq:compact-measure-lifting}
 p_\#\widehat\mu=\mu,
 \qquad \widehat\mu(K)=\mu(L).
\end{equation}
No uniqueness of $\widehat\mu$ is asserted.
\end{lem}

\begin{proof}
Because $p$ is onto, $f\circ p=g\circ p$ implies $f=g$,
and $\|f\circ p\|_{C(K)}=\|f\|_{C(L)}$.  Hence
\begin{equation*}
 \Lambda(f\circ p):=\int_L f\d\mu
\end{equation*}
is a well-defined linear functional on $p^*C(L)$.  It satisfies
\begin{equation*}
 |\Lambda(f\circ p)|\le\mu(L)\|f\circ p\|_{C(K)},\qquad
 \Lambda(1)=\mu(L),
\end{equation*}
so $\|\Lambda\|=\mu(L)$.  The real Hahn--Banach theorem extends
it to $\widehat\Lambda:C(K)\to\mathbb{R}$ with the same norm.
For $0\le g\le1$,
\begin{equation*}
 \widehat\Lambda(g)=\widehat\Lambda(1)-\widehat\Lambda(1-g)
 \ge\mu(L)-\|\widehat\Lambda\|\|1-g\|_\infty\ge0.
\end{equation*}
If $g\ge0$ is arbitrary and nonzero, apply this inequality to
$g/\|g\|_\infty$; if $g=0$, positivity is immediate.
Thus $\widehat\Lambda$ is positive.  Riesz--Markov gives a finite
positive Radon measure $\widehat\mu$ on $K$ such that
$\widehat\Lambda(g)=\int_Kg\d\widehat\mu$.
For every $f\in C(L)$,
\begin{equation*}
 \int_Lf\,d(p_\#\widehat\mu)
 =\int_Kf\circ p\d\widehat\mu
 =\widehat\Lambda(f\circ p)=\Lambda(f\circ p)=\int_Lf\d\mu.
\end{equation*}
Uniqueness in Riesz--Markov yields $p_\#\widehat\mu=\mu$.
Taking $f=1$ also gives $\widehat\mu(K)=\mu(L)$.
The extension need not be unique, so this argument asserts existence only.
\end{proof}

\begin{rem}[Existence without concentration]
\label{rem:representation-without-concentration}
Corollary \ref{thm:common-refinement-representation} follows the
refined-limit construction under $(C')$.  Its existence and uniqueness assertions, without the conclusions
about canonical position limits or synchronized approximation, require
only (A1)--(A2) and (B), provided that $\partial_{\mathrm m}^KD$
is Borel and $\kappa$ is a Borel bijection. Its inverse is then Borel
by \cite[Proposition~3.2.5]{Berberian88}. Without concentration,
these measurable identification properties must be specified separately.
Indeed, Lemma
\ref{lem:finite-minimal-concentration} constructs $\sigma_u$ from the
inner-kernel measures.  Set $\mu_u=(\kappa^{-1})_\#\sigma_u$.
Then
\begin{equation*}
 \int_{\partial_{\mathrm m}^KD}K_D(x,\eta)\,\d\mu_u(\eta)
 =\int_{\mathcal E(D)}g(x)\,\d\sigma_u(g)=u(x).
\end{equation*}
The uniqueness and converse calculations in the corollary use neither
(C) nor $(C')$.  Extending $\mu_u$ by zero to $D^K$, Lemma
\ref{lem:compact-measure-lifting} supplies a positive Radon lift
$\widehat\nu$ to $\widehat D$, and
\begin{equation*}
 \widehat\nu(\widehat D\setminus\widehat\partial_{\mathrm m}D)
 =\mu_u(D^K\setminus\partial_{\mathrm m}^KD)=0.
\end{equation*}
The push-forward identity then gives its refined representation.
This existence assertion does not identify the lift as a weak limit
of $(\lambda_n^u)$.
\end{rem}

Under the Borel bijection assumptions of Remark
\ref{rem:representation-without-concentration}, for a fixed positive
harmonic function $u$ it suffices to impose (C), along the sequence
being used, for
$\mu_u$-almost every $\eta$, where the measure $\mu_u$ is supplied
by Remark \ref{rem:representation-without-concentration}.  Indeed,
$\sigma_u$-almost every $g$ then satisfies the position concentration
used in \eqref{eq:joint-minimal-point-convergence}.  The integrands
in the proof of \eqref{eq:joint-geometric-limit} are bounded by
$\|F\|_\infty$, so the same dominated-convergence argument gives
the graph identification and all the conclusions for this $u$.

Within the standing harmonic measure framework, the assumptions and
constructions have the following roles.
\begin{center}
\small
\begin{tabular}{@{}p{0.23\textwidth}p{0.70\textwidth}@{}}
\toprule
Assumptions or data & Conclusions or purpose\\
\midrule
(A1)--(A2) & Compactness and unique intrinsic representation.\\
$(K_\xi),(C_\xi)$ & Together with (A1)--(A2), synchronized point and kernel approximation.\\
Boundary data (B) & A surjective assignment of normalized minimal functions; under $(C')$, it is automatically a Borel bijection.\\
Common refinement & Two quotient maps and lifting of representing measures; supplied by construction.\\
(C) & The canonical-limit conclusions along the original full sequence in the chosen compactification.\\
$(C')$ & Refined representation from canonical limits along one common subsequence, as in Theorem \ref{prop:direct-refined-representation}.\\
(M) & Identification of selected inner-boundary points as minimal.\\
\bottomrule
\end{tabular}
\end{center}
Condition (M) concerns only the minimality of the approximating kernels;
it does not replace (C). Subsection \ref{subsec:split-disk} shows that
$(C_\xi)$ can hold even when concentration at $\xi$ fails in the same
compactification.

Example \ref{ex:failure-boundary-compatibility} satisfies (B) even
with a Borel bijection. This measurable identification alone does not
supply the topological compatibility required for selection.

\section{Applications to capacity, source recovery, and boundary gluing}
\label{Sec5}

Section \ref{Sec4} established approximation and representation under
refinements and quotient maps. We first verify the harmonic measure
framework and conditions (A1)--(A2), (M), (B), and (C) for the classical
Laplacian. The main analytic application is developed in Sections
\ref{subsec:polar-concentration} and \ref{subsec:curve-networks}:
capacitary bounds implement the exceptional-set criterion of Proposition
\ref{prop:exceptional-set-concentration}, and kernel moment estimates
give quantitative recovery of finite signed source measures near compact
polar sets. The recovery proof uses these estimates directly; the resulting
concentration also verifies (C) for the representation and gluing results.
For rectifiable curve networks in three dimensions, an explicit harmonic
majorant gives a logarithmic bound and a confocal ellipsoid example
shows that this order cannot be improved for the stated class.
The concentration and recovery estimates retain any prescribed
relatively compact exhaustion, without regularity of its inner boundaries.

The remaining examples concern the boundary maps. A split disk
distinguishes neighborhood nonvanishing from concentration at one
refined point. Finite weighted metric graphs verify all the abstract
hypotheses and illustrate gluing with explicit kernels.

\subsection{Verification for the classical Laplacian and Martin theory in
\texorpdfstring{$\mathbb{R}^n$}{R\^{}n}}

For the Laplacian, take the family of bounded connected $C^2$
domains.  Dirichlet solutions give Radon harmonic measures satisfying
$\omega_x^U(\partial U)=1$; the positive Poisson kernels imply
mutual absolute continuity at different interior points.  Smooth boundary points
are regular; see \cite{Helm69}.  On compact subsets of $U$, the Poisson kernel
relative to a fixed reference point is bounded above and below by positive
constants.  Thus every integrable boundary function has a harmonic
extension, continuous in the interior.  More explicitly, truncate its
positive and negative parts; the relative-kernel upper bound controls
the integral of the truncation error uniformly on each compact subset.
This proves the continuity required in Definition \ref{H_m}.

We verify condition (N) in Definition
\ref{def:harmonic-measure-system}.
Let $V\Subset U$ be two such domains and first take
$f\in C(\partial U)$.  The function
$v(x)=\omega_x^U(f)$ is classically harmonic on $U$ and continuous
on $\overline V$.  Uniqueness of the continuous Dirichlet problem
on $V$ gives
\begin{equation*}
 \omega_x^U(f)=\int_{\partial V}\omega_y^U(f)\,d\omega_x^V(y),
 \qquad x\in V.
\end{equation*}
For fixed $x\in V$, define a measure on $\partial U$ by
\begin{equation*}
 \rho_x(A):=\int_{\partial V}\omega_y^U(A)\,d\omega_x^V(y),
 \qquad A\subseteq\partial U\text{ Borel}.
\end{equation*}
The Poisson formula makes the integrand Borel, and monotone convergence
proves countable additivity.  The displayed Dirichlet identity gives
$\rho_x(f)=\omega_x^U(f)$ for every continuous $f$.
Both are finite Borel measures on the compact metric space
$\partial U$, so uniqueness in the Riesz representation theorem gives
$\rho_x=\omega_x^U$.  Integration first for simple functions and
then for bounded Borel functions proves (N) in its stated form.
Proposition \ref{prop:nested-compatibility} then gives the identity
for completed measurable integrable data.  This verifies all the nested
identities used below.

These local domains induce classical harmonicity: classical harmonic
functions satisfy their mean-value identities, and conversely the
identities on balls give the classical mean-value property.  Thus this
family satisfies the standing framework of Section \ref{Sec4}.  The
verification does not invoke uniqueness for unbounded solutions on
arbitrary intersections of domains.

We now verify the assumptions of Section \ref{Sec4}.
Let $D\subseteq\mathbb{R}^d$, $d\geq2$, be a \textit{Greenian domain} and fix
$x_0\in D$.  Let $\overline D^{\,E}$ be the closure of $D$ in
the one-point compactification of $\mathbb{R}^d$; for bounded $D$,
this is its ordinary Euclidean closure.  Write $D^M$ for the \textit{Martin
compactification} and set
\begin{equation}\label{eq:euclidean-martin-common-refinement}
 \widehat D=\overline D^{\,E\vee M}
 :=
 \overline{\{(x,x):x\in D\}}^{\,\overline D^{\,E}\times D^M}.
\end{equation}
The two coordinate projections are
\begin{equation*}
 p_E:\widehat D\longrightarrow\overline D^{\,E},\qquad
 p_M:\widehat D\longrightarrow D^M.
\end{equation*}
Proposition \ref{prop:common-refinement-projections} gives boundary
quotient maps from $\widehat\partial D$ onto
$\partial^ED:=\overline D^{\,E}\setminus D$ and
$\partial^MD:=D^M\setminus D$.  Identifying points with the same
$p_E$-image recovers the Euclidean boundary (including the point at
infinity when it occurs); identifying points with the same $p_M$-image
recovers the \textit{Martin boundary}.  In the abstract notation take
$D^G:=\overline D^{\,E}$ and $p_G:=p_E$, together with
\begin{equation}\label{eq:kernel-martin-identification}
 D^K:=D^M,\qquad
 \partial_{\mathrm m}^{K}D:=\partial_{\mathrm m}^{M}D,\qquad
 p_K:=p_M,
\end{equation}
and let $K_D$ be the \textit{normalized Martin kernel}.
Thus the abstract compactification topology is the Martin topology in
this example, and
$\widehat{\partial}_{\mathrm m}D
=p_M^{-1}(\partial_{\mathrm m}^{M}D)$.

Figure \ref{fig:euclidean-martin-refinement} displays the two boundary
projections. Each refined point records the Euclidean and Martin limits
of the same sequence in $D$; thus $\widehat\partial D$ consists of
compatible pairs of boundary points.

\begin{figure}[htbp]
\centering
\begin{tikzpicture}[
  x=1cm,y=1cm,
  boundary/.style={ellipse,draw=black,line width=0.6pt,
    minimum width=3.2cm,minimum height=1.25cm,align=center,
    inner sep=5pt,font=\small},
  projection/.style={->,line width=0.65pt},
  every node/.style={text=black}
]
\node[boundary,minimum width=5.3cm] (R) at (0,2.5)
  {$\widehat\partial D$\\[4pt]$\widehat\xi=(\xi_E,\xi_M)$};
\node[font=\small,anchor=south,yshift=6pt] at (R.north)
  {Refined boundary};
\node[boundary] (E) at (-3.9,0)
  {$\partial^E D$\\[4pt]$\xi_E$};
\node[boundary] (M) at (3.9,0)
  {$\partial^M D$\\[4pt]$\xi_M$};
\node[font=\small,anchor=north,yshift=-6pt] at (E.south)
  {Euclidean boundary};
\node[font=\small,anchor=north,yshift=-6pt] at (M.south)
  {Martin boundary};
\draw[projection] (R.south west) --
  node[pos=0.47,above left=2pt,font=\small] {$p_E$}
  (E.north east);
\draw[projection] (R.south east) --
  node[pos=0.47,above right=2pt,font=\small] {$p_M$}
  (M.north west);
\end{tikzpicture}
\caption{Common refinement of the Euclidean and Martin boundaries.
The arrows are continuous quotient maps identifying equal Euclidean
and Martin coordinates, respectively; either map may be a homeomorphism.
The full Martin boundary is shown. Representation is carried by
$\widehat\partial_{\mathrm m}D=p_M^{-1}(\partial_{\mathrm m}^{M}D)$.
For a positive harmonic function $u$, the measure
$\mu_u=(p_M)_\#\widehat\mu_u$ on $\partial_{\mathrm m}^{M}D$ is unique,
whereas $\widehat\mu_u$ may be nonunique
(Corollary \ref{cor:classical-martin-representation}).}
\label{fig:euclidean-martin-refinement}
\end{figure}

\begin{prop}[Verification of the analytic and boundary assumptions]
\label{prop:euclidean-analytic-boundary-data}
There is an exhaustion of $D$ by bounded connected $C^2$ domains
containing $x_0$ which satisfies {\rm(A1)--(A2)} and {\rm(M)}.
The identifications in \eqref{eq:kernel-martin-identification}
give the boundary data {\rm(B)}; the common refinement
\eqref{eq:euclidean-martin-common-refinement} supplies its two projections.  In particular,
$\mathcal{H}_1(D)$ is compact and metrizable by Lemma
\ref{lem:derived-kernel-compactness}.
\end{prop}

\begin{proof}
Choose a proper smooth exhaustion function on $D$, take regular
values tending to infinity, and retain the component containing
$x_0$.  Connectedness of $D$ ensures that any compact path joining
$x_0$ to a prescribed point lies in one sufficiently large sublevel
set.  Passing to a subsequence gives bounded connected smooth domains
with $\overline{D_n}\subset D_{n+1}$ and union $D$.
Thus (A1) holds.

For each $D_n$, let $P_n(x,y)$ denote its classical Poisson
kernel relative to surface measure.  Smooth boundary regularity gives
strict positivity and joint continuity.  Consequently,
\begin{equation*}
 d\omega_x^{D_n}(y)=P_n(x,y)\,dS(y),\qquad
 K_{D_n}(x,y)=\frac{P_n(x,y)}{P_n(x_0,y)}.
\end{equation*}
The denominator is strictly positive for every $y\in\partial D_n$.
Thus the quotient is positive and jointly continuous, and
$K_{D_n}(x_0,y)=1$.  This verifies (A2).
Classical Poisson--Martin theory identifies the minimal Martin boundary
of $D_n$ with $\partial D_n$, so every boundary section is minimal.
This verifies (M).

For a Greenian Euclidean domain, Martin theory gives a compact
metrizable compactification $D^M$.  Its minimal boundary is Borel,
and the normalized minimal Martin kernels give a Borel bijection onto
$\operatorname{Ext}\mathcal{H}_1(D)$; see
\cite{Mart41,Helm69,Doob84}.  These facts verify (B) under
\eqref{eq:kernel-martin-identification}.  The topology of the chosen
$D^K$ is therefore precisely the Martin topology.

Finally, apply Proposition \ref{prop:common-refinement-projections}
with $D^G=\overline D^{\,E}$ and $D^K=D^M$.  It makes
$\overline D^{\,E\vee M}$ a compact metrizable common refinement,
with the two continuous quotient maps $p_E$ and $p_M$ and their
boundary restrictions.  Both maps equal the identity on $D$.
The compactness of $\mathcal H_1(D)$ is a consequence of
(A1)--(A2), rather than a further condition to check.
\end{proof}

The stronger full-sequence concentration condition (C) also holds in
this setting, so no passage to a common subsequence is needed before
applying the representation theorem. We verify this by the classical
reduction property at a minimal Martin point and decreasing
harmonic-measure integrals of a potential. Proposition
\ref{prop:poisson-estimate-concentration} gives a direct verification
from Poisson kernel estimates for bounded $C^{1,1}$ domains, and Theorem
\ref{thm:polar-quantitative-concentration} gives quantitative bounds
when an arbitrary compact polar set is removed.

\begin{lem}[Euclidean concentration at a minimal Martin point]
\label{lem:euclidean-minimal-concentration}
Let $(D_n)$ be an exhaustion of a Greenian domain
$D\subseteq\mathbb{R}^d$ by bounded connected $C^2$ domains satisfying
{\rm(A1)}.  Fix $\xi\in\partial_{\mathrm m}^{M}D$, put
$h=K_D(\cdot,\xi)$, and define a measure on $\partial D_n$ by
\begin{equation}\label{eq:euclidean-nu}
 \d\nu_n^\xi(y):=h(y)\,\d\omega_{x_0}^{D_n}(y).
\end{equation}
Then $\nu_n^\xi$ is a finite positive measure satisfying
$\nu_n^\xi(\partial D_n)=1$, and, for every open neighborhood
$V$ of $\xi$ in $D^M$,
\begin{equation}\label{eq:euclidean-h5}
 \nu_n^\xi(D^M\setminus V)\longrightarrow0
 \quad(n\to\infty).
\end{equation}
Equivalently, $\nu_n^\xi\Longrightarrow\delta_\xi$ on $D^M$.
Consequently, the identification $D^K=D^M$ in
\eqref{eq:kernel-martin-identification} verifies the global condition
{\rm(C)}, hence $(C')$ along the full sequence and $(C_\xi)$ at
every minimal boundary point.
When
$\nu_n^\xi$ is regarded as a measure on $\widehat D$, it satisfies
the fiber-concentration property \eqref{eq:fiber-concentration}.
\end{lem}

\begin{proof}
Since $h$ is harmonic on a neighborhood of $\overline{D_n}$ and
$h(x_0)=1$, the harmonic-measure representation gives
\begin{equation*}
 \nu_n^\xi(\partial D_n)
 =\int_{\partial D_n}h(y)\,\d\omega_{x_0}^{D_n}(y)
 =h(x_0)=1.
\end{equation*}

Fix an open neighborhood $V$ of $\xi$ in $D^M$. We use the
classical reduction property at a minimal Martin point
\cite[Part~I, Chapter~XII]{Doob84}. A nonnegative superharmonic
function is called a \textit{potential} if its greatest harmonic
minorant is zero. For $A\subset D$, define the \textit{reduced function}
\begin{equation*}
 R_h^A(x):=\inf\{s(x):s\geq0\text{ is superharmonic on }D,
                          \ s\geq h\text{ on }A\},\qquad x\in D,
\end{equation*}
and let $\widehat R_h^A$ denote its lower semicontinuous
regularization. The classical property needed below is that
$\widehat R_h^A$ is a potential whenever
$\xi\notin\overline A^{\,D^M}$.

Choose an open neighborhood $V_0$ of $\xi$ with
$\overline V_0^{\,D^M}\subset V$, and put
$A=D\setminus\overline V_0^{\,D^M}$. Then $A$ is open in $D$,
contains $D\setminus V$, and its closure in $D^M$ does not contain
$\xi$. Consequently $p_V:=\widehat R_h^A$ is a potential. Since
$h$ is an admissible majorant and $A$ is open, reduction gives
\begin{equation}\label{eq:martin-potential-majorant}
 0\leq p_V(x)\leq h(x)\quad(x\in D),\qquad
 p_V(x)=h(x)\quad(x\in D\setminus V).
\end{equation}
The use of an open set $A$ ensures the last equality at every point
of $D\setminus V$.

For each $n\geq1$, define
\begin{equation*}
 v_n(x):=\omega_x^{D_n}(p_V)
       =\int_{\partial D_n}p_V(y)\,\d\omega_x^{D_n}(y),
       \qquad x\in D_n.
\end{equation*}
The boundary function is Borel and bounded by the continuous
function $h$ on $\partial D_n$. Thus $v_n$ is a finite harmonic
function on $D_n$, and the superharmonic mean inequality gives
$0\leq v_n(x)\leq p_V(x)$ for every $x\in D_n$.
For $x\in D_n$, the mean-value identity for $v_{n+1}$ yields
\begin{equation*}
 \begin{split}
 v_{n+1}(x)
 &=\int_{\partial D_n}v_{n+1}(y)\,\d\omega_x^{D_n}(y)\\
 &\leq\int_{\partial D_n}p_V(y)\,\d\omega_x^{D_n}(y)
 =v_n(x).
 \end{split}
\end{equation*}
On each relatively compact subdomain of $D$, this is a decreasing
sequence of nonnegative harmonic functions bounded above by $h$.
Harnack convergence therefore gives a harmonic limit $v$ on $D$
with $0\leq v\leq p_V$. Since $p_V$ is a potential, $v=0$.
In particular,
\begin{equation*}
 \omega_{x_0}^{D_n}(p_V)=v_n(x_0)\longrightarrow0.
\end{equation*}
Using \eqref{eq:martin-potential-majorant}, we now obtain
\begin{equation*}
 \begin{split}
 0\leq\nu_n^\xi(D^M\setminus V)
 &=\int_{\partial D_n\setminus V}h(y)\,\d\omega_{x_0}^{D_n}(y)\\
 &=\int_{\partial D_n\setminus V}p_V(y)\,\d\omega_{x_0}^{D_n}(y)\\
 &\leq\omega_{x_0}^{D_n}(p_V)\longrightarrow0.
 \end{split}
\end{equation*}
This proves \eqref{eq:euclidean-h5}.

For $F\in C(D^M)$ and $\varepsilon>0$, choose $V$ so that
$|F(y)-F(\xi)|<\varepsilon$ for $y\in V$. Normalization gives
\begin{equation*}
 \left|\int_{D^M}F(y)\,\d\nu_n^\xi(y)-F(\xi)\right|
 \leq\varepsilon+2\|F\|_\infty\nu_n^\xi(D^M\setminus V).
\end{equation*}
First let $n\to\infty$ and then $\varepsilon\downarrow0$.
This proves $\nu_n^\xi\Longrightarrow\delta_\xi$.
The equivalence of \eqref{eq:weak-concentration} and
\eqref{eq:fiber-concentration}, proved above, gives the last assertion
with $D^K=D^M$ and $p_K=p_M$.
\end{proof}

Proposition \ref{prop:euclidean-analytic-boundary-data} and Lemma
\ref{lem:euclidean-minimal-concentration} verify (A1)--(A2), (M),
(B), and (C).  In particular, $\nu_n^\xi(V)\to1$ at each minimal
point, so $(C_\xi)$ holds.  The next corollary therefore applies to
the original exhaustion.

\begin{cor}[Classical representation and synchronized approximation]
\label{cor:classical-martin-representation}
\label{cor:classical-synchronized}
\label{cor:euclidean-measure-realization}
For a Greenian domain $D\subseteq\mathbb{R}^d$, use the exhaustion
and compactifications above.  Let $u>0$ be harmonic and put
\begin{equation}\label{eq:explicit-euclidean-lambda}
 d\lambda_n=u\,d\omega_{x_0}^{D_n}.
\end{equation}
Then the following hold.
\begin{enumerate}
\item The measures $\lambda_n$, regarded as measures on $\widehat D$,
have weakly convergent subsequences.  Every limit $\widehat\mu_u$
is carried by $\widehat\partial_{\mathrm m}D$ and satisfies,
for every $x\in D$,
\begin{equation}\label{eq:euclidean-refined-representation}
 u(x)=\int_{\widehat\partial_{\mathrm m}D}
 K_D(x,p_M(\widehat\eta))\,d\widehat\mu_u(\widehat\eta),
 \qquad \widehat\mu_u(\widehat D)=u(x_0).
\end{equation}
One subsequence $(n_j)$ of the indices defining this limit can be
chosen such that, for $\widehat\mu_u$-almost every $\widehat\eta$,
there are points $y_j=y_j(\widehat\eta)\in\partial D_{n_j}$ with
\begin{equation*}
 y_j\longrightarrow\widehat\eta\text{ in }\widehat D,\qquad
 K_{D_{n_j}}(\cdot,y_j)\longrightarrow K_D(\cdot,p_M(\widehat\eta))
 \text{ locally uniformly in }D.
\end{equation*}
No uniqueness of $\widehat\mu_u$ is asserted.
\item The gluing map defines $\mu_u=(p_M)_\#\widehat\mu_u$.
This is the unique finite positive Borel measure on
$\partial_{\mathrm m}^MD$ such that, for every $x\in D$,
\begin{equation}\label{eq:classical-martin-representation}
 u(x)=\int_{\partial_{\mathrm m}^MD}K_D(x,\eta)\,d\mu_u(\eta),
 \qquad \mu_u(\partial_{\mathrm m}^MD)=u(x_0).
\end{equation}
Conversely, every finite nonzero positive Borel measure on this boundary
defines a positive harmonic function by this integral.  Moreover,
\begin{equation*}
 \lambda_n(D^M)=u(x_0),\qquad
 \lambda_n\Longrightarrow\mu_u\quad\text{weakly on }D^M,
\end{equation*}
and
\begin{equation}\label{eq:euclidean-exact-inner-representation}
 u(x)=\int_{\partial D_n}K_{D_n}(x,y)\,d\lambda_n(y),\quad x\in D_n.
\end{equation}
\item For each prescribed $\xi\in\partial_{\mathrm m}^MD$, one can
select $\xi_n\in\partial D_n=\partial_{\mathrm m}D_n$ for all
sufficiently large $n$ so that
\begin{equation*}
 \xi_n\to\xi\text{ in }D^M,\qquad
 K_{D_n}(\cdot,\xi_n)\to K_D(\cdot,\xi)
 \text{ in }C_{\mathrm{loc}}^k(D)\quad(\forall k\in\mathbb N\cup\{0\}).
\end{equation*}
For each fixed $\xi$, the same sequence works for all derivative
orders $k$ and all compact subsets of $D$. On the common refinement,
its distance to $p_M^{-1}(\xi)$ tends to zero.  A subsequence converges to a point of that fiber; the full
sequence does so when the fiber is a singleton.  The positive continuous
interpolation of Corollary \ref{cor:positive-continuous-interpolation}
also holds.  This assertion does not prescribe an arbitrary point of
the fiber as the limit of the full sequence.
\end{enumerate}
\end{cor}

\begin{proof}
The hypotheses were verified above with $D^K=D^M$ and
$\partial_{\mathrm m}^KD=\partial_{\mathrm m}^MD$.
Theorem \ref{prop:direct-refined-representation} gives the refined
limits and the almost-everywhere simultaneous approximation, using
the original sequence since (C) holds.  Condition (M) makes the
approximating kernels minimal on their inner domains.
Corollary \ref{thm:common-refinement-representation}, applied to
$p_M$, gives the unique projected measure, the converse, and weak
convergence on $D^M$.  The normalization and exact inner formula
are \eqref{eq:inner-exact-representation}.  These prove (1) and (2).

For the third assertion, Lemma
\ref{lem:euclidean-minimal-concentration} gives $(C_\xi)$ at every
prescribed minimal Martin point.  Theorem
\ref{thm:synchronized-kernel} therefore supplies points on every
sufficiently late inner boundary with local uniform kernel convergence.
To obtain convergence of derivatives, fix $L\Subset U\Subset D$
with $U$ open.  For all sufficiently large $n$, the difference
$v_n=K_{D_n}(\cdot,\xi_n)-K_D(\cdot,\xi)$ is harmonic on a
neighborhood of $\overline U$.  Interior derivative estimates give,
for every multi-index $\alpha$,
\begin{equation*}
 \sup_L|\partial^\alpha v_n|
 \le C_{L,U,\alpha}\sup_{\overline U}|v_n|\longrightarrow0.
\end{equation*}
This proves the claimed local $C^k$ convergence.  Corollaries
\ref{cor:refinement-approximation} and
\ref{cor:positive-continuous-interpolation} give the fiber and
interpolation assertions.  This selection at a prescribed Martin point
is distinct from the almost-everywhere subsequence selection on the
refined boundary in the first assertion.
\end{proof}

\subsubsection{A direct verification from Poisson kernel estimates}

For bounded $C^{1,1}$ domains, the concentration condition can also
be verified by the kernel decay criterion of Corollary
\ref{prop:kernel-decay-concentration}. Here \textit{$C^{1,1}$ boundary}
means that the boundary is locally the graph of a continuously
differentiable function whose first derivatives are Lipschitz.
The following calculation gives an estimate uniform in the boundary point.

\begin{prop}[Concentration from Poisson kernel estimates]
\label{prop:poisson-estimate-concentration}
Let $D\subseteq\mathbb R^d$, $d\geq2$, be a bounded connected
$C^{1,1}$ domain, and let $(D_n)$ be bounded connected domains satisfying
\begin{equation*}
 x_0\in D_1,\qquad D_n\Subset D_{n+1}\Subset D,
 \qquad \bigcup_nD_n=D.
\end{equation*}
No regularity of $\partial D_n$ is assumed. Identify $D^M$ with
$\overline D$ and set
\begin{equation}\label{eq:smooth-normalized-poisson-kernel}
 K_D(y,\eta):=\frac{P_D(y,\eta)}{P_D(x_0,\eta)},\qquad
 \nu_n^\eta:=K_D(\cdot,\eta)\,\omega_{x_0}^{D_n},
 \qquad \eta\in\partial D,
\end{equation}
where $y\in D$, $n\geq1$, and $P_D$ is the Poisson kernel
relative to surface measure. Put
\begin{equation*}
 \delta_D(y):=\dist(y,\partial D),\qquad
 \delta_n:=\sup_{y\in\partial D_n}\delta_D(y).
\end{equation*}
Then $\delta_n\to0$, and there is $C>0$, depending only on
$D$ and $x_0$, such that
\begin{equation}\label{eq:quantitative-poisson-concentration}
 \sup_{\eta\in\partial D}
 \nu_n^\eta(\overline D\setminus B(\eta,r))
 \leq C r^{-d}\delta_n
 \quad(n\geq1,\ r>0).
\end{equation}
The constant $C$ is independent of the exhaustion, $n$, $r$, and
$\eta$. In particular, {\rm(C)} holds along the original exhaustion.
\end{prop}

\begin{proof}
The classical identification of the Martin boundary on a bounded
$C^{1,1}$ domain gives the compactification and normalized kernel in
\eqref{eq:smooth-normalized-poisson-kernel}; see
\cite{Helm69,Hunt70}. Theorem \ref{Wiener} defines
$\omega_x^{D_n}$ even if $D_n$ is irregular. Since
$K_D(\cdot,\eta)$ is harmonic on a neighborhood of
$\overline{D_n}$, its upper and lower Perron solutions on $D_n$
coincide with that function. Hence
$\nu_n^\eta(\overline D)=K_D(x_0,\eta)=1$.
We verify concentration directly from the Poisson kernel estimate.

The following exhaustion argument holds for every bounded domain,
without boundary regularity. For $a>0$, the set
\begin{equation*}
 L_a:=\{y\in D:\delta_D(y)\geq a\}
\end{equation*}
is compact and contained in $D$.  If it is nonempty, the increasing
exhaustion gives $L_a\subset D_N$ for some $N$.  For $n\geq N$,
$\partial D_n\cap D_N=\varnothing$, and hence
$\delta_D(y)<a$ for every $y\in\partial D_n$.
The same conclusion is immediate if $L_a=\varnothing$.
Thus $\limsup_n\delta_n\leq a$ for every $a>0$, proving
$\delta_n\to0$.

We use the classical estimate for the normalized Poisson kernel,
\begin{equation}\label{eq:normalized-poisson-decay}
 K_D(y,\eta)\leq C\frac{\delta_D(y)}{|y-\eta|^d}
 \qquad(y\in D,\ \eta\in\partial D).
\end{equation}
For $d\geq3$, see \cite[Lemma~2.1, (2.2)]{Hirata08}; the standard
$C^{1,1}$ estimates also cover $d=2$, as in the Laplacian case discussed
in \cite{ChenKimSongVondracek10}. For $C^2$ domains, the Poisson-kernel
estimate is stated directly in \cite[Theorem~1]{Krantz05}.
We take this boundary-kernel estimate as a classical input; the
concentration argument below uses only this estimate and normalization.
If $y\in\partial D_n\setminus B(\eta,r)$, then
$|y-\eta|\geq r$ and $\delta_D(y)\leq\delta_n$, so
\begin{equation*}
 K_D(y,\eta)\leq C r^{-d}\delta_n.
\end{equation*}
Since $\omega_{x_0}^{D_n}(\partial D_n)=1$, integration gives
\begin{equation*}
 \begin{aligned}
 \nu_n^\eta(\overline D\setminus B(\eta,r))
 &=\int_{\partial D_n\setminus B(\eta,r)}
        K_D(y,\eta)\,\d\omega_{x_0}^{D_n}(y)\\
 &\leq C r^{-d}\delta_n\,
       \omega_{x_0}^{D_n}(\partial D_n\setminus B(\eta,r))\\
 &\leq C r^{-d}\delta_n.
 \end{aligned}
\end{equation*}
The constant is independent of $\eta$, which proves
\eqref{eq:quantitative-poisson-concentration}.
For any open neighborhood $V$ of $\eta$ in $\overline D$, choose
$r>0$ with $\overline D\cap B(\eta,r)\subseteq V$.
Then
\begin{equation*}
 0\leq\nu_n^\eta(\overline D\setminus V)
 \leq C r^{-d}\delta_n\longrightarrow0.
\end{equation*}
This is (C), and also verifies
\eqref{eq:kernel-decay-away-point} directly.
\end{proof}

Thus, on bounded $C^{1,1}$ domains, Proposition
\ref{prop:poisson-estimate-concentration} may be used in place of
Lemma \ref{lem:euclidean-minimal-concentration} in the verification
of Corollary \ref{cor:classical-martin-representation} for the
smooth inner domains used there. The boundary identification
is classical; concentration follows from
\eqref{eq:normalized-poisson-decay} and normalization of harmonic
measure.  For general Greenian domains, Lemma
\ref{lem:euclidean-minimal-concentration} supplies the verification
in the Martin topology.

\subsection{Concentration and measure recovery near a compact polar set}
\label{subsec:polar-concentration}

We use the exceptional-set estimate of Proposition
\ref{prop:exceptional-set-concentration} for compact sets removed from
a $C^{1,1}$ domain. Capacity controls the part of each inner boundary
near the removed set. The error bounds hold for arbitrary compact sets
under the geometric hypotheses below and on every prescribed relatively
compact exhaustion. Zero capacity gives concentration and weak recovery
of finite signed measures and of source terms in a Poisson equation.

\subsubsection{Capacity and boundary kernels}

Let $d\geq3$, let $\Omega\subset\R^d$ be a bounded connected domain with
$C^{1,1}$ boundary, and let $\varnothing\ne F\Subset\Omega$ be compact.
Assume that $D=\Omega\setminus F$ is connected, and fix $x_0\in D$.
Until a capacity assumption is imposed, $F$ is an arbitrary compact set.

Fix any relatively compact exhaustion $(D_n)$ of $D$ by bounded
connected domains. Thus, in the notation of the Introduction,
\begin{equation}\label{eq:polar-exhaustion}
 x_0\in D_1,\qquad D_n\Subset D_{n+1}\Subset D,
 \qquad \bigcup_{n\geq1}D_n=D.
\end{equation}
The domains $D_n$ need not be regular and their boundaries need not
be smooth. Write
\begin{equation}\label{eq:polar-tn}
 \begin{split}
 \delta_\Omega(x)&=\dist(x,\partial\Omega),\qquad
 F_t=\{x\in\R^d:\dist(x,F)\leq t\}\quad(t>0),\\
 t_n&=\max_{y\in\partial D_n}\dist(y,\partial D)
 =\max_{y\in\partial D_n}
      \min\{\delta_\Omega(y),\dist(y,F)\},\\
 E_n&=\partial D_n\cap F_{t_n}.
 \end{split}
\end{equation}
All measures on $\partial D_n$ or on $F$ are also regarded as measures on
$\overline\Omega$ by setting their values to zero outside their original supports.
For arbitrary $F$, $\overline\Omega$ is used here as a compact ambient
space; it is a compactification of $D$ when $F$ has empty interior.
The notation $\omega_x^{D_n}$ denotes classical harmonic measure
defined by continuous Perron data, as in Theorem \ref{Wiener}; in
particular, $\omega_x^{D_n}(\partial D_n)=1$.

We record the two mean-value facts needed for these possibly irregular
inner domains. If $U\Subset D$ is a bounded domain and $v$ is harmonic
on a neighborhood of $\overline U$, then $v$ belongs to both the upper
and lower Perron classes for its continuous boundary values. Therefore
\begin{equation*}
 \int_{\partial U}v(y)\d\omega_x^U(y)=v(x)\qquad(\forall x\in U).
\end{equation*}
If $q\geq0$ is superharmonic on a neighborhood of $\overline U$, then
\begin{equation*}
 \int_{\partial U}q(y)\d\omega_x^U(y)\leq q(x)\qquad(\forall x\in U).
\end{equation*}
Indeed, for every $f\in C(\partial U)$ with $0\leq f\leq q$,
lower semicontinuity gives
\begin{equation*}
 \liminf_{U\ni z\to y}q(z)\geq q(y)\geq f(y)
 \qquad(y\in\partial U).
\end{equation*}
Thus $q$ belongs to the upper Perron class for $f$, and
$\int_{\partial U}f\d\omega_x^U=H_f(x)\leq q(x)$.
The nonnegative lower semicontinuous function $q|_{\partial U}$ on
the compact metric space $\partial U$ is the increasing limit of
nonnegative continuous functions $f_j$. Applying the preceding
inequality to $f_j$ and then monotone convergence proves the assertion.
Neither argument requires boundary regularity of $U$.

The kernels at points of $F$ will be constructed from harmonic
measures of auxiliary domains. This construction also supplies all
their estimates needed below.

\begin{lem}[Singular kernels from harmonic measures]
\label{lem:singular-kernels-harmonic-measure}
For $a\in\Omega\setminus\{x_0\}$ and
$0<r<\min\{\dist(a,\partial\Omega),|a-x_0|\}$, put
\begin{equation}\label{eq:shrinking-sphere-kernels}
 U_{a,r}=\Omega\setminus\overline{B(a,r)},\qquad
 p_{a,r}(x)=\omega_x^{U_{a,r}}(\partial B(a,r))\quad(x\in U_{a,r}).
\end{equation}
Then
\begin{equation}\label{eq:singular-kernel-limit}
 K_D(x,a):=\lim_{r\downarrow0}\frac{p_{a,r}(x)}{p_{a,r}(x_0)}
 \qquad(x\in\Omega\setminus\{a\})
\end{equation}
exists locally uniformly in $x$. This \textit{normalized singular kernel}
is positive and harmonic away from $a$, equals $1$ at $x_0$, and
extends continuously as zero on
$\partial\Omega$. Its restriction to $D$ is used when $a\in F$.

Let $\Phi_a(x)=\gamma_d|x-a|^{2-d}$, where
$\gamma_d=((d-2)|\mathbb S^{d-1}|)^{-1}$ and
$|\mathbb S^{d-1}|$ is the surface area of the unit sphere, so that
$-\Delta\Phi_a=\delta_a$. In expressions involving $\omega_x^\Omega$,
$\Phi_a$ is restricted to $\partial\Omega$. Define
\begin{equation}\label{eq:singular-normalization-weight}
 w(a)=\Phi_a(x_0)-\omega_{x_0}^\Omega(\Phi_a)
 \quad(a\in\Omega\setminus\{x_0\}).
\end{equation}
The function $w$ is positive and smooth. For each
$a\in\Omega\setminus\{x_0\}$,
\begin{equation}\label{eq:singular-kernel-dirichlet}
 \begin{aligned}
 w(a)K_D(x,a)&=\Phi_a(x)-\omega_x^\Omega(\Phi_a)
                         &&(x\in\Omega\setminus\{a\}),\\
 -\Delta_xK_D(\cdot,a)&=w(a)^{-1}\delta_a
                         &&\text{in }\Omega.
 \end{aligned}
\end{equation}
The second equality is in the sense of distributions. For $a\in F$,
there are constants $C,r_0>0$, independent of $a$ and $x$, such that
\begin{equation}\label{eq:singular-kernel-estimates}
 \begin{aligned}
 K_D(x,a)&\leq C|x-a|^{2-d}
                   &&(x\in\Omega\setminus\{a\}),\\
 K_D(x,a)&\leq C\delta_\Omega(x)
                   &&(x\in\Omega,\ \delta_\Omega(x)<r_0).
 \end{aligned}
\end{equation}
The kernels are jointly continuous away from $x=a$, and
$K_D(x,a)\to+\infty$ as $x\to a$ for each $a\in F$.
\end{lem}
\begin{proof}
Fix $a\in\Omega\setminus\{x_0\}$ and write
$v_a(x)=\omega_x^\Omega(\Phi_a)$ and $s_a(x)=\Phi_a(x)-v_a(x)$.
The continuous Dirichlet problem gives $v_a\in C(\overline\Omega)$,
harmonic in $\Omega$, with
\begin{equation*}
 0\leq v_a\leq M_a:=\max_{z\in\partial\Omega}\Phi_a(z).
\end{equation*}
Thus $s_a$ is harmonic away from $a$, is zero on $\partial\Omega$,
and tends to $+\infty$ at $a$. The domains $U_{a,r}$ are connected:
a path in $\Omega$ meeting the removed ball can be rerouted through
a slightly larger annulus contained in $\Omega$. The annulus is
path connected since $d\geq3$. On every sufficiently small sphere
$\partial B(a,r)$, the function $s_a$ is positive. The maximum
principle on $U_{a,r}$ gives $s_a\geq0$ there, and the strong
maximum principle makes this inequality strict. Letting $r\downarrow0$
proves $s_a>0$ on $\Omega\setminus\{a\}$.

Put $T_r=\gamma_dr^{2-d}$. The function $p_{a,r}$ is the continuous
Dirichlet solution equal to $1$ on $\partial B(a,r)$ and $0$ on
$\partial\Omega$. The same maximum principles give
$0<p_{a,r}(x)<1$ for every $x\in U_{a,r}$, so its normalizing
denominator is positive. If $T_r>M_a$, comparison of the two boundary
values gives
\begin{equation}\label{eq:singular-sphere-comparison}
 (T_r-M_a)p_{a,r}(x)\leq s_a(x)\leq T_rp_{a,r}(x)
 \qquad(x\in U_{a,r}).
\end{equation}
Consequently,
\begin{equation*}
 0\leq T_rp_{a,r}(x)-s_a(x)
 \leq\frac{M_a s_a(x)}{T_r-M_a}.
\end{equation*}
Since $s_a$ is bounded on every compact subset of
$\Omega\setminus\{a\}$ and $T_r\to\infty$, this proves local uniform
convergence $T_rp_{a,r}\to s_a$. In particular,
$T_rp_{a,r}(x_0)\to s_a(x_0)=w(a)>0$, proving
\eqref{eq:singular-kernel-limit} and the first identity in
\eqref{eq:singular-kernel-dirichlet}. The second follows from
$-\Delta\Phi_a=\delta_a$ and harmonicity of $v_a$.

On compact subsets of $\Omega\setminus\{x_0\}$, the functions
$a\mapsto\Phi_a(x_0)$ and all parameter derivatives of
$\Phi_a(z)$, $z\in\partial\Omega$, are uniformly bounded.
Differentiation under the finite measure $\omega_{x_0}^\Omega$
therefore proves smoothness of $w$. Continuity of the boundary data
in the supremum norm and the maximum principle prove the stated
joint continuity. Since $F$ is compact,
\begin{equation}\label{eq:polar-normalization}
 0<\min_{a\in F}w(a)\leq\max_{a\in F}w(a)<\infty.
\end{equation}
Now $0<s_a\leq\Phi_a$ proves the first estimate in
\eqref{eq:singular-kernel-estimates}, while boundedness of $v_a$
proves the singular growth assertion.

For completeness, the outer boundary estimate follows directly from
an exterior-ball barrier. Choose $r_0>0$ smaller than
$\dist(F,\partial\Omega)/4$ and than a tubular-neighborhood radius
of $\partial\Omega$. There is $\rho>0$ such that, for every
$z\in\partial\Omega$, the ball with center $z+2\rho\mathbf n_z$
and radius $2\rho$ lies outside $\Omega$, where $\mathbf n_z$ is the
outward unit normal. Set $z'=z+\rho\mathbf n_z$ and
\begin{equation*}
 B_z(y)=1-\left(\frac{\rho}{|y-z'|}\right)^{d-2}.
\end{equation*}
For $y\in\Omega$, the exterior-ball condition gives
$|y-z'|^2\geq\rho^2+|y-z|^2/2$. Hence $B_z\geq0$ on $\Omega$,
and on $\Omega\cap\partial B(z,r_0)$ it is at least
\begin{equation*}
 m_0=1-\left(\frac{\rho}{\sqrt{\rho^2+r_0^2/2}}\right)^{d-2}>0.
\end{equation*}
The bound $s_a\leq\Phi_a$ gives a common bound $M_0$ on
$\Omega\cap B(z,r_0)$ for every $a\in F$. Both functions are
harmonic there, and comparison on its boundary yields
$s_a\leq(M_0/m_0)B_z$. If $\delta_\Omega(x)<r_0$, choose a nearest
boundary point $z$, so $x=z-\delta_\Omega(x)\mathbf n_z$ and
$|x-z'|=\rho+\delta_\Omega(x)$. Thus
\begin{equation*}
 s_a(x)\leq\frac{M_0}{m_0}
 \left[1-\left(\frac{\rho}{\rho+\delta_\Omega(x)}\right)^{d-2}\right]
 \leq\frac{M_0(d-2)}{m_0\rho}\delta_\Omega(x).
\end{equation*}
Division by the positive minimum in \eqref{eq:polar-normalization}
proves the remaining estimate.
\end{proof}

This proof uses harmonic measures and the Euclidean fundamental
solution, without assuming a domain Green function. When $F$ is polar,
these kernels agree with the classical normalized minimal kernels at
points of $F$. This identification will only be used to verify the
boundary data; the quantitative estimates use the displayed bounds.

For the outer boundary use the Poisson kernel $P_\Omega$ relative to
surface measure, and put
\begin{equation}\label{eq:polar-kernels}
 \begin{split}
 K_D(x,\xi)&=\frac{P_\Omega(x,\xi)}{P_\Omega(x_0,\xi)}
                    \quad(x\in D,\ \xi\in\partial\Omega),\\
 \d\nu_n^\eta(y)&=K_D(y,\eta)\d\omega_{x_0}^{D_n}(y)
       \quad(y\in\partial D_n,\ \eta\in F\cup\partial\Omega,\ n\geq1).
 \end{split}
\end{equation}
Each $K_D(\cdot,\eta)$ is positive and harmonic on $D$, so
\begin{equation}\label{eq:polar-kernel-normalization}
 \nu_n^\eta(\overline\Omega)
 =\omega_{x_0}^{D_n}(K_D(\cdot,\eta))=K_D(x_0,\eta)=1.
\end{equation}
These facts hold for every compact $F$ under consideration. Minimality
is only used after $F$ is assumed polar.

We use the \textit{Newtonian kernel} $k(x,a)=|x-a|^{2-d}$. For compact $A$, define
\begin{equation}\label{eq:polar-capacity}
 \operatorname{Cap}(A)=\left[\inf_{\substack{\tau\geq0,\ \operatorname{supp}\tau\subseteq A\\
                              \tau(A)=1}}
       \iint |x-a|^{2-d}\d\tau(x)\d\tau(a)\right]^{-1},
 \qquad \frac1{+\infty}=0.
\end{equation}
In this normalization $\operatorname{Cap}(\overline B(a,r))=r^{d-2}$.
The classical \textit{equilibrium measure} $\lambda_A$ satisfies
\begin{equation}\label{eq:polar-equilibrium}
 \begin{split}
 \lambda_A(A)&=\operatorname{Cap}(A),\qquad \operatorname{supp}\lambda_A\subseteq A,\\
 e_A(x)&:=\int_A|x-a|^{2-d}\d\lambda_A(a)\leq1
                                      \quad(x\in\R^d),\\
 e_A(x)&=1\quad(x\in\operatorname{int}A).
 \end{split}
\end{equation}
The potential $e_A$ is nonnegative and superharmonic. These equilibrium
properties and finite subadditivity of capacity are classical; see
\cite[479D--479F, 479N]{FremlinMT}. For compact sets, zero Newtonian capacity
is equivalent to \textit{polarity}. We use ``\textit{compact polar set}'' with this meaning.

Throughout this subsection, $C$ denotes a finite positive constant
which may change from one estimate to another and may depend on
$\Omega,F,x_0,d$. It is independent of the exhaustion, $n$, the
boundary point, and the radius $r$ in the tail estimates. It is also
independent of the signed measures and boundary data in the recovery
estimates. When a cutoff function is specified, $C$ may additionally
depend on that function. By ``all sufficiently large $n$'' we mean
$\forall n\geq N$ for an integer $N$ which may depend on the
exhaustion and these fixed geometric data, but not on the varying
boundary point, $r$, or source measure. Any further dependence is
stated explicitly below.

\subsubsection{Uniform concentration and the capacity criterion}

\begin{lem}[Control of the inner part of the boundary]\label{lem:polar-capacity-control}
For every compact set $F$ under consideration and every exhaustion
\eqref{eq:polar-exhaustion}, $t_n\to0$, and
\begin{equation}\label{eq:polar-capacity-control}
 c_n:=\operatorname{Cap}(F_{2t_n}),\qquad
 \omega_{x_0}^{D_n}(E_n)\leq Cc_n
\end{equation}
for all sufficiently large $n$. If $\operatorname{Cap}(F)=0$,
then $c_n\to0$.
\end{lem}
\begin{proof}
Apply the compact-subset argument in the proof of Proposition
\ref{prop:poisson-estimate-concentration} to the bounded domain
$D=\Omega\setminus F$. It uses only the increasing exhaustion and
compactness of $\{x\in D:\dist(x,\partial D)\geq\varepsilon\}$,
so it gives $t_n\to0$ without regularity of $\partial D$.

For large $n$, put $q_n=e_{F_{2t_n}}$. Each point of $F_{t_n}$ is in
the interior of $F_{2t_n}$, so $q_n=1$ on $E_n$. Also
$\dist(x_0,F_{2t_n})\geq\dist(x_0,F)/2$. By
\eqref{eq:polar-equilibrium} and the superharmonic mean-value inequality,
\begin{equation}\label{eq:polar-En-measure}
 \begin{split}
 \omega_{x_0}^{D_n}(E_n)
 &\leq\int_{\partial D_n}q_n(y)\d\omega_{x_0}^{D_n}(y)
 \leq q_n(x_0)\\
 &\leq \left(\frac{\dist(x_0,F)}2\right)^{2-d}
          \lambda_{F_{2t_n}}(F_{2t_n})
 =Cc_n.
 \end{split}
\end{equation}
This estimate does not require $F$ to be polar. The function $q_n$
need not be harmonic on $D$; its superharmonic inequality is exactly
the requirement in Proposition \ref{prop:exceptional-set-concentration}.

Suppose now that $\operatorname{Cap}(F)=0$. If
$\operatorname{Cap}(F_{2t})$ did not tend to zero as $t\downarrow0$,
there would be $t_j\downarrow0$ and $b>0$ with
$\operatorname{Cap}(F_{2t_j})\geq b$. By \eqref{eq:polar-capacity},
choose measures $\tau_j$ supported on $F_{2t_j}$, with
$\tau_j(F_{2t_j})=1$ and energy at most $2/b$. A weakly convergent
subsequence has a limit $\tau$ supported on $F$, with $\tau(F)=1$.
For each $M<\infty$, the kernel
$\min\{|x-a|^{2-d},M\}$, assigned the value $M$ when $x=a$,
is continuous on a common compact set containing these supports.
Weak convergence of the product measures therefore gives
\begin{equation*}
 \iint\min\{|x-a|^{2-d},M\}\d\tau(x)\d\tau(a)\leq2/b.
\end{equation*}
Letting $M\to\infty$ proves that $\tau$ has finite energy,
contrary to $\operatorname{Cap}(F)=0$. Hence
$\operatorname{Cap}(F_{2t})\to0$, and $t_n\to0$ gives $c_n\to0$.
\end{proof}

\begin{thm}[Uniform concentration for all compact polar sets]
\label{thm:polar-quantitative-concentration}
Assume $\operatorname{Cap}(F)=0$. There is $C>0$, depending only
on $\Omega,F,x_0,d$, such that, for every exhaustion
\eqref{eq:polar-exhaustion}, there is an integer $N$ for which the
following estimates hold for all $n\geq N$ and all $0<r\leq1$:
\begin{align}
 \sup_{\xi\in\partial\Omega}
 \nu_n^\xi(\overline\Omega\setminus B(\xi,r))
 &\leq C(t_nr^{-d}+c_n),\label{eq:polar-outer-tail}\\
 \sup_{a\in F}
 \nu_n^a(\overline\Omega\setminus B(a,r))
 &\leq C(t_n+r^{2-d}c_n).\label{eq:polar-inner-tail}
\end{align}
Thus $N$ can be chosen independently of $r$ and of both boundary
parameters $\xi$ and $a$. In particular, for every fixed $r>0$,
\begin{equation}\label{eq:polar-uniform-concentration}
 \sup_{\eta\in F\cup\partial\Omega}
 \nu_n^\eta(\overline\Omega\setminus B(\eta,r))\longrightarrow0.
\end{equation}
\end{thm}
\begin{proof}
For $y\in\partial D_n\setminus E_n$, \eqref{eq:polar-tn} gives
$\delta_\Omega(y)\leq t_n$. For $y\in E_n$, the point $y$ lies in
a fixed compact neighborhood of $F$ contained in $\Omega$.
Consequently \eqref{eq:normalized-poisson-decay} gives, uniformly
in $\xi\in\partial\Omega$,
\begin{equation*}
 K_D(y,\xi)\leq
 \begin{cases}
 Ct_nr^{-d},&y\in\partial D_n\setminus(E_n\cup B(\xi,r)),\\
 C,&y\in E_n.
 \end{cases}
\end{equation*}
For $a\in F$, \eqref{eq:singular-kernel-estimates} gives
\begin{equation*}
 K_D(y,a)\leq
 \begin{cases}
 Ct_n,&y\in\partial D_n\setminus E_n,\\
 Cr^{2-d},&y\in E_n\setminus B(a,r).
 \end{cases}
\end{equation*}
Split the integral defining each $\nu_n^\eta$ over $E_n$ and
its complement. Since $\omega_{x_0}^{D_n}(\partial D_n)=1$, the
two displayed bounds give
\begin{equation}\label{eq:polar-measure-tail-bounds}
 \begin{aligned}
 \sup_{\xi\in\partial\Omega}
 \nu_n^\xi(\overline\Omega\setminus B(\xi,r))
 &\leq C\bigl(t_nr^{-d}+\omega_{x_0}^{D_n}(E_n)\bigr),\\
 \sup_{a\in F}
 \nu_n^a(\overline\Omega\setminus B(a,r))
 &\leq C\bigl(t_n+r^{2-d}\omega_{x_0}^{D_n}(E_n)\bigr).
 \end{aligned}
\end{equation}
This is the two-part integration used in Proposition
\ref{prop:exceptional-set-concentration}. The calculation only uses
the displayed kernel bounds and normalization, so it also holds for
arbitrary compact $F$ in the ambient space $\overline\Omega$. Substituting
$\omega_{x_0}^{D_n}(E_n)\leq Cc_n$ from Lemma
\ref{lem:polar-capacity-control} proves
\eqref{eq:polar-outer-tail}--\eqref{eq:polar-inner-tail}.
When $\operatorname{Cap}(F)=0$, that lemma gives $t_n,c_n\to0$,
so concentration follows, including $r>1$ by monotonicity.
\end{proof}

\begin{thm}[Capacity criterion for concentration]\label{thm:polar-capacity-characterization}
For an arbitrary nonempty compact $F\Subset\Omega$ with $D$ connected,
the following conditions are equivalent:
\begin{enumerate}
\item[(i)] $\operatorname{Cap}(F)=0$.
\item[(ii)] There exist $a\in F$, an exhaustion \eqref{eq:polar-exhaustion},
      and indices $n_j\uparrow\infty$ such that
      $\nu_{n_j}^a\Longrightarrow\delta_a$ on $\overline\Omega$.
\item[(iii)] For every exhaustion \eqref{eq:polar-exhaustion}, concentration
      \eqref{eq:polar-uniform-concentration} holds.
\end{enumerate}
Here weak convergence means convergence of integrals of all real
continuous functions on $\overline\Omega$. In (ii), $K_D(\cdot,a)$ is the singular kernel constructed in
Lemma \ref{lem:singular-kernels-harmonic-measure}.
Thus a single subsequential limit at one specified point suffices.
\end{thm}
\begin{proof}
Theorem \ref{thm:polar-quantitative-concentration} proves
{\rm(i)}$\Rightarrow${\rm(iii)}.
Normalization \eqref{eq:polar-kernel-normalization} and uniform continuity of each
continuous test function prove (iii)$\Rightarrow$(ii).

Assume (ii), and suppose for a contradiction that $\operatorname{Cap}(F)>0$.
Fix the point $a$, the exhaustion, and the subsequence in (ii). Choose
$\chi\in C_c^\infty(\Omega)$ with $0\leq\chi\leq1$ and
$\chi=1$ on a neighborhood of $F$. Define on $\overline\Omega$
\begin{equation*}
 \psi_a(y)=
 \begin{cases}
 \displaystyle\frac{\chi(y)}{K_D(y,a)},
       &y\in\Omega\setminus\{a\},\\
 0,    &y=a\ \text{or}\ y\in\partial\Omega.
 \end{cases}
\end{equation*}
Lemma \ref{lem:singular-kernels-harmonic-measure} gives positivity
and continuity of $K_D(\cdot,a)$ away from $a$, and
$K_D(y,a)\to+\infty$ as $y\to a$.
Hence $\psi_a(y)\to0$ at $a$. Also $\chi$ vanishes on a
neighborhood of $\partial\Omega$, so $\psi_a$ vanishes there.
It follows that $\psi_a\in C(\overline\Omega)$.
On $D$ we have $\psi_a(y)K_D(y,a)=\chi(y)$. Therefore (ii) gives
\begin{equation*}
 \begin{aligned}
 \int_{\partial D_{n_j}}\chi(y)\d\omega_{x_0}^{D_{n_j}}(y)
 &=\int_{\overline\Omega}\psi_a(y)\d\nu_{n_j}^a(y)\\
 &\longrightarrow\psi_a(a)=0.
 \end{aligned}
\end{equation*}
Since $t_n\to0$ and $\chi=1$ on a neighborhood of the compact
set $F$, we have $E_n\subset F_{t_n}\subset\{\chi=1\}$ for all
sufficiently large $n$. Nonnegativity consequently gives
\begin{equation}\label{eq:polar-omega-vanishes}
 0\leq\omega_{x_0}^{D_{n_j}}(E_{n_j})
 \leq\int_{\partial D_{n_j}}\chi(y)\d\omega_{x_0}^{D_{n_j}}(y)
 \longrightarrow0.
\end{equation}

Let $\lambda_F$ be the Newtonian equilibrium measure of $F$, and set
\begin{equation*}
 p(x)=\int_F K_D(x,a)\d\lambda_F(a)\quad(x\in D).
\end{equation*}
The function $p$ is harmonic on $D$: on each compact subset of $D$
the kernels are uniformly bounded, so their mean-value identities
can be integrated against $\lambda_F$. Equations
\eqref{eq:polar-equilibrium} and \eqref{eq:singular-kernel-estimates}
give $0\leq p\leq Ce_F\leq C$, and normalization gives
$p(x_0)=\lambda_F(F)=\operatorname{Cap}(F)>0$.
On $\partial D_n\setminus E_n$ the outer estimate gives $p\leq Ct_n$.
The mean-value identity therefore yields
\begin{equation*}
 0<p(x_0)=\int_{\partial D_{n_j}}p(y)\d\omega_{x_0}^{D_{n_j}}(y)
 \leq C\omega_{x_0}^{D_{n_j}}(E_{n_j})+Ct_{n_j}\longrightarrow0.
\end{equation*}
This contradiction proves (i).
\end{proof}

\begin{rem}\label{rem:polar-capacity-scope}
The necessity assertion concerns the specified singular kernels. It does
not say that representation or concentration on the appropriate Martin
boundary is impossible after removing a set of positive capacity.
For a compact polar $F$, classical boundary identification gives
$D^K=\overline\Omega$ and $\partial_{\mathrm m}^KD=F\cup\partial\Omega$,
with the singular kernels of Lemma \ref{lem:singular-kernels-harmonic-measure}
and the outer kernels in \eqref{eq:polar-kernels}; see
\cite[Section 5.5]{GardinerPau03}
for the effect of deleting a polar set. Thus
\eqref{eq:polar-uniform-concentration} verifies condition (C)
on every prescribed relatively compact exhaustion. The quantitative concentration
proof above does not use a representation theorem.
\end{rem}

For $\operatorname{Cap}(F)=0$, choose the domains in
\eqref{eq:polar-exhaustion} with smooth boundaries. The proof of Proposition
\ref{prop:euclidean-analytic-boundary-data} verifies (A1)--(A2) and
(M). The classical identification in Remark
\ref{rem:polar-capacity-scope} supplies (B), and Theorem
\ref{thm:polar-quantitative-concentration} supplies (C) directly.
Theorem \ref{prop:direct-refined-representation} and Corollary
\ref{thm:common-refinement-representation} therefore give representation
on the common refinement and uniqueness of the projected measure, as
stated in Corollary \ref{cor:classical-martin-representation}. Here the
analytic verification of (C) replaces Lemma
\ref{lem:euclidean-minimal-concentration}. The quantitative recovery
proof below proceeds directly from the kernel estimates and does not
use this representation theorem. Its inner domains may be irregular;
the two mean-value facts established above suffice.

\begin{prop}[Comparison with the covering condition]\label{prop:polar-capacity-cover}
For each $t>0$, let $N(F,t)$ be the least number of balls of radius
$t$, with centers in $F$, needed to cover $F$. Then, for every $t>0$,
\begin{equation}\label{eq:polar-cover-capacity}
 \operatorname{Cap}(F_{2t})\leq3^{d-2}t^{d-2}N(F,t).
\end{equation}
Consequently $t^{d-2}N(F,t)\to0$ implies $\operatorname{Cap}(F)=0$. The converse is false, even for a line segment in
$\R^3$.
\end{prop}
\begin{proof}
If the balls $B(a_i,t)$ cover $F$, their enlargements
$\overline B(a_i,3t)$ cover $F_{2t}$. Monotonicity, finite
subadditivity, and the capacity of a ball give \eqref{eq:polar-cover-capacity}.
The strictness assertion is proved in Section \ref{subsec:curve-networks}.
\end{proof}

\begin{prop}[Failure of uniform kernel decay]
\label{prop:polar-failure-kernel-decay}
Under the assumptions of Theorem \ref{thm:polar-quantitative-concentration},
fix $\xi\in\partial\Omega$ and $a\in F$. There is a neighborhood $V$ of $\xi$ in $\overline\Omega$ such
that, on every exhaustion considered above,
\begin{equation}\label{eq:polar-failure}
 \liminf_{n\to\infty}
 \sup_{y\in\partial D_n\setminus V}K_D(y,\xi)>0.
\end{equation}
Nevertheless (C) holds for every $\eta\in\partial\Omega\cup F$.
\end{prop}

\begin{proof}
The outer kernel \eqref{eq:polar-kernels} extends continuously and
positively to a neighborhood of $a$ in $\Omega$. In particular
$K_D(a,\xi)>0$. Choose $V$ with $a\notin\overline V$.
Compactness of $\overline{D_n}$ gives $y_n\in\overline{D_n}$ with
$|a-y_n|=\dist(a,\overline{D_n})$. Since $a\notin\overline{D_n}$,
such a minimizing point lies in $\partial D_n$: otherwise moving a
short distance from $y_n$ towards $a$ would remain inside $D_n$ and
decrease the distance.

Moreover $\dist(a,\overline{D_n})\to0$. Indeed, for every
$\varepsilon>0$ there is $z\in D$ with $|z-a|<\varepsilon$, and
$z\in D_n$ for all sufficiently large $n$. It follows that $y_n\to a$.
Thus $y_n\notin V$ eventually and
\begin{equation*}
 \liminf_n\sup_{\partial D_n\setminus V}K_D(\cdot,\xi)
 \geq\lim_nK_D(y_n,\xi)=K_D(a,\xi)>0.
\end{equation*}
The assertion about (C) follows from Theorem
\ref{thm:polar-quantitative-concentration}.
\end{proof}

This is a strict enlargement of the scope of the sufficient condition
requiring
\begin{equation*}
 \sup_{\partial D_n\setminus V}K_D(\cdot,\eta)\longrightarrow0.
\end{equation*}
When $F$ has no isolated points, the failure occurs on a whole compact
set of nonisolated boundary points. The capacitary functions in Lemma
\ref{lem:polar-capacity-control} control the contribution of that set
through \eqref{eq:polar-En-measure}.

\subsubsection{Quantitative recovery of finite measures}

For finite signed Borel measures $\alpha,\beta$ on $\overline\Omega$, put
\begin{equation}\label{eq:bounded-lipschitz-distance}
 d_{\mathrm{BL}}(\alpha,\beta)=
 \sup_{\substack{f:\overline\Omega\to\R,\ \|f\|_\infty\leq1\\\operatorname{Lip}(f)\leq1}}
 \left|\int_{\overline\Omega}f\d\alpha-\int_{\overline\Omega}f\d\beta\right|.
\end{equation}
Here $\operatorname{Lip}(f)$ is the least \textit{Lipschitz constant} of $f$ for Euclidean
distance, and $|\mu|$ denotes the \textit{variation measure} of a signed measure
$\mu$.

\begin{lem}[First moment estimate]\label{lem:polar-first-moment}
For every compact $F$ under consideration and all sufficiently large $n$,
\begin{equation}\label{eq:polar-first-moment}
 \begin{aligned}
 &\sup_{a\in F}\int_{\partial D_n}|y-a|K_D(y,a)
                    \d\omega_{x_0}^{D_n}(y)\\
 &\qquad\leq C\bigl(t_n+
             [\omega_{x_0}^{D_n}(E_n)]^{1/(d-2)}\bigr)
 \leq C'\bigl(t_n+c_n^{1/(d-2)}\bigr).
 \end{aligned}
\end{equation}
Neither inequality requires $\operatorname{Cap}(F)=0$.
If $\operatorname{Cap}(F)=0$, the right-hand side tends to zero.
\end{lem}
\begin{proof}
Fix $a\in F$. On $\partial D_n\setminus E_n$, the kernel is at most
$Ct_n$, so
\begin{equation*}
 \int_{\partial D_n\setminus E_n}|y-a|K_D(y,a)
                         \d\omega_{x_0}^{D_n}(y)
 \leq C\operatorname{diam}(\Omega)t_n.
\end{equation*}
The first estimate in \eqref{eq:singular-kernel-estimates} gives
\begin{equation}\label{eq:polar-distance-kernel}
 |y-a|K_D(y,a)
 =\bigl(|y-a|^{d-2}K_D(y,a)\bigr)^{1/(d-2)}
       K_D(y,a)^{(d-3)/(d-2)}
 \leq C K_D(y,a)^{(d-3)/(d-2)}.
\end{equation}
For $d=3$, the right-hand side is bounded by $C$, so its integral
over $E_n$ is at most $C\omega_{x_0}^{D_n}(E_n)$.
For $d>3$, H\"older's inequality with conjugate exponents
$(d-2)/(d-3)$ and $d-2$ gives
\begin{equation*}
 \begin{aligned}
 &\int_{E_n}|y-a|K_D(y,a)\d\omega_{x_0}^{D_n}(y)\\
 &\quad\leq C
 \left(\int_{E_n}K_D(y,a)\d\omega_{x_0}^{D_n}(y)\right)^{(d-3)/(d-2)}
 [\omega_{x_0}^{D_n}(E_n)]^{1/(d-2)}\\
 &\quad\leq C[\omega_{x_0}^{D_n}(E_n)]^{1/(d-2)},
 \end{aligned}
\end{equation*}
where the last step uses \eqref{eq:polar-kernel-normalization}.
The same final bound holds in dimension three. Combining the two
parts proves the first inequality in \eqref{eq:polar-first-moment},
uniformly in $a$. Lemma \ref{lem:polar-capacity-control} gives
the second inequality and its convergence assertion.
\end{proof}

\begin{thm}[Recovery, including signed measures and errors]
\label{thm:polar-measure-recovery}
Let $F$ be any compact set under consideration, and let $\mu$ be a
finite signed Borel measure supported on $F$. Define
\begin{equation}\label{eq:polar-umu}
 u_\mu(x)=\int_FK_D(x,a)\d\mu(a)\quad(x\in D),\qquad
 \d\lambda_n^{u_\mu}(y)=u_\mu(y)\d\omega_{x_0}^{D_n}(y)
 \quad(y\in\partial D_n).
\end{equation}
Then $u_\mu$ is harmonic on $D$, and, for all sufficiently large $n$,
\begin{equation}\label{eq:polar-recovery-error}
 d_{\mathrm{BL}}(\lambda_n^{u_\mu},\mu)\leq C|\mu|(F)
                         \bigl(t_n+c_n^{1/(d-2)}\bigr).
\end{equation}
Let $0\leq\varepsilon_n<\infty$ for every $n\geq1$.
If real measurable data $v_n$ on $\partial D_n$ satisfy
\begin{equation}\label{eq:polar-data-error}
 \int_{\partial D_n}|v_n-u_\mu|\d\omega_{x_0}^{D_n}\leq\varepsilon_n,
 \qquad \d\widetilde\lambda_n=v_n\d\omega_{x_0}^{D_n},
\end{equation}
then
\begin{equation}\label{eq:polar-recovery-data-error}
 d_{\mathrm{BL}}(\widetilde\lambda_n,\mu)
 \leq\varepsilon_n+C|\mu|(F)\bigl(t_n+c_n^{1/(d-2)}\bigr).
\end{equation}
Both estimates hold also with $c_n$ replaced by
$\omega_{x_0}^{D_n}(E_n)$. Their constants and starting indices are
independent of $\mu$, $v_n$, and $\varepsilon_n$.
If $\operatorname{Cap}(F)=0$, then
$\lambda_n^{u_\mu}\Longrightarrow\mu$ on $\overline\Omega$;
if also $\varepsilon_n\to0$, then
$\widetilde\lambda_n\Longrightarrow\mu$.
For a general compact $F$, only the displayed error bounds are asserted.
\end{thm}
\begin{proof}
On each compact subset of $D$, the kernels are jointly continuous
and uniformly bounded for $a\in F$, by Lemma
\ref{lem:singular-kernels-harmonic-measure}. Integration against
$\mu$ therefore defines a continuous function. For every ball
$B\Subset D$, Fubini's theorem and the mean-value identity give
\begin{equation*}
 \int_{\partial B}u_\mu(y)\d\omega_x^B(y)
 =\int_F\left(\int_{\partial B}K_D(y,a)\d\omega_x^B(y)\right)
                     \d\mu(a)=u_\mu(x)\quad(x\in B).
\end{equation*}
Thus $u_\mu$ is harmonic, without needing derivative estimates for
the kernel. Normalization and Tonelli's theorem give
\begin{equation}\label{eq:polar-variation-bound}
 \int_{\partial D_n}|u_\mu(y)|\d\omega_{x_0}^{D_n}(y)
 \leq\int_F\nu_n^a(\overline\Omega)\d|\mu|(a)=|\mu|(F).
\end{equation}
For any test function in \eqref{eq:bounded-lipschitz-distance}, Fubini's theorem and
\eqref{eq:polar-kernel-normalization} give
\begin{equation*}
 \begin{split}
 \left|\int f\d\lambda_n^{u_\mu}-\int f\d\mu\right|
 &=\left|\int_F\int_{\partial D_n}
        (f(y)-f(a))K_D(y,a)\d\omega_{x_0}^{D_n}(y)\d\mu(a)\right|\\
 &\leq\int_F\int_{\partial D_n}|y-a|K_D(y,a)
                    \d\omega_{x_0}^{D_n}(y)\d|\mu|(a).
 \end{split}
\end{equation*}
Lemma \ref{lem:polar-first-moment} proves \eqref{eq:polar-recovery-error},
as well as its version with $\omega_{x_0}^{D_n}(E_n)$.
For data satisfying \eqref{eq:polar-data-error}, every bounded
measurable function $f$ satisfies
\begin{equation*}
 \left|\int f\d\widetilde\lambda_n-\int f\d\lambda_n^{u_\mu}\right|
 \leq\|f\|_\infty\int_{\partial D_n}|v_n-u_\mu|\d\omega_{x_0}^{D_n}
 \leq\|f\|_\infty\varepsilon_n.
\end{equation*}
For the tests in \eqref{eq:bounded-lipschitz-distance}, this adds at
most $\varepsilon_n$ to the error bound and proves
\eqref{eq:polar-recovery-data-error}, including its version with
$\omega_{x_0}^{D_n}(E_n)$. These estimates hold for every compact
$F$ under consideration.

Suppose now that $\operatorname{Cap}(F)=0$, so $t_n,c_n\to0$.
For $f\in C(\overline\Omega)$ and any Lipschitz function $g$, put
$M_g=\max\{\|g\|_\infty,\operatorname{Lip}(g)\}$. Then
\eqref{eq:polar-variation-bound} gives
\begin{equation*}
 \left|\int f\d\lambda_n^{u_\mu}-\int f\d\mu\right|
 \leq2|\mu|(F)\|f-g\|_\infty
      +M_g\,d_{\mathrm{BL}}(\lambda_n^{u_\mu},\mu).
\end{equation*}
Lipschitz functions are uniformly dense in $C(\overline\Omega)$.
First let $n\to\infty$ for fixed $g$, and then choose $g$ uniformly
close to $f$. This proves weak convergence. If also
$\varepsilon_n\to0$, the preceding data-error inequality, now
applied to any $f\in C(\overline\Omega)$, proves the same convergence
for $\widetilde\lambda_n$.
\end{proof}

\subsubsection{Recovery of the source in a Poisson equation}

The normalization can be removed, and a harmonic background need not be
known. This makes the statement applicable to a source measure in a
Poisson equation.

\begin{cor}[Source recovery with a harmonic background]\label{cor:polar-source-recovery}
Let $F$ be any compact set under consideration, let $\sigma$ be a
finite signed Borel measure supported on $F$, and suppose
$u\in L^1_{\mathrm{loc}}(\Omega)$ satisfies
$-\Delta u=\sigma$ in the sense of distributions. Then there is a
harmonic function $v$ on $\Omega$ such that, using the harmonic
representative of $u$ on $D$,
\begin{equation}\label{eq:polar-source-decomposition}
 u(x)=v(x)+\int_Fw(a)K_D(x,a)\d\sigma(a)\quad(x\in D),
\end{equation}
where $w$ is the positive smooth function in
\eqref{eq:singular-normalization-weight}.
Choose $\chi\in C_c^\infty(\Omega\setminus\{x_0\})$ with $0\leq\chi\leq1$
and $\chi=1$ on a neighborhood of $F$, and define
\begin{equation}\label{eq:polar-source-reconstruction}
 b(y)=\frac{\chi(y)}{w(y)},\qquad
 \d\sigma_n(y)=b(y)u(y)\d\omega_{x_0}^{D_n}(y).
\end{equation}
The function $b$ is set to zero where $\chi$ vanishes near $x_0$ and
near $\partial\Omega$. For all sufficiently large $n$,
\begin{equation}\label{eq:polar-source-error}
 d_{\mathrm{BL}}(\sigma_n,\sigma)
 \leq C|\sigma|(F)\bigl(t_n+c_n^{1/(d-2)}\bigr)
       +C\|v\|_{C(\operatorname{supp}\chi)}c_n.
\end{equation}
Replacing $u$ by real measurable data $u_n$ on $\partial D_n$ with
$\int_{\partial D_n}|u_n-u|\d\omega_{x_0}^{D_n}\leq\varepsilon_n<\infty$
for every $n\geq1$ adds at most $\|b\|_\infty\varepsilon_n$ to
the right side. Here $\varepsilon_n\geq0$. The constant and the
starting index are independent of $\sigma$, $u$, $v$, and the data
errors; they may depend on the fixed cutoff $\chi$.
The estimate also holds with $c_n$ replaced throughout by
$\omega_{x_0}^{D_n}(E_n)$. If $\operatorname{Cap}(F)=0$, then
$\sigma_n\Longrightarrow\sigma$; the same conclusion holds for
the observed-data measures when also $\varepsilon_n\to0$.
\end{cor}
\begin{proof}
Put $s(x)=\int_Fw(a)K_D(x,a)\d\sigma(a)$ wherever this integral
is absolutely convergent, assigning $K_D(a,a)=+\infty$ in this
integral. For a compact $L\subset\Omega$, choose
$R>0$ with $|x-a|\leq R$ for $x\in L$ and $a\in F$. By
\eqref{eq:singular-kernel-dirichlet} and Tonelli's theorem,
\begin{equation*}
 \int_L\int_F w(a)K_D(x,a)\d|\sigma|(a)\d x
 \leq\gamma_d|\sigma|(F)\int_{B(0,R)}|z|^{2-d}\d z<\infty.
\end{equation*}
Thus $s$ is defined almost everywhere and belongs to
$L^1_{\mathrm{loc}}(\Omega)$; its values on the remaining null set
may be chosen arbitrarily. On $D$ the integral is defined everywhere
and harmonic by the preceding theorem. For
$\varphi\in C_c^\infty(\Omega)$, Fubini's theorem and the distributional
identity in \eqref{eq:singular-kernel-dirichlet} give
\begin{equation*}
 \begin{split}
 \int_\Omega s(x)(-\Delta\varphi(x))\d x
 &=\int_F w(a)\left(\int_\Omega K_D(x,a)(-\Delta\varphi(x))\d x\right)
                  \d\sigma(a)\\
 &=\int_F\varphi(a)\d\sigma(a).
 \end{split}
\end{equation*}
Thus $u-s$ is distributionally harmonic. Weyl's lemma gives the
harmonic function $v$ in \eqref{eq:polar-source-decomposition}.

Define $\d\mu(a)=w(a)\d\sigma(a)$. Then
$u-v=u_\mu$ and $|\mu|(F)\leq\|w\|_{C(F)}|\sigma|(F)$.
The function $b$ is smooth with compact support away from $x_0$ and
$\partial\Omega$, so it is bounded and Lipschitz on $\overline\Omega$.
Moreover, $b(a)w(a)=1$ on $F$, whence $b\mu=\sigma$.
For a test function $f$ in \eqref{eq:bounded-lipschitz-distance}, the function $bf$ has supremum
norm at most $\|b\|_\infty$ and Lipschitz constant at most
$\|b\|_\infty+\operatorname{Lip}(b)$. Therefore
\begin{equation*}
 d_{\mathrm{BL}}(b\lambda_n^{u_\mu},\sigma)
 \leq(\|b\|_\infty+\operatorname{Lip}(b))d_{\mathrm{BL}}(\lambda_n^{u_\mu},\mu).
\end{equation*}
For large $n$, every point of $\partial D_n\cap\operatorname{supp}\chi$ belongs to
$E_n$, because $\operatorname{supp}\chi$ has positive distance from $\partial\Omega$.
Consequently
\begin{equation*}
 \int_{\partial D_n}|b(y)v(y)|\d\omega_{x_0}^{D_n}(y)
 \leq C\|v\|_{C(\operatorname{supp}\chi)}\omega_{x_0}^{D_n}(E_n)
 \leq C\|v\|_{C(\operatorname{supp}\chi)}c_n.
\end{equation*}
Combining this inequality with Theorem \ref{thm:polar-measure-recovery} proves
\eqref{eq:polar-source-error}. Keeping
$\omega_{x_0}^{D_n}(E_n)$ in the preceding inequality and in
Theorem \ref{thm:polar-measure-recovery} gives the other stated bound.
Multiplication by $b$ bounds the variation of the additional data-error
measure by $\|b\|_\infty\varepsilon_n$.

If $\operatorname{Cap}(F)=0$, Theorem
\ref{thm:polar-measure-recovery} gives
$\lambda_n^{u_\mu}\Longrightarrow\mu$. Since $b$ is continuous,
$b\lambda_n^{u_\mu}\Longrightarrow b\mu=\sigma$.
The displayed bound on the harmonic background tends to zero,
proving $\sigma_n\Longrightarrow\sigma$. The data-error bound
gives the same conclusion when $\varepsilon_n\to0$.
\end{proof}

\subsection{Source recovery on rectifiable curve networks}
\label{subsec:curve-networks}

Let $d=3$ and
\begin{equation}\label{eq:curve-network}
 F=\bigcup_{e=1}^m\gamma_e([0,\ell_e])\Subset\Omega,
 \qquad \ell_e>0,
\end{equation}
where each $\gamma_e$ is a nonconstant \textit{rectifiable curve} parametrized by
\textit{arc length}. In particular,
$|\gamma_e(s)-\gamma_e(s')|\leq|s-s'|$ for every
$s,s'\in[0,\ell_e]$ and every $e\in\{1,\ldots,m\}$. Intersections at vertices and
closed curves are allowed. Assume, as before, that $D$ is connected and
$x_0\in D$. Straight line networks are a special case. Harmonicity here is for
the ambient three-dimensional Laplacian; the intrinsic metric graph
system is considered in Section \ref{subsec:finite-metric-graphs}.

\subsubsection{Logarithmic concentration and recovery}

\begin{lem}[A logarithmic harmonic majorant]\label{lem:curve-harmonic-majorant}
Set $\ell_* =\min_e\ell_e$ and
\begin{equation}\label{eq:curve-majorant}
 U_F(x)=\sum_{e=1}^m\int_0^{\ell_e}
           \frac{\d s}{|x-\gamma_e(s)|},\qquad
 L(t)=\log\left(1+\frac{\ell_*}{2t}\right)\quad(t>0).
\end{equation}
The function $U_F$ is positive and harmonic on $D$, finite at $x_0$,
and infinite at every point of $F$. In particular, $F$ is polar. Also
\begin{equation}\label{eq:curve-majorant-bounds}
 q_t:=\frac{U_F}{L(t)}\geq1\quad\hbox{on }F_t\setminus F,
 \qquad q_t(x_0)=\frac{U_F(x_0)}{L(t)}.
\end{equation}
\end{lem}
\begin{proof}
The curve integrals define the Newtonian potential of a finite positive
measure supported on $F$. Its potential is harmonic off $F$, and it is
finite at $x_0$ because $\dist(x_0,F)>0$.
For $x\in F_t\setminus F$, choose a nearest point
$a=\gamma_e(s_0)\in F$. At least one of the intervals from $s_0$ to an
endpoint of $[0,\ell_e]$ has length at least $\ell_e/2$.
Along that direction, for $0\leq s\leq\ell_*/2$,
\begin{equation*}
 |x-\gamma_e(s_0\pm s)|\leq|x-a|+
             |\gamma_e(s_0)-\gamma_e(s_0\pm s)|\leq t+s.
\end{equation*}
It follows that
\begin{equation*}
 U_F(x)\geq\int_0^{\ell_*/2}\frac{\d s}{t+s}
         =\log\left(1+\frac{\ell_*}{2t}\right).
\end{equation*}
For $x=a\in F$, the same argument gives the divergent integral
$\int_0^{\ell_*/2}s^{-1}\d s$. Thus the nonnegative superharmonic
function $U_F$, which is not identically infinite, is infinite on $F$.
This is the defining potential-theoretic criterion for polarity.
\end{proof}

\begin{thm}[Logarithmic recovery on curve networks]\label{thm:curve-measure-recovery}
For $F$ as in \eqref{eq:curve-network} and every exhaustion
\eqref{eq:polar-exhaustion}, the concentration estimates
\eqref{eq:polar-outer-tail}--\eqref{eq:polar-inner-tail} hold with $c_n$ replaced
by $L(t_n)^{-1}$. For every finite signed measure $\mu$ on $F$,
\begin{equation}\label{eq:curve-recovery-error}
 d_{\mathrm{BL}}(u_\mu\omega_{x_0}^{D_n},\mu)
 \leq C|\mu|(F)\bigl(t_n+L(t_n)^{-1}\bigr)
 \leq\frac{C'|\mu|(F)}{\log(1/t_n)}
\end{equation}
for all sufficiently large $n$. Under \eqref{eq:polar-data-error}, the same bound
holds for the observed measures after adding $\varepsilon_n$.
The source reconstruction of Corollary \ref{cor:polar-source-recovery} satisfies
\begin{equation}\label{eq:curve-source-error}
 d_{\mathrm{BL}}(\sigma_n,\sigma)
 \leq\frac{C\bigl(|\sigma|(F)+\|v\|_{C(\operatorname{supp}\chi)}\bigr)}
              {\log(1/t_n)},
\end{equation}
for all sufficiently large $n$, with the same additional error term
as in that corollary. The constants may depend on the fixed curve
network, $\Omega$, and $x_0$, and in the source estimate on $\chi$.
They and the starting index are independent of the source measures
and data errors. The starting index may depend on the exhaustion;
the constants do not.
\end{thm}
\begin{proof}
The mean-value identity for the harmonic function $q_{t_n}$ gives
\begin{equation}\label{eq:curve-En-measure}
 \omega_{x_0}^{D_n}(E_n)
 \leq\omega_{x_0}^{D_n}(q_{t_n})
 =q_{t_n}(x_0)=\frac{U_F(x_0)}{L(t_n)}.
\end{equation}
Substitute \eqref{eq:curve-En-measure} into
\eqref{eq:polar-measure-tail-bounds} to obtain the two concentration
estimates. Theorem \ref{thm:polar-measure-recovery}, in its form with
$\omega_{x_0}^{D_n}(E_n)$, gives \eqref{eq:curve-recovery-error},
including the data error, since $1/(d-2)=1$ here.
Corollary \ref{cor:polar-source-recovery}, in the same form, gives
\eqref{eq:curve-source-error}. Finally,
$L(t)/\log(1/t)\to1$ and $t\log(1/t)\to0$, so
$t+L(t)^{-1}\leq C/\log(1/t)$ for all sufficiently small $t$.
No further first moment calculation is needed.
\end{proof}

The networks in \eqref{eq:curve-network} do not satisfy the covering
condition $tN(F,t)\to0$ in Proposition \ref{prop:polar-capacity-cover}. Indeed, a nonconstant connected curve has a projection onto
some line containing an interval of positive length $a_0$. Any cover by
$N$ balls of radius $t$ then gives $2Nt\geq a_0$. Conversely, subdivision
of each arc-length parameter interval into intervals of length less than
$t$ gives $N(F,t)\leq C/t$ for small $t$. Hence
\begin{equation}\label{eq:curve-covering-order}
 0<\liminf_{t\downarrow0}tN(F,t)
 \leq\limsup_{t\downarrow0}tN(F,t)<\infty.
\end{equation}
Thus $tN(F,t)\to0$ fails already for a single segment, whereas
Theorem \ref{thm:curve-measure-recovery} applies to all the networks above.

Measures of the form
\begin{equation*}
 \sigma=\sum_e f_e\,\mathcal H^1|_{\gamma_e([0,\ell_e])}
              +\sum_{j=1}^{k}\alpha_j\delta_{a_j},
 \qquad f_e\in L^1(\mathcal H^1|_{\gamma_e([0,\ell_e])}),\quad \alpha_j\in\R,\ a_j\in F,
\end{equation*}
are included, as are arbitrary finite signed measures on $F$.
The notation $\mathcal H^1$ denotes \textit{one-dimensional Hausdorff measure}.
Equations with three-dimensional spatial variables and line sources are
established analytical models. For example, \cite{Gjerde19} studies their
singularity structure and finite element approximation. Here the object
recovered is the source measure, using function values on the approaching
boundaries and known harmonic measures. Normal derivatives of the data
are not required. The result concerns a known set $F$ and data on
$\partial D_n$ approaching it; it is not a theorem about recovering an
unknown network from a single fixed exterior boundary.

\subsubsection{Sharpness of the logarithmic order}\label{subsubsec:curve-sharpness}

The logarithmic scale agrees with the classical capacity formula for
thin prolate ellipsoids; see \cite[Section 4, equations (68)--(70)]{vandenBerg25}.
The following example gives a lower bound for the error of the specific
approximation $u_\mu\omega_{x_0}^{D_n}$.

\begin{ex}[A segment and confocal ellipsoids]\label{ex:curve-logarithmic-sharpness}
Let $\Omega=B(0,1)\subset\R^3$, $a=1/8$, $e_1=(1,0,0)$,
$F=[-a,a]e_1$, and $x_0=(1/2,0,0)$. Choose a strictly decreasing sequence
$0<\epsilon_n<1/16$ tending to zero. Define
\begin{equation}\label{eq:curve-ellipsoids}
 \begin{split}
 A_\epsilon&=\sqrt{a^2+\epsilon^2},\\
 Q_\epsilon&=\left\{x:\frac{x_1^2}{a^2+\epsilon^2}
                    +\frac{x_2^2+x_3^2}{\epsilon^2}\leq1\right\},\\
 D_n&=B(0,1-\epsilon_n)\setminus Q_{\epsilon_n}.
 \end{split}
\end{equation}
These domains form a smooth exhaustion of $D$, and $t_n=\epsilon_n$.
For the normalized kernel at $0$,
\begin{equation}\label{eq:curve-sharp-kernel}
 h_0(x)=K_D(x,0)=|x|^{-1}-1,
\end{equation}
there are constants $c,C>0$ such that
\begin{equation}\label{eq:curve-sharp-recovery}
 \frac{c}{\log(1/t_n)}
 \leq d_{\mathrm{BL}}(h_0\omega_{x_0}^{D_n},\delta_0)
 \leq\frac{C}{\log(1/t_n)}
\end{equation}
for every sufficiently large $n$.
\end{ex}
\begin{proof}
The ellipsoids decrease strictly to $F$, and the outer balls increase
strictly to $B(0,1)$. Since $Q_{\epsilon_n}\subset B(0,1/4)$, the domains
are connected, contain $x_0$, and satisfy \eqref{eq:polar-exhaustion}.
The distance from a point on the outer sphere to $\partial D$ is
$\epsilon_n$. A point on the ellipsoid can be written, after rotation
about its first axis, with axial coordinate $A_\epsilon s$ and transverse
length $\epsilon\sqrt{1-s^2}$, where $-1\leq s\leq1$.
If $|A_\epsilon s|\leq a$, its distance to $F$ is at most $\epsilon$.
For $s\geq a/A_\epsilon$, the square of the distance to the endpoint
$ae_1$ is
\begin{equation*}
 (A_\epsilon s-a)^2+\epsilon^2(1-s^2)
 =\epsilon^2+a^2s^2-2aA_\epsilon s+a^2\leq\epsilon^2.
\end{equation*}
The last quadratic term is decreasing on $[a/A_\epsilon,1]$ and is
negative at its left endpoint. The case $s<0$ is symmetric. Equality
of the distance to $\epsilon$ occurs at $(0,\epsilon,0)$. Therefore
$t_n=\epsilon_n$.

Consider the harmonic function
\begin{equation*}
 U(x)=\int_{-a}^a\frac{\d s}{|x-se_1|}\quad(x\notin F).
\end{equation*}
Direct integration, writing $r_\pm=|x\mp ae_1|$, gives
\begin{equation}\label{eq:curve-segment-potential}
 U(x)=\log\frac{r_++r_-+2a}{r_++r_--2a}.
\end{equation}
The identity $r_++r_-=2A_\epsilon$ on $\partial Q_\epsilon$ then yields
\begin{equation}\label{eq:curve-Lambda}
 U=\Lambda_\epsilon\quad\hbox{on }\partial Q_\epsilon,
 \qquad
 \Lambda_\epsilon=\log\frac{A_\epsilon+a}{A_\epsilon-a}
                 =2\log\frac{A_\epsilon+a}{\epsilon}.
\end{equation}
In particular, $\Lambda_\epsilon$ is comparable to $\log(1/\epsilon)$.

On the outer sphere $|x|=1-\epsilon_n\geq7/8$,
\begin{equation*}
 0<U(x)\leq\frac{2a}{7/8-a}=\frac13=:M,
 \qquad
 U(x_0)\geq\frac{2a}{|x_0|+a}=\frac25>M.
\end{equation*}
Let
$b_n(x)=\omega_x^{D_n}(\partial Q_{\epsilon_n})$.
It is the continuous harmonic function equal to $1$ on the inner
ellipsoid and $0$ on the outer sphere. On each component of
$\partial D_n$, the boundary inequalities needed for the maximum
principle give
\begin{equation}\label{eq:curve-bn-comparison}
 \frac{U(x)-M}{\Lambda_{\epsilon_n}}
       \leq b_n(x)\leq\frac{U(x)}{\Lambda_{\epsilon_n}}
                  \quad(x\in D_n).
\end{equation}
Indeed, the left function is at most $1$ on the inner component and
at most $0$ on the outer component; the right function is $1$ on the
inner component and nonnegative on the outer component. Evaluating at
$x_0$ gives
\begin{equation}\label{eq:curve-bn-log}
 \frac{1}{15\Lambda_{\epsilon_n}}
 \leq\omega_{x_0}^{D_n}(\partial Q_{\epsilon_n})
 \leq\frac{U(x_0)}{\Lambda_{\epsilon_n}}.
\end{equation}

For the kernel indexed by the origin, the boundary values of
$\Phi_0$ on the unit sphere are the constant $\gamma_3$. Thus
$\omega_x^\Omega(\Phi_0)=\gamma_3$ and
$w(0)=\gamma_3(|x_0|^{-1}-1)=\gamma_3$.
Equation \eqref{eq:singular-kernel-dirichlet} gives
\eqref{eq:curve-sharp-kernel} directly. The function $f(x)=|x|$ on $\overline B$
satisfies $\|f\|_\infty\leq1$, $\operatorname{Lip}(f)\leq1$, and $f(0)=0$.
On the inner ellipsoid, $|x|<1/4$, and therefore
$f(x)h_0(x)=1-|x|\geq3/4$. Consequently
\begin{equation*}
 \begin{split}
 d_{\mathrm{BL}}(h_0\omega_{x_0}^{D_n},\delta_0)
 &\geq\int_{\partial D_n}|x|h_0(x)\d\omega_{x_0}^{D_n}(x)\\
 &\geq\frac34\omega_{x_0}^{D_n}(\partial Q_{\epsilon_n})
 \geq\frac{1}{20\Lambda_{\epsilon_n}}.
 \end{split}
\end{equation*}
This proves the lower bound in \eqref{eq:curve-sharp-recovery}; the upper bound
is Theorem \ref{thm:curve-measure-recovery} with $\mu=\delta_0$.
\end{proof}

This example rules out an $o(1/\log(1/t_n))$ bound valid for all
finite measures and all the exhaustions in
Theorem \ref{thm:curve-measure-recovery}, even for a single fixed segment and a positive
measure. The approximation formula $u_\mu\omega_{x_0}^{D_n}$ is fixed
throughout this sharpness statement. The assertion is for the class of
finite measures, which includes point measures; it is not an assertion
about an optimal rate for every prescribed smooth density along the
segment.

\subsection{Splitting a disk boundary point}
\label{subsec:split-disk}
The following example separates neighborhood nonvanishing from point
concentration, and lifted representation from uniqueness after projection.
All computations use the classical disk Poisson kernel.

Let $D=\mathbb{D}$, $x_0=0$, and $D_n=r_n\mathbb{D}$, where
$0<r_n\uparrow1$ strictly.  Put
\begin{equation*}
 h(z)=\frac{1-|z|^2}{|1-z|^2},\qquad
 a(z)=\arctan\frac{\operatorname{Im}z}{1-\operatorname{Re}z},\qquad
 j(z)=(z,a(z)).
\end{equation*}
The function $h$ is the normalized minimal function associated with
$1\in\partial\mathbb{D}$.  Define
\begin{equation*}
 \widehat D=\overline{j(D)}
 \subset\overline{\mathbb{D}}\times[-\pi/2,\pi/2],\qquad p(z,t)=z.
\end{equation*}
The graph embedding identifies $D$ with an open dense subset of the
compact metric space $\widehat D$.  Away from $1$, $a$ extends
continuously to the circle, so each boundary fiber of $p$ is a singleton
except possibly
\begin{equation*}
 F=p^{-1}(1)=\{1\}\times[-\pi/2,\pi/2].
\end{equation*}
The sequences constructed below verify this equality, including its
endpoints.  Write $\xi_\alpha=(1,\alpha)$ and assign
$h_{\xi_\alpha}=h$ for every $\alpha\in[-\pi/2,\pi/2]$.

\begin{prop}[Nonvanishing without point concentration]
\label{prop:split-disk}
The disk exhaustion satisfies {\rm(A1)--(A2)} and {\rm(M)}.
For each $\xi_\alpha$, take $D^K=\widehat D$ in $(K_\xi)$.
Then $(C_{\xi_\alpha})$ holds and synchronized approximation is possible
along every sufficiently late inner circle.  Nevertheless, the measures
\begin{equation*}
 \widehat\nu_n=j_\#\bigl(h\,\omega_0^{D_n}\bigr)
\end{equation*}
satisfy
\begin{equation*}
 \widehat\nu_n\Longrightarrow\lambda_F\quad\text{weakly on }\widehat D,
 \qquad \lambda_F\ne\delta_{\xi_\alpha}\quad(\xi_\alpha\in F),
\end{equation*}
where $\lambda_F$ is normalized length measure on $F$.
Representing measures $\widehat\mu$ for $h$ on $F$ are nonunique,
but all satisfy $p_\#\widehat\mu=\delta_1$ on the original minimal boundary.
\end{prop}

\begin{proof}
The inner circles are smooth and their normalized Poisson kernels are
strictly positive, jointly continuous, and minimal.  Thus (A1)--(A2) and
(M) hold.  Write $y=re^{i\theta}$, with $r=r_n$.  Since
$d\omega_0^{D_n}=d\theta/(2\pi)$, the inner measure has density
\begin{equation*}
 d\nu_n(y)=\frac{1-r^2}{1-2r\cos\theta+r^2}\frac{d\theta}{2\pi}.
\end{equation*}
The substitution
\begin{equation*}
 t=\frac{1+r}{1-r}\tan\frac\theta2,\qquad
 \theta_r(t)=2\arctan\left(\frac{1-r}{1+r}t\right)
\end{equation*}
can be checked by writing $b=(1-r)/(1+r)$.  Then
\begin{equation*}
 \tan(\theta/2)=bt,\quad
 d\theta=\frac{2b}{1+b^2t^2}\,dt,\quad
 1-2r\cos\theta+r^2
 =\frac{(1-r)^2(1+t^2)}{1+b^2t^2}.
\end{equation*}
Substitution gives the exact identity
\begin{equation*}
 \frac{1-r^2}{1-2r\cos\theta+r^2}\frac{d\theta}{2\pi}
 =\frac{dt}{\pi(1+t^2)}.
\end{equation*}
For each fixed real $t$, $re^{i\theta_r(t)}\to1$, and
\begin{equation*}
 \frac{r\sin\theta_r(t)}{1-r\cos\theta_r(t)}
 =\frac{2rt/(1+r)}{1+\frac{1-r}{1+r}t^2}\longrightarrow t.
\end{equation*}
Hence $j(re^{i\theta_r(t)})\to(1,\arctan t)$.
For $\varphi\in C(\widehat D)$, dominated convergence yields
\begin{equation}\label{eq:split-disk-limit}
 \begin{aligned}
 \int\varphi\d\widehat\nu_n
 &\longrightarrow\int_{\mathbb{R}}
 \varphi(1,\arctan t)\frac{dt}{\pi(1+t^2)}\\
 &=\frac1\pi\int_{-\pi/2}^{\pi/2}\varphi(1,\alpha)\,d\alpha
 =\int\varphi\,d\lambda_F.
 \end{aligned}
\end{equation}
Every open neighborhood $V$ of any $\xi_\alpha$ has
$\lambda_F(V)>0$.  Portmanteau therefore gives
$\liminf_n\widehat\nu_n(V)\ge\lambda_F(V)>0$, proving
$(C_{\xi_\alpha})$.  On the other hand,
$\lambda_F\ne\delta_{\xi_\alpha}$ for every $\alpha$.

For explicit synchronized selection, if $|\alpha|<\pi/2$, set
\begin{equation*}
 \theta_n=2\arctan\left(\frac{1-r_n}{1+r_n}\tan\alpha\right),
 \qquad y_n=r_ne^{i\theta_n}.
\end{equation*}
Then $j(y_n)\to\xi_\alpha$.  At the two endpoints take instead
$\theta_n=\pm\sqrt{1-r_n}$.  Since
\begin{equation*}
 \frac{r_n\sin\theta_n}{1-r_n\cos\theta_n}\longrightarrow\pm\infty,
\end{equation*}
these sequences tend to $\xi_{\pm\pi/2}$.  They also prove the asserted
fiber description.  In every case $y_n\to1$, and
\begin{equation}\label{eq:split-disk-kernels}
 K_{D_n}(z,y_n)=\frac{r_n^2-|z|^2}{|y_n-z|^2}
 \longrightarrow\frac{1-|z|^2}{|1-z|^2}=h(z)
\end{equation}
locally uniformly in $D$, because the denominators are bounded away
from zero on each fixed compact subset for all sufficiently large $n$.

Finally, on the refined boundary define
\begin{equation*}
 \widehat K(z,\eta)=\frac{1-|z|^2}{|p(\eta)-z|^2}.
\end{equation*}
This equals $h(z)$ on $F$.  Thus every positive Borel measure
$\sigma$ carried by $F$ with $\sigma(F)=1$ represents $h$:
\begin{equation*}
 h(z)=\int_F\widehat K(z,\eta)\d\sigma(\eta),\qquad
 p_\#\sigma=\delta_1.
\end{equation*}
In particular, different Dirac measures on $F$ and $\lambda_F$ give
different lifted representations.  On the original circle, the disk
kernels supply the boundary data (B), and the projected measures converge
to $\delta_1$, as follows either from the calculation above or from
Lemma \ref{lem:euclidean-minimal-concentration}.  Corollary
\ref{thm:common-refinement-representation}, with this circle as the
minimal boundary and $p$ as the gluing map, proves uniqueness of the
projected representation.

The set of representing measures on $F$ must be distinguished from the
limits of the specified inner-boundary measures.  For the radial
exhaustion used here, the full sequence
$\widehat\nu_n\Longrightarrow\lambda_F$, so every convergent
subsequence has that same limit.  In particular,
$\delta_{\xi_\alpha}$ represents $h$ but is not a subsequential limit
of this sequence, since $\lambda_F\ne\delta_{\xi_\alpha}$.
Thus nonuniqueness of refined representation does not imply
nonuniqueness of the canonical limit for a fixed exhaustion.
\end{proof}

The assignment on the full refined boundary satisfies (B), but it is
not injective: every point of $F$ has the same kernel. Consequently
$(C')$ cannot hold there, by Lemma
\ref{lem:automatic-boundary-measurability}; the explicit limit above
also shows this directly. This causes no conflict with Theorem
\ref{thm:synchronized-kernel}, which uses only the single-point data
$(K_\xi)$.  For representation, (B) is used on the original circle,
and the split compactification is the common refinement of itself
and the original circle compactification, with projection $p$.  In the original compactification
one has, for every open neighborhood $V$ of $1$,
\begin{equation*}
 (p_\#\widehat\nu_n)(\overline{\mathbb D}\setminus V)
 =\widehat\nu_n\bigl(\widehat D\setminus p^{-1}(V)\bigr)
 \longrightarrow0,
\end{equation*}
equivalently $p_\#\widehat\nu_n\Longrightarrow\delta_1$.
Thus the distinction between nonvanishing and point concentration occurs
in the same refined compactification, whereas uniqueness belongs to the
projected minimal boundary.

\subsection{Verification for finite weighted metric graphs}
\label{subsec:finite-metric-graphs}

We give a finite graph model in which all the assumptions of Section
\ref{Sec4}, including global concentration (C), can be verified directly.
The underlying space includes the edge interiors.  The topology on a
finite vertex set alone would be discrete, and its adjacency boundary
would not be a topological boundary.

\subsubsection{Metric realization and the Dirichlet problem}

Let $G=(V,E)$ be a finite connected graph with positive symmetric edge
weights $\mu_e$.  Realize each edge $e$ as an interval of length
$\ell_e:=\mu_e^{-1}$, and identify its endpoints according to $G$.
Write $X$ for the resulting compact \textit{metric graph}, equipped with its
\textit{path metric}.  Subdivision of edges is allowed.  A function is \textit{harmonic}
on an open subset of $X$ if it is continuous, affine on each edge
segment, and satisfies the \textit{Kirchhoff condition} at every vertex in that
open set:
\begin{equation}\label{eq:graph-kirchhoff}
 \sum_{e\sim v}\partial_e u(v)=0.
\end{equation}
Here $\partial_e u(v)$ is the derivative along $e$, directed away from
$v$, and incident edge ends are counted separately.  For an entire edge
$e=vw$, its contribution is
\begin{equation*}
 \partial_e u(v)=\frac{u(w)-u(v)}{\ell_e}
 =\mu_e\bigl(u(w)-u(v)\bigr).
\end{equation*}
Thus \eqref{eq:graph-kirchhoff} is exactly the weighted discrete
harmonicity condition used in the original vertex description.  At a
vertex of degree one which lies in the open set it is the zero-derivative
condition.

The length convention also explains subdivision.  If an edge of
conductance $\mu_e$ is divided into lengths
$\lambda\ell_e$ and $(1-\lambda)\ell_e$, where $0<\lambda<1$,
the new conductances are
$\mu_e/\lambda$ and $\mu_e/(1-\lambda)$.  For
\begin{equation*}
 u(x_\lambda)=(1-\lambda)u(v)+\lambda u(w)
\end{equation*}
we have
\begin{equation*}
 \frac{\mu_e}{\lambda}\bigl(u(v)-u(x_\lambda)\bigr)
 =\mu_e\bigl(u(v)-u(w)\bigr).
\end{equation*}
This identity proves preservation of the vertex equation, including
the case $u(v)=u(w)$; no division by a difference of function values
is needed.

\begin{lem}[Finite graph Dirichlet problem]\label{lem:metric-graph-dirichlet}
Let $U$ be a connected open subset of $X$ with finite nonempty
topological boundary.  Every function $f:\partial U\to\mathbb{R}$
has a unique continuous extension to $\overline U$ which is harmonic
in $U$.  The solution preserves order and constants, satisfies the
maximum principle, and attains its boundary data continuously.
\end{lem}

\begin{proof}
Cut the finitely many edges at the boundary points and, when several
edge ends approach the same boundary vertex, treat those ends separately
and prescribe the same value $f$ on each of them.  Subdivide further
if necessary.  There are finitely many unknown interior vertex values.
After linear interpolation on each edge segment, the Kirchhoff equations
have the form $Az=y$, where
\begin{equation*}
 A_{ii}=k_i+\sum_j a_{ij},\qquad A_{ij}=-a_{ij}\quad(i\ne j).
\end{equation*}
Here, for $i\ne j$, $a_{ij}=a_{ji}\geq0$ is the total conductance
between interior vertices $i,j$; put $a_{ii}=0$.  An edge whose two
ends are the same vertex has zero contribution to its vertex equation.
The number $k_i\geq0$ is the total conductance from vertex
$i$ to prescribed boundary vertices.  For every real vector $z$,
\begin{equation}\label{eq:graph-energy-identity}
 z^TAz=\sum_i k_i z_i^2+
       \frac12\sum_{i,j}a_{ij}(z_i-z_j)^2.
\end{equation}
Every connected component of the interior vertex network is joined to
the boundary.  If the right-hand side vanishes, then $z_i=z_j$ whenever
$a_{ij}>0$, and $z_i=0$ whenever $k_i>0$.  Each component has
a path to a vertex with $k_i>0$; equality along that path forces
all its entries to be zero.  Thus $z=0$.
The matrix $A$ is positive definite, so the system has a unique
solution.  A segment with no interior vertex is determined directly by
its endpoint values.

An interior maximum of a harmonic function propagates along incident
edges: each outgoing derivative is nonpositive, and their sum is zero,
so all are zero.  Connectedness then propagates the maximum to the whole
domain and to its boundary.  This proves the maximum principle.
If two boundary data satisfy $f\le g$, their solution difference
has nonpositive boundary values.  A positive interior maximum would
propagate to the boundary, a contradiction; hence their solutions are
ordered.  A constant function satisfies every edge and vertex equation
and has the prescribed constant boundary data, so uniqueness proves
preservation of constants.  The affine extensions reach the assigned
boundary values, proving regularity.
\end{proof}

For such $U$ and $b\in\partial U$, denote by $H_b^U$ the solution
with boundary values $\chi_{\{b\}}$.  Define
\begin{equation}\label{eq:graph-harmonic-measure}
 \omega_x^U:=\sum_{b\in\partial U}H_b^U(x)\,\delta_b,\qquad x\in U.
\end{equation}
The maximum principle gives
\begin{equation*}
 H_b^U(x)>0\quad(\forall x\in U,\ \forall b\in\partial U),\qquad
 \sum_{b\in\partial U}H_b^U(x)=1\quad(\forall x\in U).
\end{equation*}
Indeed, an interior zero of a nonnegative solution forces it to vanish
identically, contradicting its boundary value at $b$.

\begin{prop}[The induced harmonic measure system]
\label{prop:metric-graph-system}
The measures \eqref{eq:graph-harmonic-measure}, for connected open sets
with finite nonempty boundary, form a regular closed Radon harmonic
measure system.  The induced harmonic functions of Definition
\ref{def:system-harmonic} are precisely the continuous edgewise affine
functions satisfying \eqref{eq:graph-kirchhoff}.
\end{prop}

\begin{proof}
Each measure satisfies $\omega_x^U(\partial U)=1$ and $\omega_x^U(\{b\})>0$ for every
boundary point.  Consequently its null sets do not depend on the
interior point.  All boundary functions are integrable, and their
extensions are continuous by Lemma \ref{lem:metric-graph-dirichlet}.
The measures are Radon because the boundary is finite.

To verify (N), let $V\Subset U$ be two domains of this family
and let $f$ be a function on the finite boundary $\partial U$.
The function $v(x)=\omega_x^U(f)$ is harmonic in $U$ and
continuous on $\overline V$.  Uniqueness of the Dirichlet problem
on $V$ gives
\begin{equation*}
 \omega_x^U(f)=v(x)
 =\int_{\partial V}v(y)\,\d\omega_x^V(y)
 =\int_{\partial V}\omega_y^U(f)\,\d\omega_x^V(y).
\end{equation*}
This is (N); all boundary functions here are bounded.
Regularity was proved in Lemma \ref{lem:metric-graph-dirichlet}.
To exhaust any such domain, split its boundary into incident edge ends
and remove successively shorter terminal segments.  Sufficiently short
truncations are connected, their closures are compactly contained in
the domain, and their union is the domain.  Each is regular, so the
system is regular closed.

Every Kirchhoff harmonic function satisfies the local mean-value
identities by uniqueness of the Dirichlet problem.  Conversely, apply
those identities first on intervals inside edges: they imply affine
interpolation.  Apply them next on a small star about each vertex:
let the incident arms have lengths $\ell_i$ and endpoint values
$a_i$.  The Dirichlet solution has central value
\begin{equation*}
 c=\frac{\sum_i a_i/\ell_i}{\sum_i1/\ell_i},
\end{equation*}
because $\sum_i(a_i-c)/\ell_i=0$.  The mean-value identity at the
center gives $u(v)=c$, so
$\sum_i(a_i-u(v))/\ell_i=0$, exactly
\eqref{eq:graph-kirchhoff}.  This proves the
claimed agreement of harmonic structures.
\end{proof}

\subsubsection{Verification of the assumptions of Section \ref{Sec4}}

Fix a nonempty set $B$ of degree-one vertices of $X$, assume that
$D:=X\setminus B$ is connected, and fix $x_0\in D$.  Set
\begin{equation*}
 D^K:=X,\qquad
 H_b:=H_b^D,\qquad
 K_D(x,b):=\frac{H_b(x)}{H_b(x_0)}\quad(b\in B).
\end{equation*}
Choose pairwise disjoint terminal edge segments incident with the points
of $B$, avoiding $x_0$ and all other vertices.  Let
$\varepsilon_n\downarrow0$ be strictly decreasing and smaller than all
their lengths.  Delete the closed segment of length $\varepsilon_n$
at each $b\in B$, and call the remaining open set $D_n$.  Denote the
cut point on the edge incident with $b$ by $b_n$.  Then
\begin{equation}\label{eq:graph-exhaustion}
 x_0\in D_n,\qquad D_n\Subset D_{n+1},\qquad
 \bigcup_nD_n=D,\qquad
 \partial D_n=\{b_n:b\in B\}.
\end{equation}

\begin{prop}[Complete verification for finite metric graphs]
\label{prop:graph-all-assumptions}
The exhaustion \eqref{eq:graph-exhaustion} satisfies {\rm(A1)--(A2)}
and {\rm(M)}.  The choice
\begin{equation*}
 D^K=X,\qquad \partial_{\mathrm m}^KD=B,\qquad
 \kappa(b)=K_D(\cdot,b)
\end{equation*}
gives the boundary data {\rm(B)}. For each $b\in B$, put
$\nu_n^b:=K_D(\cdot,b)\,\omega_{x_0}^{D_n}$.
For every open neighborhood $V$ of $b$ in $X$,
\begin{equation}\label{eq:graph-concentration}
 \nu_n^b(X\setminus V)\longrightarrow0
 \quad(n\to\infty).
\end{equation}
Equivalently, $\nu_n^b\Longrightarrow\delta_b$ on $X$.
Thus {\rm(C)}, and consequently $(C_b)$, hold.  Taking $D^G=X$
gives the identity common refinement $\widehat D=X$, with
$p_K=\operatorname{id}_X$.  Any other compact metrizable
compactification $D^G$ gives $\widehat D=D^G\vee X$.
\end{prop}

\begin{proof}
Lemma \ref{lem:metric-graph-dirichlet} and
\eqref{eq:graph-exhaustion} give the regular exhaustion of $D$,
which verifies (A1). For $b_n\in\partial D_n$,
\begin{equation*}
 K_{D_n}(x,b_n)
 =\frac{H_{b_n}^{D_n}(x)}{H_{b_n}^{D_n}(x_0)}.
\end{equation*}
The numerator and denominator are strictly positive in the interior,
and the numerator is continuous in $x$.  Since $\partial D_n$ is
finite and discrete, this kernel is jointly continuous on
$D_n\times\partial D_n$.  This proves (A2).

Every nonnegative harmonic function on $D$ extends continuously to
$X$: it is affine on the terminal edge approaching each $b\in B$,
and therefore has a finite, nonnegative endpoint limit.  Uniqueness in
Lemma \ref{lem:metric-graph-dirichlet} gives
\begin{equation}\label{eq:graph-boundary-expansion}
 u(x)=\sum_{b\in B}u(b)H_b(x)\qquad(\forall x\in D).
\end{equation}
The map taking a nonnegative harmonic function to its boundary vector
$(u(b))_{b\in B}$ is linear, injective by uniqueness, and onto
$\mathbb{R}_+^B$ by the Dirichlet problem and order preservation.
A vector with two positive coordinates splits into two nonproportional
nonnegative vectors, whereas a vector with just one positive coordinate
can only split into multiples of itself.  Thus the extreme rays of
$\mathbb{R}_+^B$, and consequently of the harmonic cone, are precisely
those generated by the coordinate vectors and by $H_b$, respectively.
Every nonzero nonnegative harmonic function on $D$ is strictly positive in $D$.
Consequently
\begin{equation*}
 \operatorname{Ext}\mathcal{H}_1(D)
 =\{K_D(\cdot,b):b\in B\}.
\end{equation*}
These functions are distinct, as is seen from their boundary values.
The map $\kappa$ is therefore a Borel bijection, giving the boundary parametrization (B).
The same argument on $D_n$, whose cut endpoints all have degree one
in $\overline{D_n}$, proves (M).

Fix $b\in B$ and put $h_b:=K_D(\cdot,b)$.
The mean-value identity and normalization give
\begin{equation*}
 \nu_n^b(\partial D_n)
 =\int_{\partial D_n}h_b\,\d\omega_{x_0}^{D_n}
 =h_b(x_0)=1.
\end{equation*}
For $c\ne b$, we have $c_n\to c$ and $h_b(c)=0$.
Continuity on $X$ therefore gives
\begin{equation*}
 0\leq\nu_n^b(\{c_n\})
 =h_b(c_n)\omega_{x_0}^{D_n}(\{c_n\})
 \leq h_b(c_n)\longrightarrow0.
\end{equation*}
There are finitely many such points, so
$\nu_n^b(\{b_n\})\to1$.  More explicitly,
$\nu_n^b(\{b_n\})=1-\sum_{c\ne b}\nu_n^b(\{c_n\})\to1$.
If $V$ is an open neighborhood of $b$ in $X$, then $b_n\in V$
for all sufficiently large $n$.  For these $n$,
\begin{equation*}
 0\le\nu_n^b(X\setminus V)
 \le\sum_{c\ne b}\nu_n^b(\{c_n\})\longrightarrow0.
\end{equation*}
This proves \eqref{eq:graph-concentration}.
The neighborhood criterion is equivalent to
$\nu_n^b\Longrightarrow\delta_b$, as established after
\eqref{eq:weak-concentration}.
Continuity of $h_b$ and its zero values at the other endpoints also
verify the sufficient condition in Corollary
\ref{prop:kernel-decay-concentration}.
The assertions about the
common refinement follow from Proposition
\ref{prop:common-refinement-projections}, including the identity
case $D^G=X$.
\end{proof}

\begin{cor}[Representation and approximation on a finite metric graph]
\label{cor:graph-representation-approximation}
Let $u>0$ be harmonic on $D$, choose a compact metrizable
compactification $D^G$, and set $\widehat D=D^G\vee X$, with
projection $p_K:\widehat D\to X$.  The canonical measures
\begin{equation*}
 \lambda_n:=u\,\omega_{x_0}^{D_n}
\end{equation*}
have weakly convergent subsequences on $\widehat D$.  Every limit
$\widehat\mu_u$ is carried by $p_K^{-1}(B)$ and gives,
for every $x\in D$,
\begin{equation}\label{eq:graph-refined-representation}
 u(x)=\int_{p_K^{-1}(B)}K_D(x,p_K(\widehat\eta))\,
                  d\widehat\mu_u(\widehat\eta),\qquad
 \widehat\mu_u(\widehat D)=u(x_0).
\end{equation}
At $\widehat\mu_u$-almost every refined boundary point, inner-boundary
points and their kernels admit simultaneous approximation along one
common subsequence, with the points chosen separately, as in
\eqref{eq:refined-synchronized-approximation}.

The push-forward $\mu_u=(p_K)_\#\widehat\mu_u$ is the unique
measure representing $u$ on $B$, and, for every $x\in D$,
the formula is explicitly
\begin{equation}\label{eq:graph-representation}
 u(x)=\int_BK_D(x,b)\,d\mu_u(b),\qquad
 \mu_u=\sum_{b\in B}u(b)H_b(x_0)\,\delta_b,
 \qquad \mu_u(B)=u(x_0).
\end{equation}
The projected measures satisfy
\begin{equation*}
 \lambda_n\Longrightarrow\mu_u\quad\text{weakly on }X,
\end{equation*}
and the exact inner-boundary formula is
\begin{equation*}
 u(x)=\int_{\partial D_n}K_{D_n}(x,y)\,d\lambda_n(y)
 \qquad(x\in D_n).
\end{equation*}
For every $b\in B$, the actual cut points in
\eqref{eq:graph-exhaustion} satisfy
\begin{equation}\label{eq:graph-cut-kernel-convergence}
 b_n\longrightarrow b,\qquad
 K_{D_n}(\cdot,b_n)\longrightarrow K_D(\cdot,b)
 \quad\text{locally uniformly in }D.
\end{equation}
The positive continuous interpolation of Corollary
\ref{cor:positive-continuous-interpolation} holds for this sequence.
The measure $\widehat\mu_u$ need not be unique, but its push-forward
is independent of the chosen weakly convergent subsequence.
\end{cor}

\begin{proof}
Proposition \ref{prop:graph-all-assumptions} verifies (A1)--(A2),
(C), and (M), with the boundary data (B).  In particular, $(C')$ holds
with $n_j=j$, so we keep the original exhaustion in Theorem
\ref{prop:direct-refined-representation}.  It first supplies the refined
limits, \eqref{eq:graph-refined-representation}, and the simultaneous
subsequence approximation.  Applying the gluing map $p_K$ and
Corollary \ref{thm:common-refinement-representation} then gives the
unique representing measure $\mu_u$ on $B$ and convergence of the
projected sequence.

To compute this measure, use the continuous endpoint values and
$H_b(c)=\delta_{bc}$.  Since $B$ is finite, evaluating its
representation at each endpoint $c\in B$ gives
\begin{equation*}
 u(c)=\sum_{b\in B}\frac{H_b(c)}{H_b(x_0)}\mu_u(\{b\})
 =\frac{\mu_u(\{c\})}{H_c(x_0)}.
\end{equation*}
Hence $\mu_u(\{c\})=u(c)H_c(x_0)$, which proves the explicit
formula in \eqref{eq:graph-representation}.  Evaluation at $x_0$
gives $\mu_u(B)=u(x_0)$.  These calculations identify the coefficients
of the measure already obtained by gluing.  The exact inner formula
follows from the mean-value identity:
\begin{equation*}
 \int_{\partial D_n}K_{D_n}(x,y)\,d\lambda_n(y)
 =\int_{\partial D_n}u(y)\,d\omega_x^{D_n}(y)=u(x).
\end{equation*}

For a fixed $b$, Theorem \ref{thm:synchronized-kernel} selects
$\xi_n\in\partial D_n$ converging to $b$ with locally uniform
kernel convergence.  Choose a neighborhood of $b$ inside its terminal
edge, disjoint from all the other terminal segments.  For all large
$n$, its intersection with $\partial D_n$ consists of the single
point $b_n$.  Since $\xi_n\to b$, eventually $\xi_n=b_n$.
Thus the actual cut points satisfy \eqref{eq:graph-cut-kernel-convergence}.
Corollary \ref{cor:positive-continuous-interpolation}, applied to this
sequence, supplies the positive continuous interpolating functions.
\end{proof}

\subsubsection{Boundary splitting and gluing}

A geometric boundary vertex can be approached along several incident
edges.  To obtain the preceding model for a connected open subgraph,
separate those incident ends into degree-one boundary vertices, retaining
all interior identifications.  The resulting completion $D^K$ carries
the independent endpoint values of harmonic functions.  Identifying
the copies of each geometric boundary vertex gives a compact quotient
$D^G$ and a map
\begin{equation*}
 q:D^K\longrightarrow D^G,\qquad q|_D=\operatorname{id}_D.
\end{equation*}
The common refinement $D^G\vee D^K$ is the graph of $q$, hence is
canonically homeomorphic to $D^K$.  Indeed, this graph is compact
and the points $(q(x),x)$, $x\in D$, are dense in it.
Proposition \ref{prop:harmonic-measure-gluing} pushes harmonic measures
to the geometric boundary.  A representing kernel need not descend
through $q$, because distinct ends in a fiber may carry distinct
minimal kernels.

\begin{ex}[A punctured circle]\label{ex:punctured-circle}
Let $D$ be a circle with one point $p$ removed, parametrized as
$(0,1)$, and put $x_0=1/2$.  The split completion is
$D^K=[0,1]$, with $B=\{0,1\}$.  Its harmonic measures and normalized
minimal kernels are
\begin{equation*}
 \omega_t^D=(1-t)\delta_0+t\delta_1,\qquad
 K_D(t,0)=2(1-t),\qquad K_D(t,1)=2t
 \quad(0<t<1).
\end{equation*}
For $u(t)=a(1-t)+bt$, where $a,b\geq0$ and $a+b>0$, the unique
representing measure is
\begin{equation*}
 \mu_u=\frac a2\,\delta_0+\frac b2\,\delta_1.
\end{equation*}
If $D_n=(\varepsilon_n,1-\varepsilon_n)$ with
$0<\varepsilon_n\downarrow0$ strictly and $\varepsilon_1<1/2$, then
\begin{equation*}
 \omega_{1/2}^{D_n}
 =\tfrac12\delta_{\varepsilon_n}
  +\tfrac12\delta_{1-\varepsilon_n}.
\end{equation*}
In particular,
\begin{equation*}
 \nu_n^0=(1-\varepsilon_n)\delta_{\varepsilon_n}
            +\varepsilon_n\delta_{1-\varepsilon_n}
 \Longrightarrow\delta_0,
\end{equation*}
and the analogous formula gives $\nu_n^1\Longrightarrow\delta_1$.

The quotient identifying $0$ and $1$ restores the circle, whose
geometric boundary is $\{p\}$.  Its pushed-forward harmonic measure
is $\delta_p$ for every interior point.  The relative kernel on that one-point
boundary is therefore $1$; it cannot parametrize the two extreme
points $2(1-t)$ and $2t$.  Thus the split completion provides the boundary data (B),
but fails if one uses the single geometric boundary point instead.
This distinguishes the gluing of harmonic measures from the descent
of a kernel representation.
\end{ex}

This finite-dimensional model verifies all the hypotheses of Section
\ref{Sec4} without a separate boundary convergence theorem.  In particular,
Theorem \ref{prop:direct-refined-representation} applies directly:
the canonical inner-boundary measures first give a representation on the
split completion or its refinement, and the projection to the boundary
parametrizing the minimal kernels gives the unique representing measure.
The further geometric gluing in Example \ref{ex:punctured-circle}
does not preserve these kernels and is not the minimal-boundary projection.  No claim
about arbitrary infinite weighted graphs is needed.

\section*{Acknowledgments}
This research work is partly supported by the National Natural Science Foundation of China (NSFC 12371096). We thank Professors Jun Geng and Sibei Yang for their helpful suggestions.





\renewcommand{\baselinestretch}{1}
\bibliographystyle{plain}

\end{document}